\documentclass[letter,final,oneside,notitlepage,10pt]{article}

\usepackage{amssymb}
\usepackage{amsmath}
\usepackage{amstext}
\usepackage{mathrsfs}
\usepackage[all]{xy}
\usepackage{xstring}
\usepackage{amsthm}
\usepackage{graphicx}
\usepackage[margin=2cm,font=normalsize,labelfont=bf]{caption}
\usepackage{verbatim}
\usepackage{enumerate}
\usepackage{fancyhdr}
\usepackage[normalem]{ulem}
\usepackage{hyperref}
\usepackage[latin1]{inputenc}
\usepackage[english]{babel}
\usepackage{rotating}
\usepackage[]{geometry}

\def\Z {\mathbb{Z}}

\def\id{\mathrm{id}}
\def\h {\check{\mathrm{H}}}

\def\trivlin{\mathbf{I}}

\def\quand{\quad\text{ and }\quad}
\def\quomma{\quad\text{, }\quad}
\def\quere{\quad\text{ where }\quad}

\def\ev{\mathrm{ev}}
\def\lw#1#2{{}^{#1\!}#2}

\def\lli#1{\,_{#1}\!}

\def\upi{\underline{\pi}}

\def\idmorph#1{#1_{dis}}
\renewcommand{\varepsilon}{\epsilon}
\def\bigset#1#2{\left\lbrace\;\begin{minipage}[c]{#1}\begin{center}#2\end{center}\end{minipage}\;\right\rbrace}

\newcommand\erf[1]{(\ref{#1})}

\newlength{\myl}

\newcommand{\ueins}{{\mathrm{U}}(1)}

\def\af{\mathcal{A}na^{\infty}}
\def\act#1#2{#1/\!\!/#2}
\def\hom#1#2{\mathcal{H}\!om(#1,#2)}
\def\triv{\mathcal{T}\!\!riv}

\def\vC {\check {\mathcal C}}

\def\brackets#1{\IfStrEq{#1}{-}{}{(#1)}}
\def\buntech#1#2{\mathcal{B}\hspace{-0.01em}un^{#2}_{\hspace{-0.04em}#1}}
\def\bun#1#2{\buntech{#1}{}\brackets{#2}}
\def\zwoabun#1#2{2\text{\text{-}}\buntech{#1}{}\brackets{#2}}

\def\grbtech#1{\mathcal{G}\hspace{-0.06em}r\hspace{-0.06em}b_{\hspace{-0.07em}#1}}
\def\grb#1#2{\grbtech#1\brackets{#2}}

\def\iso {\xymatrix{\ar[r]^{\sim}&}}

\newcommand{\alxydim}[2]{\begin{aligned}\xymatrix#1{#2}\end{aligned}}
\renewcommand{\to}{\!\xymatrix@R=0cm@C=1.4em{\ar[r] &}}
\renewcommand{\mapsto}{\!\xymatrix@R=0cm@C=1.4em{\ar@{|->}[r] &}\!}
\renewcommand{\Rightarrow}{\!\xymatrix@R=0cm@C=1.4em{\ar@{=>}[r] &}\!}
\renewcommand{\Leftarrow}{\!\xymatrix@R=0cm@C=1.4em{\ar@{<=}[r] &}\!}
\newcommand{\incl}{\!\xymatrix@R=0cm@C=1.4em{\ar@{^(->}[r] &}\!}

\renewcommand\Leftrightarrow{\!\xymatrix@R=0cm@C=1.4em{\ar@{<=>}[r] &}\!}

\makeatletter

\newcounter{authorcounter}
\newcounter{adresscounter}

\gdef\@ntitle{\@title}
\def\subtitle#1{\gdef\@subtitle{#1}}
\def\@subtitle{}

\def\authortagsused{0}
\def\adresstag#1{\if!#1!\else$^{\;#1\;}$\fi}

\renewcommand{\author}[2][]{
  \stepcounter{authorcounter}
  \if!#1!\else\gdef\authortagsused{1}\fi
  \ifnum\value{authorcounter}=1
    \def\@authorstringa{#2\adresstag{#1}}
    \def\@authorstringb{#2}
    \def\@authorstringc{#2\adresstag{#1}}
  \else
    \g@addto@macro\@authorstringa{\ and #2\adresstag{#1}}
    \g@addto@macro\@authorstringb{\ and #2}
    \g@addto@macro\@authorstringc{\\#2\adresstag{#1}}
  \fi}
\def\@author{\ifnum\value{denseversion}=0\@authorstringa\else\@authorstringb\fi}

\def\@adressstringa{}
\def\@adressstringb{}
\newcommand{\adress}[2][]{
  \stepcounter{adresscounter}
  \ifnum\value{adresscounter}=1
    \g@addto@macro\@adressstringa{\ifnum\authortagsused=0\def\br{\\}\else\def\br{, }\fi\adresstag{#1}#2}
    \g@addto@macro\@adressstringb{\def\br{\\}\adresstag{#1}\parbox[t]{14cm}{#2}}
  \else
    \g@addto@macro\@adressstringa{\\[\bigskipamount]\adresstag{#1}#2}
    \g@addto@macro\@adressstringb{\\[\medskipamount]\adresstag{#1}\parbox[t]{14cm}{#2}}
  \fi}
\def\@adress{\ifnum\value{denseversion}=0\@adressstringa\else\@adressstringb\fi}

\def\preprint#1{\gdef\@preprint{#1}}
\def\@preprint{}
\def\keywords#1{\gdef\@keywords{#1}}
\def\@keywords{}
\def\email#1{\gdef\@email{#1}}
\def\@email{}
\def\msc#1{\gdef\@msc{#1}}
\def\@msc{}
\def\dedication#1{\gdef\@dedication{#1}}
\def\@dedication{}
\def\mybaselinestretch#1{\gdef\@mybaselinestretch{#1}}
\def\@mybaselinestretch{}

\def\refname{References}

\def\showkeywords{\begin{flushleft}\footnotesize\textbf{Keywords}: \@keywords\end{flushleft}}
\def\showmsc{\begin{flushleft}\footnotesize\textbf{MSC 2010}: \@msc\end{flushleft}}

\newcounter{denseversion}
\newlength{\zeilenlaenge}
\def\putindent#1{
  \settowidth{\zeilenlaenge}{#1}
  \ifnum\zeilenlaenge>\textwidth
    #1
  \else
    \noindent #1
  \fi
}

\def\nobr{~\hspace{-0.26em}}
\def\maps{\nobr:\nobr}
\def\df{\nobr := \nobr}
\def\eq{\nobr = \nobr}

\let\Oldin\in\renewcommand{\in}{\nobr\Oldin\nobr}
\let\Oldtimes\times\renewcommand{\times}{\nobr\Oldtimes}
\let\Oldotimes\otimes\renewcommand{\otimes}{\nobr\Oldotimes}

\newlength{\myparskip}
\newlength{\myproofparskip}

\mybaselinestretch{1.2}
\renewcommand{\baselinestretch}{\@mybaselinestretch}

\def\denseversion{
  \setcounter{denseversion}{1}
  \usepackage[left=2.2cm,right=2.2cm,top=2.2cm]{geometry}
  \mybaselinestretch{1.1}
  \renewcommand{\baselinestretch}{\@mybaselinestretch}
  \normalfont
  \fancyfoot[C]{\itshape{\hspace{4.5cm}--$\,\,$\thepage$\,\,$--}}}

\renewcommand{\baselinestretch}{1.2}

\renewcommand{\emph}[1]{\def\reserved@a{it}\ifx\f@shape\reserved@a\uline{#1}\else\textit{#1}\fi}

\newcommand{\mytableofcontents}{
   \ifnum\value{denseversion}=0
     \tableofcontents
   \else
     \renewcommand{\baselinestretch}{0.8}
     \normalfont
     \tableofcontents
     \renewcommand{\baselinestretch}{\@mybaselinestretch}
     \normalfont
   \fi}

\newcounter{mythm}[subsection]
\newcounter{mainthm}

\def\setsecnumdepth#1{
  \setcounter{secnumdepth}{#1}
  \setcounter{mythm}{0}
  \ifnum \c@secnumdepth >0
    \ifnum \c@secnumdepth >1
      \def\themythm{\thesubsection.\arabic{mythm}}
      \numberwithin{equation}{subsection}
      \renewcommand\theequation{\thesubsection.\arabic{equation}}
    \else
      \def\themythm{\thesection.\arabic{mythm}}
      \numberwithin{equation}{section}
      \renewcommand\theequation{\thesection.\arabic{equation}}
    \fi
  \else
    \def\themythm{\arabic{mythm}}
  \fi}

\newenvironment{mythmenv}{\strut\ \setlength{\parskip}{\myproofparskip}}{\setlength{\parskip}{\myparskip}}

\newlength{\mythmskip}
\newlength{\mythmtopskip}
\newtheoremstyle{mythmstylea}{\mythmtopskip}{\mythmskip}{\it}{}{\bf}{.}{0em}{}
\newtheoremstyle{mythmstyleb}{\mythmtopskip}{\mythmskip}{}{}{\bf}{.}{0em}{}

\theoremstyle{mythmstylea}
\newtheorem{mytheorem}[mythm]{Theorem}
\newtheorem{mydefinition}[mythm]{Definition}
\newtheorem{mycorollary}[mythm]{Corollary}
\newtheorem{myproposition}[mythm]{Proposition}
\newtheorem{mylemma}[mythm]{Lemma}

\newtheorem{mymaintheorem}[mainthm]{Theorem}
\newtheorem{mymaincorollary}[mainthm]{Corollary}
\newtheorem{mymaindefinition}[mainthm]{Definition}

\theoremstyle{mythmstyleb}
\newtheorem{myremark}[mythm]{Remark}
\newtheorem{myexample}[mythm]{Example}

\newenvironment{theorem}[1][]{\begin{mytheorem}[#1]\begin{mythmenv}}{\end{mythmenv}\end{mytheorem}}
\newenvironment{definition}[1][]{\begin{mydefinition}[#1]\begin{mythmenv}}{\end{mythmenv}\end{mydefinition}}
\newenvironment{corollary}[1][]{\begin{mycorollary}[#1]\begin{mythmenv}}{\end{mythmenv}\end{mycorollary}}
\newenvironment{proposition}[1][]{\begin{myproposition}[#1]\begin{mythmenv}}{\end{mythmenv}\end{myproposition}}
\newenvironment{lemma}[1][]{\begin{mylemma}[#1]\begin{mythmenv}}{\end{mythmenv}\end{mylemma}}
\newenvironment{remark}[1][]{\begin{myremark}[#1]\begin{mythmenv}}{\end{mythmenv}\end{myremark}}
\newenvironment{example}[1][]{\begin{myexample}[#1]\begin{mythmenv}}{\end{mythmenv}\end{myexample}}

\renewenvironment{proof}[1][Proof]{\noindent #1. \begin{mythmenv}}{\ \strut\hfill{$\square$}\end{mythmenv}\medskip}

\def\maketitle{
  \setlength{\parskip}{\myparskip}
  \newpage
  \noindent
  \begin{center}
    \LARGE\@title\\
    \if!\@subtitle!\else \smallskip\LARGE\@subtitle\\\fi
    \bigskip
    \if!\@author!\else     \bigskip\large\@author\\\fi
    \ifnum\value{denseversion}=0
      \if!\@adress!\else     \bigskip\normalsize\@adress\\\fi
      \if!\@email!\else      \bigskip\normalsize\textit{\@email}\\\fi
    \else\fi
    \if!\@dedication!\else \bigskip\normalsize{\@dedication}\\\fi
  \end{center}
  \ifnum\value{denseversion}=0\vskip 1.5cm\else\vskip0.5cm\fi
  \thispagestyle{empty}}

\def\kobiburl#1{
   \IfBeginWith
     {#1}
     {http://arxiv.org/abs/}
     {\kobibarxiv{#1}}
     {\kobiblink{#1}}}
\def\kobibarxiv#1{\href{#1}{\texttt{[arxiv:\StrGobbleLeft{#1}{21}]}}}
\def\kobiblink#1{Available as: \href{#1}{\texttt{\StrSubstitute{#1}{_}{\underline{\;\;}}}}}

\def\kobib#1{
  \begin{raggedright}
  \ifnum\value{denseversion}=0\else\small\fi

  \end{raggedright}
  \ifnum\value{denseversion}=0\else
      \noindent
      \if!\@authorstringc!\else
        \ifnum\authortagsused=0\ifnum\value{authorcounter}>1\normalsize\@authorstringc\\[\medskipamount]\else\fi\else\normalsize\@authorstringc\\[\medskipamount]\fi       \fi
      \if!\@adress!\else\normalsize\@adress\\\fi
      \ifnum\authortagsused=0\ifnum\value{authorcounter}=1\if!\@email!\else\linebreak\normalsize\textit{\@email}\\\fi\else\fi\else\fi
  \fi}

\newenvironment{commentfigure}{}
\newenvironment{sidewayscommentfigure}{\begin{minipage}}{\end{minipage}}

\def\showcomments{ -- Comments suppressed}

\newif\if@fewtab\@fewtabtrue{
  \count255=\time\divide\count255 by 60
  \xdef\hourmin{\number\count255}
  \multiply\count255 by-60\advance\count255 by\time
  \xdef\hourmin{\hourmin:\ifnum\count255<10 0\fi\the\count255}}
\def\ps@draft{
  \let\@mkboth\@gobbletwo
  \def\@oddfoot{
    \hbox to 7 cm{\tiny \versionno\hfil}
    \hskip -7cm\hfil\rm\thepage\hfil{\tiny\draftdate}}
  \def\@oddhead{}
  \def\@evenhead{}
  \let\@evenfoot\@oddfoot}
\def\draftdate{\number\month/\number\day/\number\year\ \ \ \hourmin }
\newcommand\version[1]{
  \typeout{}\typeout{#1}\typeout{}
  \vskip-1.7cm \centerline{\fbox{{\normalsize\tt DRAFT -- #1 -- 
  \draftdate\showcomments}}} \vskip0.92cm}
\def\draft#1{
  \def\versionno{#1}
  \pagestyle{draft}\thispagestyle{draft}
  \gdef\@ntitle{\version\versionno \@title}
  \global\def\draftcontrol{1}}
\global\def\draftcontrol{0}

\makeatother

\usepackage[latin1]{inputenc}
\usepackage[english]{babel}

\def\quot#1{``#1''}

\title{Four Equivalent Versions of Non-Abelian Gerbes}

\author[a]{Thomas Nikolaus}
\author[b]{Konrad Waldorf}

\adress[a]{Fachbereich Mathematik\br Bereich Algebra und Zahlentheorie\\Bundesstraße 55\br 20146 Hamburg\br Germany}
\adress[b]{Fakultät für Mathematik\br Universität Regensburg\\Universitätsstraße 31\br 93053 Regensburg\br Germany}

\msc{Primary 55R65, Secondary 53C08,55N05,22A22}

\usepackage{amstext}

\begin{document}


\setsecnumdepth{3}

\maketitle

\begin{abstract}
We recall and partially improve four versions of smooth, non-abelian gerbes: \v Cech cocycles, classifying maps, bundle gerbes, and  principal 2-bundles. We prove that all these four versions are equivalent, and so establish new relations between  interesting recent developments. Prominent partial results that we prove are  a bijection between the  continuous and smooth non-abelian cohomology, and an explicit equivalence between bundle gerbes and principal 2-bundles as 2-stacks. 
\showkeywords
\showmsc
\end{abstract}


\mytableofcontents

\section{Introduction}

\label{intro}

Let $G$ be a Lie group and $M$ be a smooth manifold.
There are (among others) the following four   ways to say what a  {\em smooth $G$-bundle} over $M$ is:
\begin{enumerate}[(1)]

\item 
\emph{\v Cech 1-Cocycles}: an open cover $\left \lbrace U_i \right \rbrace$ of $M$, and for each non-empty intersection $U_i \cap U_j$ a smooth map $g_{ij}:U_i \cap U_j \to G$ satisfying the cocycle condition 
\begin{equation*}
g_{ij} \cdot g_{jk}=g_{ik}\text{.}
\end{equation*}

\item 
\emph{Classifying maps}: a continuous map 
$$ f: M \to \mathfrak B G$$
to the classifying space $\mathfrak B G$ of the group $G$.

\item
\emph{Bundle 0-gerbes}: a surjective submersion $\pi:Y \to M$ and a smooth map $g: Y \times_M Y \to G$ satisfying 
\begin{equation*}
\pi_{12}^{*}g \cdot \pi_{23}^{*}g = \pi_{13}^{*}g\text{,}
\end{equation*}
where $\pi_{ij}: Y \times_M Y \times_M Y \to Y \times_M Y$ denotes the projection to the $i$th and the $j$th factors.

\item
\emph{Principal bundles}: a surjective submersion $\pi: P \to M$ with a smooth action of $G$ on $P$ that preserves $\pi$, such that the map
\begin{equation*}
P \times G \to P \times_M P : (p,g) \mapsto (p,p.g)
\end{equation*}
is a diffeomorphism. 

\end{enumerate}
It is well-known that these four versions of \quot{smooth $G$-bundles} are all equivalent. Indeed, (1) forms the smooth $G$-valued \v Cech cohomology in degree one, whereas (2) is known to be equivalent to continuous $G$-valued \v Cech cohomology, which in turn coincides with the smooth $G$-valued \v Cech cohomology. Further, (3) and (4) form equivalent categories; and isomorphism classes of the objects  (3) are in bijection with equivalence classes of the cocycles (1).

In this article we provide an analogous picture for \emph{smooth $\Gamma$-gerbes}, where  $\Gamma$ is a strict Lie 2-group. In particular, $\Gamma$ can be the automorphism 2-group of an ordinary Lie group $G$, in which case the term \quot{non-abelian $G$-gerbe} is commonly used. We compare the following four versions:
\begin{description}
\item[\normalfont Version I:] 
\emph{Smooth, non-abelian \v Cech $\Gamma$-cocycles} (Definition \ref{def:cocycles}). These form the classical, smooth groupoid-valued cohomology $\h^1(M,\Gamma)$ in the sense of Giraud \cite{giraud} and Breen \cite[Ch. 4]{breen3}, \cite{breen2}. 

\item[\normalfont Version II:] 
\emph{Classifying maps} (Definition \ref{def:classmap}). These are continuous maps $f: M \to \mathfrak B |\Gamma|$ to the classifying space of the geometric realization of $\Gamma$; such maps have been introduced and studied by Baez and Stevenson \cite{baez8}.

\item[\normalfont Version III:]
\emph{$\Gamma$-bundle gerbes} (Definition \ref{def:grb}). These have been developed by Aschieri, Cantini and Jurco \cite{aschieri} as a generalization of the abelian bundle gerbes of Murray \cite{murray}. Here we present an equivalent  definition by applying a higher categorical version  \cite{nikolaus2} of  Grothendieck's stackification construction to the monoidal pre-2-stack of principal $\Gamma$-bundles. 

\item[\normalfont Version IV:]
\emph{Principal $\Gamma$-2-bundles} (Definition \ref{def:zwoabun}). These have been introduced by Bartels \cite{bartels}; their total spaces are Lie groupoids on which the Lie 2-group $\Gamma$ acts in a certain way. Compared to Bartels' definition, ours uses a  stricter and easier notion of such an action. 

\end{description}
We prove that all four versions are equivalent, and follow the same line of arguments as in the case of $G$-bundles outlined above:
\begin{itemize}

\item 
Baez and Stevenson have shown that homotopy classes of classifying maps of Version II are in bijection with the continuous groupoid-valued \v Cech cohomology $\h^1_c(M,\Gamma)$. We prove  (Proposition \ref{smoothening}) that the inclusion of \emph{smooth} into \emph{continuous} \v Cech $\Gamma$-cocycles induces a bijection $\h^1_c(M,\Gamma) \cong \h^1(M,\Gamma)$. These two results establish the equivalence between our Versions I and II (Theorem \ref{th:classifying}). 
\item
$\Gamma$-bundle gerbes and principal $\Gamma$-2-bundles over $M$ form bicategories. We prove (Theorem \ref{th:equivalence}) that these bicategories are equivalent, and so establish the equivalence between Versions III and IV. Our proof provides explicit 2-functors in both directions. 

\item
We prove the equivalence between Versions I and III by showing that non-abelian $\Gamma$-bundle gerbes are classified by the non-abelian cohomology group $\h^1(M,\Gamma)$ (Theorem \ref{th:gerbesclass}).
\end{itemize}

The  first aim of this paper is to  simplify and clarify the notion of a non-abelian gerbe. This concerns  the notion of a $\Gamma$-bundle gerbe (Version III), for which we give a new, conceptually clear, and manifestly 2-categorical definition. It also concerns the notion of a principal 2-bundle (Version IV), for which we provide a new definition that is carefully balanced between generality and simplicity.

The second aim of this paper is to make it possible to compare and transfer available  results between  the various versions. Indeed, none of the three equivalences above is available in the existing literature. As an example why such equivalences can be useful, we use Theorem \ref{th:equivalence} -- the equivalence between $\Gamma$-bundle gerbes and principal $\Gamma$-2-bundles -- in order to carry two facts  about $\Gamma$-bundle gerbes over to principal $\Gamma$-2-bundles. We prove: 
\begin{enumerate}
\item 
Principal $\Gamma$-2-bundles form a 2-stack over smooth manifolds (Theorem \ref{th:2stack}).
This is a new and evidently important result, since it explains precisely in which way one can \emph{glue} 2-bundles from local patches. 

\item
If $\Gamma$ and $\Omega$ are weakly equivalent Lie 2-groups, the 2-stacks of principal $\Gamma$-2-bundles and principal $\Omega$-2-bundles are equivalent (Theorem \ref{extequivbun}). This is another new result that generalizes the well-known fact that principal $G$-bundles and principal $H$-bundles form equivalent stacks, whenever $G$ and $H$ are isomorphic Lie groups.
 
\end{enumerate}
The two facts about $\Gamma$-bundle gerbes (Theorems \ref{th:stack} and \ref{extequiv}) on which these results are based are proved in an outmost abstract way: the first is a mere consequence of the definition of $\Gamma$-bundle gerbes that we give, namely via a 2-stackification procedure for principal $\Gamma$-bundles. The second follows  from the fact that principal $\Gamma$-bundles and principal $\Omega$-bundles form equivalent monoidal pre-2-stacks, which we deduce as a corollary of their description by anafunctors.

The present paper is part of a larger program. In a forthcoming paper, we address the discussion of non-abelian lifting problems, in particular string structures. In a second forthcoming paper we will present the picture of four equivalent versions in a setting \emph{with connections}, based on the results of the present paper. Our motivation is  to understand the role of 2-bundles with connections in higher gauge theories, where they serve as \quot{B-fields}.  
Here, two (non-abelian) 2-groups are especially important, namely the string group \cite{baez9} and the Jandl group \cite{nikolaus2}. More precisely, string-2-bundles appear in supersymmetric sigma models that describe fermionic string theories \cite{bunke1}; while Jandl-2-bundles appear in unoriented sigma models that describe e.g. bosonic type-I string theories \cite{schreiber1}.

This paper is organized as follows. In Section \ref{preliminaries} we recall and summarize the theory of principal groupoid bundles and their description by anafunctors. The rest of the paper is based on this theory. In Sections \ref{versionI}, \ref{versionII}, \ref{versionIII} and \ref{sec:2bundle} we introduce our four versions of smooth $\Gamma$-gerbes, and establish all but one equivalence. The remaining equivalence, the one between bundle gerbes and principal 2-bundles, is discussed in Section \ref{sec:equivalences}.

\paragraph{Acknowledgements.}  
We thank Christoph Wockel for helpful
discussions. We also thank the Erwin Schrödinger Institute in Vienna and the Instituto Superior Técnico in Lisbon for kind invitations.  
TN is supported by the Collaborative 
Research Centre 676 ``Particles, Strings and the Early Universe - the Structure of Matter 
and Space-Time'' and the cluster of excellence
``Connecting particles with the cosmos''.

\setsecnumdepth{2}
\section{Preliminaries}
\label{preliminaries}
There is no claim of originality in this section. Our sources are Lerman \cite{lerman1}, Metzler \cite{metzler}, Heinloth \cite{heinloth} and Moerdijk-Mr\v{c}un \cite{moerdijk}. A slightly different but equivalent approach is developed in \cite{murray5}.

\subsection{Lie Groupoids and Groupoid Actions on Manifolds}

We assume that the reader is familiar with the notions of Lie groupoids, smooth functors and smooth natural transformations. 
In this paper, the following examples of Lie groupoids  appear:

\begin{example}
\label{example_groupoids}
\label{exliecat}
\begin{enumerate}[(a)]
\item \label{exM}
Every smooth manifold $M$ defines a Lie groupoid  denoted by $\idmorph{M}$ whose objects and morphisms are $M$, and all of whose structure maps are identities.

\item \label{exBG}
Every Lie group $G$ defines a Lie groupoid denoted by $\mathcal{B}G$, with one object, with $G$ as its smooth manifold of morphisms, and with the composition  $g_2 \circ g_1 := g_2g_1$.

\item
Suppose $X$ is a smooth manifold and $\rho:H \times X \to X$ is a smooth left action of a Lie group $H$  on $X$. Then, a Lie groupoid $\act X H$ is defined with objects $X$ and morphisms $H \times X$, and with
\begin{equation*}
s(h,x) := x
\quad\text{, }\quad
t(h,x) := \rho(h,x) 
\quad\text{ and }\quad
\id_x := (1,x)\text{.}
\end{equation*}
The composition is
\begin{equation*}
(h_2,x_2) \circ (h_1,x_1) := (h_2h_1,x_1)\text{,}
\end{equation*}
where $x_2 = \rho(h_1,x_1)$. The Lie groupoid $\act X H$ is called the \emph{action groupoid} of the action of $H$ on $X$.

\item \label{extHG}
Let $t: H \to G$ be a homomorphism of Lie groups. Then, 
\begin{equation*}
\rho: H \times G \to G: (h,g) \mapsto (t(h)g)
\end{equation*}
defines a smooth left action of $H$ on $G$. Thus, we have a Lie groupoid $\act G H$. 

\item
To every Lie groupoid $\Gamma$ one can associate an \emph{opposite Lie groupoid} $\Gamma^{\mathrm{op}}$ which has the source and the target map exchanged.

\end{enumerate}
\end{example}

We say that a \emph{right action} of a Lie groupoid $\Gamma$ on a smooth manifold $M$
is a pair $(\alpha,\rho)$ consisting of smooth maps $\alpha: M \to \Gamma_0$ and $\rho: M \lli \alpha \times_{t} \Gamma_1 \to M$ such that
\begin{equation*}
\rho(\rho(x,g),h)= \rho(x,g \circ h)
\quad\text{, }\quad
\rho(x,\id_{\alpha(x)})=x
\quad\text{ and }\quad
\alpha(\rho(x,g)) = s(g)
\end{equation*}
for all possible $g,h\in \Gamma_1$, $p\in \Gamma_0$ and $x\in M$. The map $\alpha$ is  called  \emph{anchor}. Later on we will replace the letter $\rho$ for the action by the symbol $\circ$ that denotes the composition of the groupoid.  A \emph{left action} of $\Gamma$ on $M$ is a right action of the opposite Lie groupoid $\Gamma^{\mathrm{op}}$. A smooth map $f:M \to M'$ between  $\Gamma$-spaces with actions $(\alpha,\rho)$ and $(\alpha',\rho')$ is called \emph{$\Gamma$-equivariant} if 
\begin{equation*} 
\alpha' \circ f=\alpha
\quad\text{ and }\quad
f(\rho(x,g)) = \rho' (f(x),g)\text{.}
\end{equation*}

\begin{example}
\label{exliegroupoidact}
\begin{enumerate}[(a)]
\item
Let $\Gamma$ be a Lie groupoid. Then, $\Gamma$ acts on the right on its morphisms $\Gamma_1$ by  $\alpha :=s$ and $\rho := \circ\;$. It acts on the left on its morphisms by $\alpha:= t$ and $\rho:=\circ\;$.

\item
Let $G$ be a Lie group. Then, a right/left action of the Lie groupoid $\mathcal{B}G$ (see Example \ref{example_groupoids} \erf{exBG}) on $M$ is the same as an ordinary smooth right/left action of $G$ on  $M$. 

\item
Let $X$ be a smooth manifold. A right/left action of $\idmorph X$ (see Example \ref{example_groupoids} \erf{exM}) on $M$ is the same as a smooth map $\alpha: M \to X$. 

\end{enumerate}
\end{example}

\subsection{Principal Groupoid Bundles}

We give the definition of a principal  bundle in exactly the same way as we are going to define principal \emph{2-}bundles in Section \ref{sec:2bundle}.

\begin{definition}
\label{def:pgb}
Let $M$ be a smooth manifold, and let $\Gamma$ be a Lie groupoid.
\begin{enumerate}
\item 
A \emph{principal $\Gamma$-bundle over $M$} is a smooth manifold $P$ with a surjective submersion $\pi: P \to M$ and a right $\Gamma$-action $(\alpha,\rho)$ that respects the projection $\pi$, such that
\begin{equation*}
\tau: P \lli{\alpha} \times_{t} \Gamma_1 \to P \times_M P : (p,g) \mapsto (p,\rho(p,g))
\end{equation*}
is a diffeomorphism.

\item
Let $P_1$ and $P_2$ be  principal $\Gamma$-bundles over $M$. A \emph{morphism} $\varphi: P_1 \to P_2$
is a $\Gamma$-equivariant smooth map that respects the projections to $M$.
\end{enumerate}
\end{definition}

Principal $\Gamma$-bundles over $M$ form a category $\bun {\Gamma} M$.
In fact, this category is a groupoid, i.e. all morphisms between principal $\Gamma$-bundles are invertible. 
There is an evident notion of a pullback $f^{*}P$ of a principal $\Gamma$-bundle $P$ over $M$ along a smooth map $f: X \to M$, and similarly, morphisms between principal $\Gamma$-bundles pull back. These define a functor
\begin{equation*}
f^{*}: \bun \Gamma M \to \bun \Gamma X\text{.}
\end{equation*} 
These functors make principal $\Gamma$-bundles a prestack over smooth manifolds. One can  easily show that this prestack is a stack (for the Grothendieck topology of surjective submersions).

\begin{example}[Ordinary principal bundles]
\label{ex:grpdbun}
For $G$ a Lie group, we have an equality of categories
\begin{equation*}
\bun {\mathcal{B}G} M = \bun G M\text{,}
\end{equation*}
i.e. Definition \ref{def:pgb} reduces consistently to the definition of an ordinary principal $G$-bundle. 
\end{example}

\begin{example}
[Trivial principal groupoid bundles]  
\label{trivialbundle}
For $M$ a smooth manifold and $f: M \to \Gamma_0$ a smooth map, 
$P := M \lli f \times_t \Gamma_1$
and
$\pi(m,g) := m$ define a surjective submersion, and
$\alpha(m,g) := s(g)$ and $\rho((m,g),h) \df (m, g \circ h)$
define a right action of $\Gamma$ on $P$ that preserves the fibers. 
The map $\tau$
we have to look at has the inverse 
\begin{equation*}
\tau^{-1}: P \times_M P \to P \lli{\pi} \times_{t} \Gamma_1: ((m,g_1),(m,g_2)) \mapsto ((m,g_1),g_1^{-1} \circ g_2)\text{,}
\end{equation*}
which is smooth. Thus we have defined a principal $\Gamma$-bundle, which we denote by $\trivlin_f$ and which we  call the \emph{trivial bundle for the map $f$}. Any bundle that is isomorphic to a trivial bundle is called \emph{trivializable}. 
\end{example}

\begin{example}[Discrete structure groupoids]
\label{ex:idmorphbundle}
For $X$ a smooth manifold, we have an equivalence of categories
\begin{equation*}
\bun {\idmorph X}M \cong \idmorph{C^{\infty}(M,X)}\text{.}
\end{equation*}
Indeed, for a given principal $\idmorph{X}$-bundle $P$ one observes that the anchor $\alpha:P \to X$ descends along the bundle projection to a smooth map $f:M \to X$, and that isomorphic bundles determine the same map. Conversely, one associates to a smooth map $f: M \to X$ the trivial principal $\idmorph{X}$-bundle $\trivlin_f$ over $M$. 
\end{example}

\begin{example}[Exact sequences]
Let
\begin{equation}
\label{exact1}
\alxydim{}{1 \ar[r] & H \ar[r]^{t} & G \ar[r]^{p} & K \ar[r] & 1}
\end{equation} 
be an exact sequence of Lie groups, and let  $\Gamma := \act GH$ be the action groupoid associated to the Lie group homomorphism $t:H \to G$ as explained in Example \ref{exliecat} \erf{extHG}. In this situation, $p:G \to K$ is a surjective submersion, and
\begin{equation*}
\alpha:G \to \Gamma_0:g \mapsto g
\quand
\rho: G \lli{\alpha}\times_{t}\Gamma_1 \to G: (g,(h,g')) \mapsto g'
\end{equation*}
define a smooth right action of $\Gamma$ on $G$ that preserves $p$. The inverse of the  map $\tau$ is
\begin{equation*}
\tau^{-1}: G \times_K G \to G \lli{\alpha}\times_{t}\Gamma_1: (g_1,g_2) \mapsto (g_1, (t^{-1}(g_1g_2^{-1}),g_2))\text{,}
\end{equation*}
which is smooth because $t$ is an embedding. Thus, $G$ is a principal $\Gamma$-bundle over $K$. 
\end{example}

Next we  provide some elementary statements about trivial principal $\Gamma$-bundles.

\begin{lemma}
\label{lem:loctriv}
A principal $\Gamma$-bundle over $M$ is trivializable if and only if it has a smooth section.
\end{lemma}

\begin{proof}
A trivial bundle $\trivlin_f$ has the section
\begin{equation*}
s_f: M \to \trivlin_{f}: x \mapsto (x,\id_{f(x)})\text{;}
\end{equation*}
and so any trivializable bundle has a section. Conversely, suppose a principal $\Gamma$-bundle $P$ has a smooth section $s: M \to P$. Then, with $f := \alpha \circ s$,
\begin{equation*}
\varphi: \trivlin_f \to P: (m,g) \mapsto \rho(s(m),g)
\end{equation*}
is an isomorphism.
\end{proof}

The following consequence shows that principal $\Gamma$-bundles of Definition \ref{def:pgb}  are locally trivializable in the usual sense.

\begin{corollary}
Let $P$ be a principal $\Gamma$-bundle over $M$. Then, every point $x\in M$ has an open neighborhood $U$ over which $P$ has a  trivialization: a smooth map $f: U \to \Gamma_0$ and a morphism
$\varphi: \trivlin_f \to P|_U$.
\end{corollary}

\begin{proof}
One can choose $U$ such that the surjective submersion $\pi:U \to P$ has a smooth section. Then, Lemma \ref{lem:loctriv} applies to the restriction $P|_U$.
\end{proof}

We  determine the Hom-set $\hom {\trivlin_{f_1}} {\trivlin_{f_2}}$ between trivial principal $\Gamma$-bundles defined by smooth maps $f_1,f_2\maps M \to \Gamma_0$. To a bundle morphism $\varphi: \trivlin_{f_1} \to \trivlin_{f_2}$ one associates the smooth function $g: M \to \Gamma_1$ which is uniquely defined by the condition
\begin{equation*}
(\varphi \circ s_{f_1})(x) = s_{f_2}(x) \circ g(x)\text{.}
\end{equation*}
for all $x\in M$. It is straightforward to see that

\begin{lemma}
\label{homsets}
The above construction defines a bijection
\begin{equation*}
\hom {\trivlin_{f_1}} {\trivlin_{f_2}} \to \left \lbrace g \in C^{\infty}(M,\Gamma_1) \;|\; s \circ  g=f_1\;\text{ and }\;t \circ g=f_2 \right \rbrace\text{,}
\end{equation*}
under which identity morphisms correspond to constant maps and the composition of  bundle morphisms corresponds to the point-wise composition of functions. \end{lemma}

Finally, we consider the case of principal bundles for action groupoids. 

\begin{lemma}
\label{actionbundles}
Suppose $\act X H$ is a smooth action groupoid. The category $\bun {\act XH}M$ is equivalent to a category with 
\begin{itemize}
\item
Objects: principal $H$-bundles $P_H$ over $M$ together with a smooth, $H$-anti-equivariant map $f\colon P_H \to X$, i.e. $f(p\cdot h) = h^{-1}f(p)$. 

\item
Morphisms: bundle morphisms $\varphi_H:P_H \to P'_H$ that respect the maps $f$ and $f'$.
\end{itemize}
\end{lemma}
\begin{proof}
For a principal $\act XH$-bundle $(P, \alpha, \rho)$ we set $P_H := P$ with the given projection to $M$. The action of $H$ on $P_H$ is defined by
\begin{equation*}
p\star h := \rho(p,(h,h^{-1}\cdot\alpha(p)))\text{.}
\end{equation*}
This action is smooth, and it follows from the axioms of the principal bundle $P$ that it is principal.
The map $f: P_H \to X$ is  the anchor $\alpha$. 
The remaining steps are straightforward and  left as an exercise.
\end{proof}

\subsection{Anafunctors}\label{sec:ana}

An anafunctor is a generalization of a smooth functor between Lie groupoids, similar to a Morita equivalence, and also known as a Hilsum-Skandalis morphism.   The idea goes  back to Benabou \cite{benabou1}, also see \cite{johnstone1}. The references for the following definitions are \cite{lerman1,metzler}. 

\begin{definition}
Let $\mathcal{X}$ and $\mathcal{Y}$ be Lie groupoids. 
\begin{enumerate}
\item 
An \emph{anafunctor}\ $F: \mathcal{X} \to \mathcal{Y}$
is a smooth manifold $F$, a left action $(\alpha_l,\rho_l)$ of $\mathcal{X}$ on $F$, and  a right action $(\alpha_r,\rho_r)$ of $\mathcal{Y}$ on $F$ such that the actions commute and $\alpha_l: F \to \mathcal{X}_0$ is a principal $\mathcal{Y}$-bundle over $\mathcal{X}_0$.  

\item

A \emph{transformation} between anafunctors $f: F \Rightarrow F'$ is a smooth map $f: F \to F'$ which is $\mathcal{X}$-equivariant, $\mathcal{Y}$-equivariant, and satisfies $\alpha_l' \circ f= \alpha_l$ and $\alpha_r' \circ f = \alpha_r$. 
\end{enumerate}
\end{definition}

The smooth manifold $F$ of an anafunctor is called its \emph{total space}. Notice that the condition that the two actions on $F$ commute implies that each respects the anchor of the other. 
For fixed Lie groupoids $\mathcal{X}$ and $\mathcal{Y}$, anafunctors $F: \mathcal{X} \to \mathcal{Y}$ and transformations form a category $\af(\mathcal{X},\mathcal{Y})$. Since transformations are in particular morphisms between principal $\mathcal{Y}$-bundles, every transformation is invertible so that $\af(\mathcal{X},\mathcal{Y})$ is in fact a groupoid. 

\begin{example}[Anafunctors from ordinary functors]
\label{ex:anafunctor}
Given a smooth functor $\phi:\mathcal{X} \to \mathcal{Y}$, we obtain an anafunctor in the following way. We set $F := \mathcal{X}_0 \lli\phi\times _t \mathcal{Y}_1$ with anchors $\alpha_l:F \to \mathcal{X}_0$ and $\alpha_r: F \to \mathcal{Y}_0$ defined by $\alpha_l(x,g) := x$ and $\alpha_r(x,g) := s(g)$, and  actions \begin{equation*}
\rho_l: \mathcal{X}_1 \lli s \times_{\alpha_l} F \to F
\quad\text{ and }\quad
\rho_r: F \lli{\alpha_r}\times_t \mathcal{Y}_1 \to F
\end{equation*}
 defined by $\rho_l(f,(x,g)) := (t(f),\phi(f) \circ g)$ and $\rho_r((x,g),f) := (x, g \circ f)$. In the same way, a smooth natural transformation $\eta:\phi\Rightarrow \phi'$ defines a transformation $f_{\eta}:F \Rightarrow F'$ by $f_{\eta}(x,g) := (x,\eta(x) \circ g)$. Conversely, one can show that an anafunctor comes from a smooth functor, if its principal $\Gamma$-bundle has a smooth section. 
\end{example}

\begin{example}[Anafunctors with discrete source]
\label{anafunctors}
For $M$ a smooth manifold and $\Gamma$ a Lie groupoid, we have an equality of categories
\begin{equation*}
\bun\Gamma M = \af(\idmorph{M},\Gamma)\text{.}
\end{equation*}
Further, \emph{trivial} principal $\Gamma$-bundles correspond to smooth \emph{functors}.  
In particular, with Example \ref{ex:grpdbun} we have,
\begin{enumerate}[(a)]
\item
For $G$ a Lie group and $M$ a smooth manifold, an anafunctor $F: \idmorph{M} \to \mathcal{B}G$ is the same as an ordinary principal $G$-bundle over $M$.

\item
For $M$ and $X$ smooth manifolds, an anafunctor $F: \idmorph{M} \to \idmorph{X}$ is the same as a smooth map. 

\end{enumerate}
\end{example}

\begin{example}[Anafunctors with discrete target]
\label{anafunctors:c}
For $\Gamma$ a Lie groupoid and $M$ a smooth manifold, we have an equivalence of categories
\begin{equation*}
\idmorph{C^{\infty}(\Gamma_0,M)^{\Gamma}} \cong \af(\Gamma,\idmorph M) 
\end{equation*}
where $C^{\infty}(\Gamma_0,M)^{\Gamma}$ denotes the set of smooth maps $f: \Gamma_0 \to M$ such that $f \circ s = f \circ t$ as maps $\Gamma_1 \to M$. The equivalence is induced by regarding a map $f \in C^{\infty}(\Gamma_0,M)^{\Gamma}$ as a smooth functor $f:\Gamma \to \idmorph{M}$, which in turn induces an anafunctor. Conversely, an anafunctor $F: \Gamma \to \idmorph{M}$ is in particular an $\idmorph{M}$-bundle over $\Gamma_0$, which is nothing but a smooth function $f: \Gamma_0 \to M$ by Example \ref{ex:idmorphbundle}. The additional $\Gamma$-action assures the $\Gamma$-invariance of $f$. 
\end{example}

\begin{example}[Anafunctors between one-object Lie groupoids]
Let $G$ and $H$ be Lie groups, and let $\mathcal{B}G$ and $\mathcal{B}H$ be the associated one-object Lie groupoids (Example \ref{example_groupoids} \ref{exBG}). Then, there is an equivalence of categories
\begin{equation*}
\act{\mathrm{Hom}(G,H)}H \cong \af(\mathcal{B}G,\mathcal{B}H) \text{,}
\end{equation*}
where the action of $H$ on $\mathrm{Hom}(G,H)$ is by point-wise conjugation. The functor which establishes this equivalence sends a smooth group homomorphism $\alpha:G \to H$ to the evident smooth functor $F_{\alpha}: \mathcal{B}G \to \mathcal{B}H$ and converts this into an anafunctor (Example \ref{ex:anafunctor}). A morphism $h:\alpha_1 \to \alpha_2$ is sent to the natural transformation $\eta_h: F_{\alpha_1} \to F_{\alpha_2}$ whose component at the single object is the morphism $h \in H$. 
In order to see that this is essentially surjective, it suffices to notice that the principal $H$-bundle of any smooth anafunctor $F: \mathcal{B}G \to \mathcal{B}H$ has a section. The proof that the functor is full and faithful is straightforward. 
\end{example}

For the following definition, we suppose $\mathcal{X}$, $\mathcal{Y}$ and $\mathcal{Z}$ are Lie groupoids, and
$F:\mathcal{X} \to \mathcal{Y}$
and
$G: \mathcal{Y} \to \mathcal{Z}$
are anafunctors given by $F=(F,\alpha_l,\rho_l,\alpha_r,\rho_r)$ and $G=(G,\beta_l,\tau_l,\beta_r,\tau_r)$.

\begin{definition}
 The \emph{composition}
$G \circ F: \mathcal{X} \to \mathcal{Z}$
is the anafunctor  defined in the following way:
\begin{enumerate}
\item 
Its total space is
\begin{equation*}
E := (F \lli{\alpha_r}\times_{\beta_l} G) / \sim
\end{equation*}
where $(f,\tau_l(h,g)) \sim (\rho_r(f,h),g)$ for all $h\in \mathcal{Y}_1$ with $\alpha_r(f)=t(h)$ and $\beta_l(g)=s(h)$.

\item
The anchors are  $(f,g) \mapsto \alpha_l(f)$ and $(f,g) \mapsto \beta_r(g)$. 
\item
The actions $\mathcal{X}_1 \lli{s}\times_{\alpha} E \to E$ and $E \lli{\beta}\times_{t} \mathcal{Z}_1 \to E$ are given, respectively, by
\begin{equation*}
(\gamma,(f,g)) \mapsto (\rho_l(\gamma,f),g)
\quand
((f,g), \gamma) \mapsto (f,\tau_r(g,\gamma))\text{.}
\end{equation*}
\end{enumerate}
\end{definition}

\begin{remark}
Lie groupoids, anafunctors and transformations form a bicategory. This bicategory  is equivalent to the bicategory of differentiable stacks (also known as geometric stacks) \cite{pronk}.
\end{remark}

In this article, anafunctors serve  two purposes. The first is that one can use conveniently the composition of anafunctors to define \emph{extensions} of principal groupoid bundles:

\begin{definition}
\label{def:extension}
 If $P: \idmorph M \to \Gamma$ is a principal $\Gamma$-bundle over $M$, and $\Lambda: \Gamma \to \Omega$ is an anafunctor, then the principal $\Omega$-bundle
\begin{equation*}
\Lambda P := \Lambda \circ P :\idmorph M \to \Omega
\end{equation*}
is called the \emph{extension of $P$ along $\Lambda$}.
\end{definition}

Unwinding this definition, the principal $\Omega$-bundle $\Lambda P$ has the total space
\begin{equation}
\label{bundleextension}
\Lambda P = (P \lli{\alpha}\times_{\alpha_l} \Lambda) \;/\; \sim
\end{equation}
where $(p,\rho_l(\gamma,\lambda)) \sim (\rho(p,\gamma),\lambda)$ for all $p\in P$, $\lambda\in \Lambda$ and $\gamma\in \Gamma_1$ with $\alpha(p)=t(\gamma)$ and $\alpha_l(\lambda)=s(\gamma)$. Here $\alpha$ is the anchor and $\rho$ is the action of $P$, and $\Lambda = (\Lambda,\alpha_l,\alpha_r,\rho_l,\rho_r)$. The bundle projection is $(p,\lambda) \mapsto \pi(p)$, where $\pi$ is the bundle projection of $P$, the anchor  is $(p,\lambda) \mapsto \alpha_r(\lambda)$, and the action is $(p,\lambda)\circ \omega=(p,\rho_r(\lambda,\omega))$. 

Extensions of bundles are accompanied by extensions of bundle morphisms. If $\varphi: P_1 \to P_2$ is a morphism between $\Gamma$-bundles, a morphism $\Lambda\varphi: \Lambda P_1 \to \Lambda P_2$ is defined by $\Lambda\varphi(p_1,\lambda) := (\varphi(p_1),\lambda)$ in terms of \erf{bundleextension}. Summarizing, we have

\begin{lemma}
\label{extensionfunctor}
Let $M$ be a smooth manifold and $\Lambda:\Gamma \to \Omega$ be an anafunctor. Then, extension along $\Lambda$ is a functor
\begin{equation*}
\Lambda: \bun\Gamma M \to \bun\Omega M\text{.}
\end{equation*}
Moreover, it commutes with pullbacks and so extends to a morphism between stacks. 
\end{lemma}

Next we suppose that  $t:H \to G$ is a Lie group homomorphism, and $\act GH$ is the associated action groupoid of Example \ref{exliecat} \erf{extHG}. We look at the functor $\Theta: \act GH \to \mathcal{B}H$ which is defined by $\Theta(h,g) := h$. Combining Lemma
\ref{actionbundles} with the extension along $\Theta$, we obtain 

\begin{lemma}
\label{crossedbundles}
The category $\bun {\act GH}M$ of principal $\act G H$-bundles over a smooth manifold $M$  is equivalent to a category with
\begin{itemize}
\item 
Objects: principal $H$-bundles $P_H$ over $M$ together with a section of  $\Theta (P_H)$.

\item
Morphisms: morphisms $\varphi$ of $H$-bundles so that $\Theta(\varphi)$ preserves the sections.
\end{itemize}
\end{lemma}

The second motivation for introducing anafunctors is that they provide the inverses to certain smooth functors which are not necessarily equivalences of categories.

\begin{definition}\label{defweak}
A smooth functor or anafunctor $F:\mathcal{X} \to \mathcal{Y}$ is called a \emph{weak equivalence}, if there exists an anafunctor $G: \mathcal{Y} \to \mathcal{X}$ together with transformations $G \circ F \cong \id_{\mathcal{X}}$ and $F \circ G \cong \id_\mathcal{Y}$. 
\end{definition}

We have the following immediate consequence for the stack morphisms of Lemma \ref{extensionfunctor}.

\begin{corollary}
\label{co:extensionweakequivalence}
Let $\Lambda: \Gamma \to \Omega$ be a weak equivalence between Lie groupoids. Then, extension of principal  bundles along $\Lambda$ is an equivalence
$\Lambda: \bun\Gamma M \to \bun\Omega M$ of categories. Moreover, these define an equivalence between the stacks $\bun\Gamma-$ and $\bun \Omega-$.
\end{corollary}

Concerning the claimed generalization of invertibility, we have 
the following well-known theorem, see \cite[Lemma 3.34]{lerman1}, \cite[Proposition 60]{metzler}. 

\begin{theorem}
\label{weak}
A smooth functor $F: \mathcal{X} \to \mathcal{Y}$ is a weak equivalence if and only if the following two conditions are satisfied:
\begin{enumerate}[(a)]
\item 
it is smoothly essentially surjective: the map
\begin{equation*}
s \circ \mathrm{pr}_2: \mathcal{X}_0 \lli{F_0}\times_{t} \mathcal{Y}_1 \to \mathcal{Y}_0
\end{equation*}
is a surjective submersion.

\item
it is smoothly fully faithful: the diagram
\begin{equation*}
\alxydim{@C=1.5cm@R=1.3cm}{\mathcal{X}_1 \ar[r]^{F} \ar[d]_{s \times t} & \mathcal{Y}_1 \ar[d]^{s \times t} \\ \mathcal{X}_0 \times \mathcal{X}_0 \ar[r]_-{F \times F} & \mathcal{Y}_0 \times \mathcal{Y}_0}
\end{equation*}
is a pullback diagram.
\end{enumerate}
\end{theorem}

\begin{remark}
\label{rem:inverses}
One can show that any smooth functor $F:\mathcal{X} \to \mathcal{Y}$ that is a weak equivalence actually has a \emph{canonical} inverse  anafunctor. \end{remark}

\subsection{Lie 2-Groups and crossed Modules}

A (strict) \emph{Lie 2-group} is a  Lie groupoid $\Gamma$ whose objects and morphisms are Lie groups, and all of whose structure maps are Lie group homomorphisms. One can conveniently bundle the multiplications and the inversions into smooth functors
\begin{equation*}
m: \Gamma \times \Gamma \to \Gamma
\quad\text{ and }\quad
i: \Gamma \to \Gamma\text{.} 
\end{equation*}

\begin{example} 
For $A$ an abelian Lie group, the Lie groupoid $\mathcal{B}A$ from Example \ref{exliecat} (b) is a Lie 2-group. The condition that $A$ is abelian is necessary.
\end{example}

\begin{example}
\label{crossedmodule}
Let $t: H \to G$ be a homomorphism of Lie groups, and let $\act GH$ be the corresponding Lie groupoid from Example \ref{exliecat} \erf{extHG}.  This Lie groupoid becomes  a Lie 2-group if the following structure is given: a smooth left action of $G$ on $H$ by Lie group homomorphisms, denoted $(g,h) \mapsto \lw gh$, satisfying
\begin{equation*}
t(\lw gh) = gt(h)g^{-1}
\quad\text{ and }\quad
\lw{t(h)}x=hxh^{-1}
\end{equation*} 
for all $g\in G$ and $h,x\in H$. Indeed, the objects $G$ of $\act G H$  already form a Lie group, and the multiplication on the morphisms $H \times G$ of $\act G H$ is the semi-direct product
\begin{equation}
\label{multiplication}
(h_2,g_2) \cdot (h_1,g_1) = (h_2\;\lw{g_2}{h_1},g_2g_1)\text{.}
\end{equation}
The homomorphism $t:H \to G$ together with the action of $G$ on $H$ is called a \emph{smooth crossed module}. Summarizing, every smooth crossed module defines a Lie 2-group.
\end{example}

\begin{remark}
\label{rem:crossedmodule}
Every Lie 2-group $\Gamma$ can be obtained from a smooth crossed module. Indeed, one puts
$G := \Gamma_0$ and $H := \mathrm{ker}(s)$,
equipped with the Lie group structures defined by the multiplication functor $m$ of $\Gamma$. The homomorphism $t : H \to G$ is the target map $t: \Gamma_1 \to \Gamma_0$, and the action of $G$ on $H$ is given by the formula
$\lw g\gamma := \id_{g} \cdot \gamma \cdot \id_{g^{-1}}$  for $g \in \Gamma_0$ and $\gamma \in \mathrm{ker}(s)$. These two constructions are inverse to each other (up to canonical Lie group isomorphisms and strict Lie 2-group isomorphisms, respectively). 
\end{remark}

\begin{example}
\label{aut2group}
Consider a connected Lie group $H$, so that its automorphism group $\mathrm{Aut}(H)$ is again a Lie group \cite{onishchik1}. Then, we have a smooth crossed module $(\mathrm{Aut}(H),H,i,\ev)$, where $i: H \to \mathrm{Aut}(H)$ is the assignment of inner automorphisms to group elements, and $\ev: \mathrm{Aut}(H) \times H \to H$ is the evaluation action. The associated Lie 2-group is denoted $\mathrm{AUT}(H)$ and is called the \emph{automorphism 2-group} of $H$.
\end{example}

\begin{example}
\label{exactsequence}
Let
\begin{equation*}
\alxydim{}{1 \ar[r] & H \ar[r]^{t} & G \ar[r]^{p} & K \ar[r] & 1}
\end{equation*} 
be an exact sequence of Lie groups, i.e. an exact sequence in which $p$ is a  submersion and $t$ is an embedding. The homomorphisms $t:H \to G$ and $p:G \to K$ define action groupoids $\act G H$ and $\act K G$ as explained in Example \ref{exliecat}. The first one is even a Lie 2-group:  the action  of $G$ on $H$ is defined by $\lw gh := t^{-1}(gt(h)g^{-1})$. This is well-defined: since 
\begin{equation*}
p(gt(h)g^{-1}) = p(g)p(t(h))p(g^{-1})=p(g)p(g)^{-1}=1\text{,}
\end{equation*}
the element $gt(h)g^{-1}$ lies in the image of $t$, and has a unique preimage. The action is smooth because $t$ is an embedding.  
The axioms of a crossed module are obviously satisfied. 
\end{example}

If a Lie groupoid $\Gamma$ is a Lie 2-group in virtue of a multiplication functor $m: \Gamma \times \Gamma \to \Gamma$, then the category $\bun \Gamma M$ of principal $\Gamma$-bundles over a smooth manifold $M$ is monoidal: 

\begin{definition}
\label{def:tensorproduct}
Let $P: \idmorph M \to \Gamma$ and $Q:\idmorph M \to \Gamma$ be principal $\Gamma$-bundles. The \emph{tensor product} $P \otimes Q$ is the anafunctor
\begin{equation*}
\alxydim{@C=1.5cm}{\idmorph M \ar[r]^-{\mathrm{diag}} & \idmorph M \times \idmorph M \ar[r]^-{P \times Q} & \Gamma \times \Gamma \ar[r]^-{m} &\Gamma \text{.}}
\end{equation*}
\end{definition}

\begin{example}
\label{ex:tensorproduct}
\begin{enumerate}[(a)]
\item 
\label{ex:tensorproduct:trivial}
Since trivial principal $\Gamma$-bundles $\trivlin_f$ correspond to \emph{smooth} functors $f:\idmorph M \to \Gamma$ (Example \ref{anafunctors}), it is clear that  $\trivlin_{f} \otimes \trivlin_g = \trivlin_{fg}$. 

\item
\label{ex:tensorproduct:general}
Unwinding Definition \ref{def:tensorproduct} in the  general case, the tensor product of two principal $\Gamma$-bundles $P_1$ and $P_2$ with anchors $\alpha_1$ and $\alpha_2$, respectively, and actions $\rho_1$ and $\rho_2$, respectively, is given by
\begin{equation}
\label{eq:tensorexplicit}
P_1 \otimes P_2 = ((P_1 \times_M P_2) \lli{m \circ (\alpha_1 \times \alpha_2)} \times_{t} \Gamma_1) \;/\; \sim\text{,}
\end{equation}
where 
\begin{equation}
\label{deftensorproductrel}
(p_1,p_2,m(\gamma_1,\gamma_2) \circ \gamma) \sim (\rho_1(p_1, \gamma_1), \rho_2(p_2 , \gamma_2), \gamma)
\end{equation}
for all $p_1\in P_1$, $p_2\in P_2$ and morphisms $\gamma,\gamma_1,\gamma_2\in \Gamma_1$ satisfying $t(\gamma_i)=\alpha_i(p_i)$ for $i=1,2$ and $s(\gamma_1)s(\gamma_2)\eq t(\gamma)$. The bundle projection is $\tilde\pi(p_1,p_2,\gamma) := \pi_1(p_1)=\pi_2(p_2)$, the anchor is $\tilde\alpha(p_1,p_2,\gamma) :=  s(\gamma)$, and the $\Gamma$-action is given by $\tilde\rho((p_1,p_2,\gamma),\gamma') := (p_1,p_2,\gamma \circ \gamma')$.
\end{enumerate}
\end{example}

As a consequence of Lemma \ref{extensionfunctor} and the fact that the composition of anafunctors is associative up to coherent transformations, we have

\begin{proposition}
\label{prop:monoidalstack}
For $M$ a smooth manifold and $\Gamma$ a Lie 2-group, the tensor product
\begin{equation*}
\otimes : \bun\Gamma M \times \bun\Gamma M \to \bun\Gamma M
\end{equation*}
equips the groupoid of principal $\Gamma$-bundles over $M$ with a monoidal structure.  Moreover, it turns the stack $\bun\Gamma-$ into a monoidal stack. \end{proposition}

Notice that the tensor unit of the monoidal groupoid $\bun\Gamma M$ is the trivial principal $\Gamma$-bundle $\trivlin_{1}$ associated to the constant map $1: M \to \Gamma_0$, or, in terms of anafunctors, the one associated to the constant functor $1: M \to \Gamma$.

A (weak) \emph{Lie 2-group homomorphism} between Lie 2-groups $(\Gamma,m_{\Gamma})$ and $(\Omega,m_{\Omega})$ is an anafunctor $\Lambda: \Gamma \to \Omega$ together with a transformation
\begin{equation}
\label{eq:lie2grouphom}
\alxydim{@=1.3cm}{\Gamma \times \Gamma \ar[r]^-{m_{\Gamma}}   \ar[d]_{\Lambda \times \Lambda} &  \Gamma \ar[d]^{\Lambda} \ar@{=>}[dl]|*+{\eta} \\ \Omega \times \Omega \ar[r]_-{m_{\Omega}} & \Omega}
\end{equation}
satisfying the evident coherence condition. Under the equivalence with smooth crossed modules (Remark \ref{rem:crossedmodule}),  Lie 2-group homomorphisms correspond to so-called butterflies \cite{Aldrovandi2009}. 
A Lie 2-group homomorphism is called \emph{weak equivalence}, if the anafunctor $\Lambda$ is a weak equivalence.
Since extensions and tensor products are both defined via  composition of anafunctors,  we immediately obtain

\begin{proposition}
\label{extensionstackmorphism}
Extension along a Lie 2-group homomorphism $\Lambda: \Gamma \to \Omega$ between Lie 2-groups is a monoidal functor
\begin{equation*}
\Lambda: \bun{\Gamma}M \to \bun\Omega M
\end{equation*}
between monoidal categories. Moreover, these form a monoidal morphism between monoidal stacks.
\end{proposition}

Since a monoidal functor is an equivalence of monoidal categories if it is an equivalence of the underlying categories, Corollary \ref{co:extensionweakequivalence} implies:

\begin{corollary}
\label{co:extensionweakequivalence2groups}
For $\Lambda: \Gamma \to \Omega$ a weak equivalence between Lie 2-groups, the monoidal functor of Proposition \ref{extensionstackmorphism} is an equivalence of monoidal categories. Moreover, these form a monoidal equivalence between monoidal stacks. 
\end{corollary}

If we represent the Lie 2-group $\Gamma$ by a smooth crossed module $t:H \to G$ as described in Example \ref{crossedmodule}, we want to determine explicitly what the tensor product looks like under the correspondence of $\act GH$-bundles and principal $H$-bundles with anti-equivariant maps to $G$, see Lemma \ref{actionbundles}.

\begin{lemma}
\label{lem:actiontensorbundles}
Let $t: H \to G$ be a crossed module and let $P$ and $Q$ be $\act GH$-bundles over $M$. Let $(P_H,f)$ and $(Q_H,g)$ be the principal $H$-bundles together with their $H$-anti-equivariant maps that belong to $P$ and $Q$, respectively, under the equivalence of Lemma \ref{actionbundles}. Then, the principal $H$-bundle that corresponds to the tensor product $P \otimes Q$ is given by 
\begin{equation*}
(P \otimes Q)_H = \big(P \times_M Q\big) / \sim
\quere
(p \star h,q) \sim (p,q \star (\lw{f(p)^{-1}}{h}))\text{.}
\end{equation*}
The action of $H$ on $(P \otimes Q)_H$ is  $ [(p,q)] \star h = [(p \star h,q)] $, and the $H$-anti-equivariant map of $(P \otimes Q)_H$ is   
$ [(p,q) ] \mapsto f(p) \cdot g(q)$.
\end{lemma}

Similar to the tensor product of principal $\Gamma$-bundles, the dual $P^{\vee}$ of a principal $\Gamma$-bundle $P$ over $M$ is the extension of $P$ along the inversion $i:\Gamma \to \Gamma$ of the 2-group, $P^{\vee} := i(P)$. The equality $m \circ (\id,i) = 1$ of functors $M \to \Gamma$ induces a \quot{death map} $d: P \otimes P^{\vee} \to \trivlin_{1}$.
We are going to use this bundle morphism in Section \ref{sec:gerbeprop}, but omit a further systematical treatment of duals for the sake of brevity.

\setsecnumdepth{1}

\section{Version I: Groupoid-valued Cohomology}

\label{versionI}

We have already mentioned group valued \v Cech 1-cocycles in the introduction. They consist of an open cover $\mathscr{U} = \{U_i\}_{i\in I}$ of $M$ and smooth functions 
$g_{ij}: U_i \cap U_j \to G$ satisfying the cocycle condition
$ g_{ij} \cdot g_{jk} = g_{ik}$. 
Segal realized \cite{segal3} that this is the same as a smooth functor 
$$g: \vC (\mathscr{U}) \to \mathcal{B}G $$
where $\mathcal{B} G $ denotes the one-object groupoid introduced in Example \ref{example_groupoids} \eqref{exBG} and 
$\vC (\mathscr{U})$ denotes the \emph{\v Cech groupoid} corresponding to the cover $\mathscr{U}$. It has objects $\bigsqcup_{i\in I} U_i$ and morphisms $\bigsqcup_{i,j\in I} U_i \cap U_j$,  and its structure maps are
\begin{equation*}
s(x,i,j) = (x,i)  
\quomma
t(x,i,j) =  (x,j)
\quomma
\id_{(x,i)} = (x,i,i)
\quand
(x,j,k) \circ (x,i,j) = (x,i,k)\text{.}
\end{equation*}

Analogously, smooth natural transformations between smooth functors $\vC (\mathscr{U}) \to \mathcal{B}G$ give rise to \v Cech coboundaries. Thus the set $\big[\vC(\mathscr{U}), \mathcal{B}G\big]$ of equivalence classes of smooth functors equals the usual first \v Cech cohomology with respect to the cover $\mathscr{U}$. The classical first \v Cech-cohomology $\h^1(M,G)$  of $M$ is hence given by the colimit over all open covers $\mathscr{U}$ of $M$
$$  \h^1(M,G) = \varinjlim_{\mathscr{U}} \big[\vC(\mathscr{U}), \mathcal{B}G\big]\text{.} $$

We use this coincidence  in order  to define the 0-th \v Cech cohomology with coefficients in a general Lie groupoid $\Gamma$:

\begin{definition}
If $\Gamma$ is a Lie groupoid we set
$$ \h^0(M,\Gamma) := \varinjlim_{\mathscr{U}} \big[\vC(\mathscr{U}), \Gamma\big]$$
where the colimit is taken over all covers $\mathscr{U}$ of $M$ and $\big[\vC(\mathscr{U}), \Gamma\big]$ denotes the set of equivalence classes of smooth functors $\vC(\mathscr{U}) \to \Gamma$.
\end{definition}

\begin{remark}
The choice of the degree is such that $ \h^0(M,\Gamma)$ agrees in the case $\Gamma = \idmorph G$  (Example \ref{example_groupoids} \eqref{exM}) with the classical 0-th \v Cech-cohomology $ \h^0 (M, G)$ of $M$ with values in $G$.
\end{remark}

The geometrical meaning of the set is given in the following well-known theorem, which can be proved e.g. using Lemma \ref{homsets}.
\begin{theorem}
\label{bijection}
There is a  bijection
\begin{equation*}
\h^0(M,\Gamma) \cong \bigset{4.2cm}{Isomorphism classes of principal $\Gamma$-bundles over $M$}\text{.}
\end{equation*}
\end{theorem}

If $\Gamma$ is not only a Lie 2-groupoid but a Lie 2-group one can also define a \emph{first} cohomology group $ \h^1(M,\Gamma)$. Indeed, in this case one can consider the Lie 2-groupoid $\mathcal {B} \Gamma$ with one object, morphisms  $\Gamma_0$ and 2-morphisms  $\Gamma_1$.
Multiplication in $\Gamma$ gives the composition of morphisms in $\mathcal B \Gamma$. Let 
$\big[ \vC(\mathscr{U}) , \mathcal B \Gamma]$
denote the set of equivalence classes of smooth, weak 2-functors from the \v Cech-groupoid $\vC(\mathscr{U})$ to the Lie 2-groupoid
$\mathcal B \Gamma$. For the definition of weak functors see \cite{benabou2} -- below we will determine this set   explicitly.

\begin{definition}
\label{def:cohomology}
For a 2-group $\Gamma$ we set
$$ \h^1(M,\Gamma) := \varinjlim_{\mathscr{U}} \big[\vC( \mathscr{U} ), \mathcal B\Gamma\big]\text{.}$$
\end{definition} 

\begin{remark}
This agrees for  $\Gamma = \idmorph G$ with the classical $ \h^1(M,G)$. Furthermore, for an abelian Lie group $A$ the Lie groupoid $\mathcal BA$ is even a 2-group and $ \h^1(M,\mathcal B A)$ agrees with the classical \v Cech-cohomology $ \h^2(M,A)$.
\end{remark}

Unwinding the above definition, we get Version I of smooth $\Gamma$-gerbes:

\begin{definition}
\label{def:cocycles}
Let $\Gamma$ be a Lie 2-group, and let $\mathscr{U}=\left \lbrace U_{\alpha} \right \rbrace_{\alpha \in A}$ be an open cover of $M$.
\begin{enumerate}
\item 
A \emph{$\Gamma$-1-cocycle} with respect to $\mathscr{U}$ is a pair $(f_{\alpha\beta}, g_{\alpha\beta\gamma})$ of smooth maps
\begin{equation*}
 f_{\alpha\beta} : U_\alpha \cap U_\beta \to \Gamma_0 
\quad \text{ and } \quad
g_{\alpha\beta\gamma} : U_\alpha \cap U_\beta \cap U_\beta \to \Gamma_1
\end{equation*}
satisfying $s(g_{\alpha\beta\gamma}) = f_{\beta\gamma} \cdot f_{\alpha\beta}$ and $t(g_{\alpha\beta\gamma}) = f_{\alpha\gamma}$, and
\begin{equation}
\label{cocyclecond}
g_{\alpha\beta\delta} \circ (g_{\beta\gamma\delta} \cdot \id_{f_{\alpha\beta}}) = 
g_{\alpha \gamma \delta} \circ (\id_{f_{\gamma\delta}} \cdot g_{\alpha\beta\gamma}) \text{.}
\end{equation}
Here, the symbols $\circ$ and $\cdot$ stand for the composition and multiplication of $\Gamma$, respectively.

\item
Two $\Gamma$-1-cocycles $(f_{\alpha\beta},g_{\alpha\beta\gamma})$ and $(f_{\alpha\beta}', g_{\alpha\beta\gamma}')$ are \emph{equivalent}, 
if there exist smooth maps $h_{\alpha}: U_{\alpha} \to \Gamma_0$ and 
$s_{\alpha\beta}: U_{\alpha} \cap U_{\beta} \to \Gamma_1$
with
\begin{multline*}
s(s_{\alpha\beta}) = g_{\alpha\beta}' \cdot h_{\alpha}
\quad\text{, }\quad
t(s_{\alpha\beta}) = h_{\beta} \cdot g_{\alpha\beta}
\\\quad\text{ and }\quad
(\id_{h_{\gamma}} \cdot g_{\alpha\beta\gamma})  \circ (s_{\beta\gamma} \cdot \id_{f_{\alpha\beta}}) \circ (\id_{f_{\beta\gamma}} \cdot s_{\alpha\beta}) = s_{\alpha\gamma} \circ (g'_{\alpha\beta\gamma} \cdot \id_{h_{\alpha}})\text{.}
\end{multline*}
\end{enumerate}
\end{definition}

\begin{remark}
For a crossed module $t:H \to G$ and $\Gamma := \act GH$ the associated Lie 2-group (Example \ref{crossedmodule}) one can reduce $\Gamma$-1-cocycles  to pairs 
\begin{equation*}
 \tilde f_{\alpha\beta} : U_\alpha \cap U_\beta \to G 
\quad \text{ and } \quad
\tilde g_{\alpha\beta\gamma} : U_\alpha \cap U_\beta \cap U_\beta \to H
\end{equation*}
which satisfies then a cocycle condition similar to \erf{cocyclecond}. Analogously, coboundaries can be reduced to pairs 
\begin{equation*}
\tilde h_{\alpha}: U_{\alpha} \to G \quad \text{ and } \quad
\tilde s_{\alpha\beta}: U_{\alpha} \cap U_{\beta} \to H \text{.}
\end{equation*}
This yields the  common definition of non-abelian cocycles, which can for example be found in \cite{breen3} or \cite{baez8}.   
\end{remark}

\begin{example}
In  case of the crossed module $i:H \to \mathrm{Aut}(H)$ with $\Gamma = \mathrm{AUT}(H)$ (see Example \ref{aut2group}) $\Gamma$-1-cocycles consist of pairs $\tilde f_{\alpha\beta}: U_{\alpha} \cap U_{\beta} \to \mathrm{Aut}(H)$ and $\tilde g_{\alpha\beta}: U_{\alpha} \cap U_{\beta} \cap U_{\gamma} \to H$.  Cocycles of this kind classify so-called Lie groupoid $H$-extensions \cite[Proposition 3.14]{laurent1}, which can hence be seen as another equivalent version for $\mathrm{AUT}(H)$-gerbes.
\end{example}

\setsecnumdepth{1}

\section{Version II: Classifying Maps} 

\label{versionII}

It is well known that for a Lie group $G$ the \emph{smooth} \v Cech-cohomology $ \h^1(M,G)$ and the \emph{continuous} \v Cech-cohomology $ \h_c^1(M,G)$  agree if $M$ is a smooth manifold (in particular paracompact).
This can e.g. be shown by locally approximating continuous cocycles by smooth ones without changing the cohomology class -- see \cite{wockel2}  (even for $G$ infinite-dimensional). Below we  generalize this fact to non-abelian cohomology for certain Lie 2-groups $\Gamma$.
Here the continuous \v Cech-cohomology $ \h^1_c(M,\Gamma)$ is defined in the same way as the smooth one (Definition \ref{def:cohomology}) but with all maps  continuous instead of smooth. A Lie groupoid $\Gamma$ is called \emph{smoothly separable}, if the set  $\upi_0\Gamma$ of isomorphism classes of objects is a smooth manifold for which the projection $\Gamma_0 \to \upi_0\Gamma$ is a submersion.  

\begin{proposition}
\label{smoothening}
Let $M$ be a  smooth manifold and let $\Gamma$ be a smoothly separable Lie 2-group. Then, the  inclusion
$$  \h^1(M,\Gamma) \to  \h_c^1(M,\Gamma) $$
of smooth into continuous \v Cech cohomology is a bijection.
\end{proposition}

\begin{remark}
It is possible that the assumption of being smoothly separable is not necessary, but a proof not assuming this would certainly be more involved than ours. Anyway, \emph{all} Lie 2-groups we are interested in are smoothly separable. 
\end{remark}

\begin{proof}[Proof of Proposition \ref{smoothening}] We denote by $\upi_1\Gamma $ the Lie subgroup of $\Gamma_{1}$ consisting of automorphisms of $1 \in \Gamma_0$. Since it has two commuting group structures -- composition and multiplication -- it is abelian.
The idea of the proof is to reduce the statement via  long exact sequences to statements proved in \cite{wockel2}. The exact sequence we need can be found in \cite{breen3}:
$$\h^0(M,\idmorph {(\upi_0\Gamma)}) \to \h^1(M,\mathcal B \upi_1 \Gamma) \to \h^1(M,\Gamma) 
  \to \h^1(M,\idmorph {(\upi_0\Gamma)}) \to \h^2(M,\mathcal B\upi_1\Gamma). $$
Note that $\h^1(M,\Gamma)$ and $\h^1(M,\idmorph {(\upi_0\Gamma)})$ do not have group structures: hence, exactness is only meant as exactness of pointed sets. But we actually have more structure, namely an action of $\h^1(M,B\upi_1\Gamma)$ on $\h^1(M,\Gamma)$. This action factors to an action of
$$C := \mathrm{coker}\Big(\h^0(M,\idmorph {(\upi_0\Gamma)}) \to \h^1(M,\mathcal B\upi_1 \Gamma)\Big).$$
In fact on the non-empty fibres of the morphism $\h^1(M,\Gamma) \to \h^1(M,\idmorph {(\upi_0\Gamma)})$ this action is simply transitive. In other words: $\h^1(M,\Gamma)$ is a $C$-Torsor over 
$$ K := \mathrm{ker}\Big(\h^1(M,\idmorph {(\upi_0\Gamma)}) \to \h^2(M,\mathcal B\upi_1 \Gamma)\Big).$$
The same type of sequence also exists in continuous cohomology, namely
$$ \h^0_c(M,\idmorph {(\upi_0\Gamma)}) \to \h^1_c(M,\mathcal B\upi_1 \Gamma) \to \h^1_c(M,\Gamma) 
  \to \h^1_c(M,\idmorph {(\upi_0\Gamma)}) \to \h^2_c(M,\mathcal B\upi_1 \Gamma).$$ 
With 
\begin{eqnarray*}
C' &:=& \mathrm{coker}\Big(\h^0_c(M,\idmorph {(\upi_0\Gamma)}) \to \h^1_c(M,\mathcal B\upi_1 \Gamma)\Big)
\\
K' &:=& \mathrm{ker}\Big(\h^1_c(M,\idmorph {(\upi_0\Gamma)}) \to \h^2_c(M,\mathcal B\upi_1 \Gamma)\Big)\text{,}
\end{eqnarray*}
we exhibit $\h^1_c(M,\Gamma)$ as a $C'$-Torsor over $K'$. 
  
The natural inclusions of smooth into continuous cohomology form a chain map between the two sequences. From 
\cite{wockel2} we know that they are isomorphisms on the second, fourth and fifth factor. In particular we have an induced isomorphism $K \iso K'$. Lemma \ref{lem:images} below additionally shows that  the induced morphism $C \to C'$ is an isomorphism. Using these isomorphisms we see that $\h^1(M,\Gamma)$ and $\h^1_c(M,\Gamma)$ are both $C$-torsors over $K$ and that the natural map $$\h^1(M,\Gamma) \to \h^1_c(M,\Gamma)$$ is a morphism of torsors. But each morphism of group torsors is bijective, which concludes the proof.
\end{proof}

\begin{lemma}
\label{lem:images}
The images of 
$$f: \h^0(M,\idmorph {(\upi_0\Gamma)}) \to \h^1(M,\mathcal B\upi_1 \Gamma) \quad \text{and} \quad f':\h^0_c(M,\idmorph {(\upi_0\Gamma)}) \to \h^1_c(M,\mathcal B\upi_1 \Gamma)$$
are isomorphic.
\end{lemma}
\begin{proof}
 $\h^0(M,\idmorph {(\upi_0\Gamma)})$ is the group of smooth maps $s: M \to \upi_0\Gamma$ and $\h^0_c(M,\idmorph {(\upi_0\Gamma)})$ is the group of continuous maps $t: M \to \upi_0\Gamma$. The groups $\h^1(M,\mathcal B\upi_1 \Gamma) = \h^2(M,\upi_1 \Gamma)$ and $\h^1_c(M,\mathcal B\upi_1 \Gamma) = \h^2_{c}(M,\upi_1 \Gamma)$ are isomorphic by the result of \cite{wockel2}. Under the connecting homomorphism
\begin{equation*}
\h^0(\upi_0\Gamma,\idmorph {(\upi_0\Gamma)}) \to \h^1(\upi_0\Gamma,\mathcal B\upi_1 \Gamma)
\end{equation*}
the identity $\id_{\upi_0\Gamma}$ is sent to a class $\xi_{\Gamma}$ with the property that  $f(s) = s^{*}\xi_{\Gamma}$ and $f'(t) = t^{*}\xi_{\Gamma}$. Hence it suffices to show that for each continuous map $t: M \to \upi_0\Gamma$ there is a smooth map $s:M \to  \upi_0\Gamma$ with $s^*\xi_\Gamma = t^*\xi_\Gamma$.
It is well known that for each continuous map $t$ between smooth manifolds a homotopic smooth map $s$ exists. It remains to show that the pullback $\h^1(\upi_0\Gamma,\mathcal{B}\upi_1 \Gamma) \to \h^1(M,\mathcal{B}\upi_1 \Gamma)$ along smooth maps is homotopy invariant. This can e.g. be seen by choosing smooth (abelian) $\mathcal{B}\upi_1 \Gamma$-bundle gerbes as representatives, in which case the homotopy invariance can be deduced from the existence of connections.  
\end{proof}

It is a standard result in topology that the continuous $G$-valued \v Cech cohomology of paracompact spaces is in bijection with homotopy classes of maps to the classifying space $\mathfrak BG$ of the group $G$. A model for the classifying space $\mathfrak BG$ is for example the geometric realization of the nerve of the groupoid $\mathcal B G$, or Milnor's join construction \cite{milnor}.

Now let $\Gamma$ be a Lie 2-group, and let $|\Gamma|$ denote the geometric realization of the nerve of $\Gamma$. Since the nerve is a simplicial topological group, $|\Gamma|$ is a topological group. Version II for smooth $\Gamma$-gerbes is:

\begin{definition}[{\cite{baez8}}]
\label{def:classmap}
A classifying map for a smooth $\Gamma$-gerbe is a continuous map
\begin{equation*}
f: M \to \mathfrak{B}|\Gamma|\text{.}
\end{equation*}
\end{definition}

We denote by $\big[ M , \mathfrak B |\Gamma|\big]$ the set of homotopy classes of classifying maps. 

\begin{proposition}[{{\cite[Theorem 1]{baez8}}}]
\label{bstheorem}
Let $\Gamma$ be a Lie 2-group. Then there is a bijection 
$$ \h_c^1(M,\Gamma) \cong \big[ M , \mathfrak B |\Gamma|\big]$$
where the topological group $|\Gamma|$ is the geometric realization of the nerve of $\Gamma$.
\end{proposition}

Note that the assumption of \cite[Theorem 1]{baez8} that $\Gamma$ is well-pointed is automatically satisfied because  Lie groups are well-pointed. Propositions \ref{smoothening} and \ref{bstheorem}  imply the following equivalence theorem between Version I and Version II.

\begin{theorem}
\label{th:classifying}
For $M$ a smooth manifold and $\Gamma$ a smoothly separable Lie 2-group, there is a  bijection
\begin{equation*}
\h^1(M,\Gamma) \cong \big[ M , \mathfrak B |\Gamma|\big].
\end{equation*}
\end{theorem}

\begin{remark}
\label{rem:baas}
 Baez and Stevenson argue in \cite[Section 5.2.]{baez8} that the space $\mathfrak B |\Gamma|$ is homotopy equivalent to a certain geometric realization of the Lie 2-groupoid  $|\mathcal B\Gamma|$ from Section \ref{versionI}. Baas, Böstedt and Kro have shown \cite{baas2} that $|\mathcal{B}\Gamma|$ classifies concordance classes of charted $\Gamma$-2-bundles. In particular, charted $\Gamma$-2-bundles are a further equivalent version of smooth $\Gamma$-gerbes. 
\end{remark}

\setsecnumdepth{2}  
 
\section{Version III: Groupoid Bundle Gerbes}

\label{versionIII}

Several definitions of non-abelian bundle gerbes have appeared in literature so far \cite{aschieri,jurco1,murray5}. The approach we give here not only shows a conceptually clear way to define  non-abelian bundle gerbes, but also produces systematically a whole bicategory. Moreover, these bicategories form a 2-stack over smooth manifolds (with the Grothendieck topology of surjective submersions).

\subsection{Definition via the Plus Construction}

Recall that the stack $\bun\Gamma-$ of principal $\Gamma$-bundles is monoidal if $\Gamma$ is a Lie 2-group (Proposition \ref{prop:monoidalstack}). Associated to the monoidal stack $\bun\Gamma-$ we have a pre-2-stack 
\begin{equation*}
\triv\grb \Gamma - := \mathcal{B}(\bun\Gamma-)
\end{equation*}
of \emph{trivial $\Gamma$-gerbes}. Explicitly, there is one trivial $\Gamma$-gerbe $\mathcal{I}$ over every smooth manifold $M$. The 1-morphisms from $\mathcal{I}$ to $\mathcal{I}$ are principal $\Gamma$-bundles over $M$, and the 2-morphisms between those are  morphisms of principal $\Gamma$-bundles.
Horizontal composition is given by the tensor product of principal $\Gamma$-bundles, and vertical composition is the ordinary composition of $\Gamma$-bundle morphisms. 

Now we apply the \emph{plus construction} of \cite{nikolaus2} in order to stackify this pre-2-stack. The resulting 2-stack is by definition the \emph{2-stack of $\Gamma$-bundle gerbes}, i.e.
\begin{equation*}
\grb\Gamma- := \left ( \triv\grb\Gamma- \right )^{+}\text{.}
\end{equation*}
Unwinding the details of the plus construction, we obtain the following definitions:

\begin{definition} \label{def:grb}
Let $M$ be a smooth manifold. A \emph{$\Gamma$-bundle gerbe over $M$} is a surjective submersion $\pi\maps Y \to M$, a principal $\Gamma$-bundle $P$ over $Y^{[2]}$ and an associative morphism
\begin{equation*}
\mu: \pi_{23}^{*}P \otimes \pi_{12}^{*}P \to \pi_{13}^{*}P
\end{equation*}
of $\Gamma$-bundles over $Y^{[3]}$.
\end{definition}

The morphism $\mu$ is called the \emph{bundle gerbe product}. 
Its associativity is the evident condition for bundle morphisms over $Y^{[4]}$.

In order to proceed with the 1-morphisms, we say that a \emph{common refinement} of two surjective submersions $\pi_1: Y_1 \to M$ and $\pi_2: Y_2 \to M$ is a smooth manifold $Z$ together with surjective submersions $Z \to Y_1$ and $Z \to Y_2$ such that the diagram
\begin{equation*}
\alxydim{@=0.6cm}{& Z \ar[dr]^{} \ar[dl]_{} & \\ Y_1 \ar[dr]_{\pi_1} && Y_2 \ar[dl]^{\pi_2} \\ & M&}
\end{equation*}
is commutative.

We fix the following convention: suppose $P_1$ and $P_2$ are $\Gamma$-bundles over surjective submersions $U_1$ and $U_2$, respectively, and $V$ is a common refinement of $U_1$ and $U_2$. Then, a bundle morphism $\varphi: P_1 \to P_2$ is understood to be a bundle morphism between the pullbacks of $P_{1}$ and $P_2$ to the common refinement $V$. For example, in the following definition
this convention applies to $U_1 = Y_1^{[2]}$, $U_2 = Y_2^{[2]}$ and $V =\ Z^{[2]}$.

\begin{definition}
\label{def1mor}
Let $\mathcal{G}_1$ and $\mathcal{G}_2$ be $\Gamma$-bundle gerbes over $M$. A \emph{1-morphism} $\mathcal{A}: \mathcal{G}_1 \to \mathcal{G}_2$ is a common refinement $Z$ of the surjective submersions of $\mathcal{G}_1$ and $\mathcal{G}_2$ together with a principal $\Gamma$-bundle $Q$ over $Z$ and a morphism
\begin{equation*}
\beta: P_2 \otimes \zeta_1^{*}Q  \to \zeta_2^{*}Q \otimes P_1
\end{equation*}
of $\Gamma$-bundles over $Z^{[2]}$,  where $\zeta_{1},\zeta_2: Z^{[2]} \to Z$ are the two projections, such that $\alpha$ is compatible with the bundle gerbe products $\mu_1$ and $\mu_2$.
\end{definition}

The compatibility of $\alpha$ with $\mu_1$ and $\mu_2$ means that the diagram
\begin{equation}
\label{compgerbemorph}
\alxydim{@C=2cm@R=1.3cm}{\pi_{23}^{*}P_2 \otimes \pi_{12}^{*}P_2 \otimes \zeta_1^{*}Q  \ar[d]_{\id \otimes \zeta_{12}^{*}\beta} \ar[r]^-{\mu_2 \otimes \id} & \pi_{13}^{*}P_2 \otimes \zeta_1^{*}Q \ar[dd] ^{\zeta_{13}^{*}\beta} \\ \pi_{23}^{*}P_2 \otimes \zeta_2^{*}Q \otimes \pi_{12}^{*}P_1 \ar[d]_{\zeta_{23}^{*}\beta \otimes \id} & \\ \zeta_3^{*}Q \otimes \pi_{23}^{*}P_1 \otimes \pi_{12}^{*}P_1 \ar[r]_-{ \id \otimes \mu_1} & \zeta_3^{*}Q \otimes \pi_{13}^{*}P_1}
\end{equation} 
of morphisms of $\Gamma$-bundles over $Z^{[3]}$ is commutative. 

If $\mathcal{A}_{12}: \mathcal{G}_1 \to \mathcal{G}_2$ and $\mathcal{A}_{23}: \mathcal{G}_2 \to \mathcal{G}_3$ are 1-morphisms between bundle gerbes over $M$, the composition $\mathcal{A}_{23} \circ \mathcal{A}_{12}: \mathcal{G}_1 \to \mathcal{G}_3$ is given by the fibre product $Z := Z_{23} \times_{Y_2} Z_{12}$, the principal $\Gamma$-bundle $Q := Q_{23} \otimes Q_{12}$ over $Z$, and the morphism
\begin{equation*}
\alxydim{@C=2cm}{P_3 \otimes \zeta_1^{*}Q   \ar[r]^-{\beta_{23} \otimes \id} & \zeta_2^{*}Q_{23} \otimes P_2 \otimes \zeta_1^{*}Q_{12} \ar[r]^-{\id \otimes \beta_{12}} &   \zeta_2^{*}Q \otimes P_1 \text{.}}
\end{equation*}
The identity 1-morphism $\id_{\mathcal{G}}$ associated to a $\Gamma$-bundle gerbe $\mathcal{G}$ is given by $Y$ regarded as a common refinement of $\pi:Y \to M$ with itself,  the trivial $\Gamma$-bundle $\trivlin_1$ (the tensor unit of $\bun\Gamma Y$), and the evident morphism $\trivlin_1 \otimes P \to P \otimes \trivlin_1$. 

In order to define 2-morphisms, suppose that $\pi_1: Y_1 \to M$ and $\pi_2:Y_2 \to M$ are surjective submersions, and that $Z$ and $Z'$ are common refinements of $\pi_1$ and $\pi_2$. Let $W$ be a  common refinement of $Z$ and $Z'$  with surjective submersions $r:W \to Z$ and $r':W \to Z'$.  We obtain  two maps
\begin{equation*}
\alxydim{}{s_1: W \ar[r]^-{r} & Z \ar[r] & Y_1 
\quad\quand\quad
t_1: W \ar[r]^-{r'} & Z' \ar[r] & Y_1\text{,}}
\end{equation*}
and analogously,  two maps $s_2,t_2: W \to Y_2$. These patch together to maps 
$$x_W:=(s_1,t_1)~:~ W \to Y_1 \times_M Y_1 \qquad\text{and}\qquad y_W:=(s_2,t_2)~:~ W \to Y_2 \times_M Y_2.$$

\begin{definition}
\label{def2mor}
Let $\mathcal{G}_1$ and $\mathcal{G}_2$ be $\Gamma$-bundle gerbes over $M$, and let $\mathcal{A},\mathcal{A}': \mathcal{G}_1 \to \mathcal{G}_2$ be 1-morphisms. A \emph{2-morphism} $\varphi: \mathcal{A} \Rightarrow \mathcal{A}'$ is a common refinement $W$ of the common refinements $Z$ and $Z'$, together with a morphism 
\begin{equation*}
\varphi: y_W^*P_2 \otimes r^{*}Q \to {r'}^{*}Q' \otimes x_W^* P_1
\end{equation*}
of $\Gamma$-bundles over $W$ that is compatible with the morphisms $\beta$ and $\beta'$.
\end{definition}

The compatibility means that a certain diagram over $W^{[2]}$ commutes. Fibrewise over a point $(w,w') \in W \times_M \nobr W$ this diagram looks as follows:
\begin{equation}
\label{eq:comp2morphcomplicated}
\alxydim{@R=1.3cm@C=0.35cm}{
{P_2}|_{s_2(w'),t_2(w')} \otimes {P_2}|_{s_2(w),s_2(w')} \otimes Q|_{r(w)}  \ar[rr]^-{\id \otimes \beta} \ar[d]_{\mu_2 \otimes \id} & &
{P_2}|_{s_2(w'),t_2(w')} \otimes Q|_{r(w')} \otimes {P_1}|_{s_1(w),s_1(w')} \ar[d]^{\varphi \otimes \id} \\ 
{P_2}|_{s_2(w),t_2(w')} \otimes Q|_{r(w)}\ar[d]_{\mu_2^{-1} \otimes \id} && 
{Q'}|_{r'(w')} \otimes {P_1}|_{s_1(w'),t_1(w')} \otimes {P_1}|_{s_1(w),s_1(w')} \ar[d]^{id \otimes \mu_1} \\
{P_2}|_{t_2(w),t_2(w')} \otimes {P_2}|_{s_2(w),t_2(w)}\otimes Q|_{r(w)}  \ar[d]_{\id \otimes \varphi} &&
 {Q'}|_{r'(w')} \otimes {P_1}|_{s_1(w),t_1(w')} \ar[d]^{id \otimes \mu_1^{-1}}\\
{P_2}|_{t_2(w),t_2(w')} \otimes {Q'}|_{r'(w)} \otimes {P_1}|_{s_1(w),t_1(w)} \ar[rr]_{\beta' \otimes \id}&&
{Q'}|_{r'(w')} \otimes {P_1}|_{t_1(w),t_1(w')} \otimes {P_1}|_{s_1(w),t_1(w)}
}
\end{equation} 
Finally we identify two 2-morphisms  $(W_1,r_1,r_1',\varphi_1)$ and $(W_2,r_2,r_2',\varphi_2)$ if the pullbacks of $\varphi_1$ 
and $\varphi_2$ to $W \times_{Z \times_{\times} Z'} \nobr W'$ agree. Explicitly, this condition means that for all $w_{1} \in W_1$ 
and $w_{2} \in W_{2}$ with $r_1(w_1)=r_2(w_2)$ and $r_1'(w_1) = r_2'(w_2)$, and for all
$p_2 \in y_{W_1}^{*}P_2 = y_{W_2}^{*}P_2$ and $q\in r_1^{*}Q = r_2^{*}Q$ we have $\varphi_1(p_2,q) = \varphi_2(p_2,q)$.

\begin{remark}
\label{rem:2morphisms}
\begin{itemize}
\item 
In the above situation of a common refinement $W$ of two common refinements $Z,Z'$ of surjective submersions $Y_1,Y_2$, the diagram 
\begin{equation}
\label{diag:refinement}
\alxydim{@=0.8cm}{
& Z \ar[dr]^{} \ar[dl]_{} & \\ 
Y_1  & W \ar[u]_-{r}\ar[d]^-{r'} & Y_2 \\
 & Z' \ar[lu]\ar[ru] &
}
\end{equation}
is \emph{not} necessarily commutative. In fact, diagram \erf{diag:refinement} commutes if and only if the two maps $x_W\maps W \to Y_1 \times_M Y_1$ and $y_W: W \to Y_2 \times_M Y_2$ factor through the diagonal maps $Y_1 \to Y_1 \times_M Y_1$ and $Y_2 \to Y_2 \times_M Y_2$, respectively.

\item
In the case that a 2-morphism $\varphi$ is defined on a common refinement $Z$ for which diagram \erf{diag:refinement} \emph{does} commute, Definition \ref{def2mor} can be simplified. As remarked before, the two maps $x_W$ and $y_W$ factor through the diagonals, over which  the bundles $P_1$ and $P_2$  have canonical trivializations (see Corollary \ref{co:groupoidstructure}). Under these trivializations, $\varphi$ can be identified with a bundle morphism
\begin{equation*}
\varphi: Q \to Q' \text{.}
\end{equation*}
Furthermore, the compatibility diagram \erf{eq:comp2morphcomplicated}  simplifies to the diagram
\begin{equation}
\label{eq:comp2morphsimple}
\alxydim{@=1.3cm}{P_2 \otimes \eta_1^{*}Q \ar[r]^{\beta} \ar[d]_{\id \otimes \eta_1^{*}\varphi}  & \eta_2^{*}Q \otimes P_1 \ar[d]^{\eta_2^{*}\varphi \otimes \id} \\ P_2 \otimes \eta_1^{*}Q' \ar[r]_{\beta'} & \eta_2^{*}Q' \otimes P_1\text{.}}
\end{equation} 
\end{itemize}
\end{remark}

Next we  define the vertical composition $\varphi_{23} \bullet \varphi_{12}: \mathcal{A}_1 \Rightarrow \mathcal{A}_3$ of 2-morphisms $\varphi_{12}: \mathcal{A}_1 \Rightarrow \mathcal{A}_2$ and $\varphi_{23} \maps \mathcal{A}_2 \Rightarrow \mathcal{A}_3$. The refinement is the fibre product $W :=  W_{12} \times_{Z_2} W_{23}$ of the covers of $\varphi_{12}$ and $\varphi_{23}$. The bundle gerbe products induce  isomorphisms
$$ x_W^* P_1 \cong x_{W_{23}}^*P_1 \otimes x_{W_{12}}^*P_1 \qquad \text{and} \qquad y_W^* P_2 \cong y_{W_{23}}^*P_2 \otimes y_{W_{12}}^*P_2$$
over $W$. Under these identifications, the morphism $y_W^*P_2 \otimes Q_1 \to Q_3 \otimes x_W^*P_1$ for the 2-morphism $\varphi_{23} \bullet \varphi_{12}$ is defined as
\begin{equation*}
\alxydim{@C=0.6cm}{
y_{W_{23}}^*P_2 \otimes y_{W_{12}}^*P_2 \otimes Q_1 \ar[rr]^-{\id \otimes \varphi_{12}} && 
y_{W_{23}}^*P_2 \otimes Q_2 \otimes x_{W_{12}}^*P_1 \ar[rr]^-{\varphi_{23} \otimes \id} &&
Q_3 \otimes x_{W_{23}}^*P_1  \otimes x_{W_{12}}^*P_1\text{.}
}
\end{equation*}
The identity for vertical composition is just the identity refinement and the identity morphism. Finally we come to the horizontal composition 
\begin{equation*}
\varphi_{23} \circ \varphi_{12}: \mathcal{A}_{23} \circ \mathcal{A}_{12} \Rightarrow \mathcal{A}_{23}' \circ \mathcal{A}_{12}'
\end{equation*}
of 2-morphisms $\varphi_{12}: \mathcal{A}_{12} \Rightarrow \mathcal{A}_{12}'$ and $\varphi_{23}: \mathcal{A}_{23} \Rightarrow \mathcal{A}_{23}'$: its refinement $W$ is given by $W_{12} \times_{(Y_2 \times Y_2)} W_{23}$. We look at the three relevant maps  $x_W: W \to Y_1 \times_M Y_1$, $y_W: W \to Y_2 \times_M Y_2$ and $z_W: W \to Y_3 \times_M Y_3$. The morphism $\varphi$ of the 2-morphism $\varphi_{23} \circ \varphi_{12}$ is defined as the composition
\begin{equation*}
\alxydim{}{
z_W^* P_3 \otimes Q_{23} \otimes Q_{12} \ar[rr]^-{\varphi_{23} \otimes \id} &&
Q'_{23} \otimes y_W^*P_2 \otimes Q_{12} \ar[rr]^-{\id \otimes \varphi_{12}} &&
Q'_{23} \otimes Q'_{12} \otimes x_W^*P_1 \text{.}
}
\end{equation*}

It follows from the properties of the plus construction \cite{nikolaus2} that (a) these definitions fit together into a bicategory $\grb\Gamma M$, and that (b) these form a pre-2-stack $\grb\Gamma-$ over smooth manifolds. That means, there are \emph{pullback 2-functors}
\begin{equation*}
f^{*}: \grb\Gamma N \to \grb\Gamma M
\end{equation*}
associated to  smooth maps $f:M \to N$, and that these are compatible with the composition of smooth maps. Pullbacks of $\Gamma$-bundle gerbes, 1-morphisms, and 2-morphisms are  obtained by just taking the pullbacks of all involved data.
Finally, the plus construction implies (c):

\begin{theorem}[{{\cite[Theorem 3.3]{nikolaus2}}}]
\label{th:stack}
The pre-2-stack $\grb\Gamma-$ of $\Gamma$-bundle gerbes is a 2-stack. 
\end{theorem}

\begin{remark}
\label{rem:grbgrp}
Every 2-stack over smooth manifolds defines a 2-stack over Lie groupoids \cite[Proposition 2.8]{nikolaus2}. This way, our approach produces automatically bicategories $\grb \Gamma{\mathcal{X}}$ of $\Gamma$-bundle gerbes over a Lie groupoid $\mathcal{X}$. In particular, for an action groupoid $\mathcal{X}=\act MG$ we have a bicategory $\grb\Gamma{\act MG}$ of \emph{$G$-equivariant $\Gamma$-bundle gerbes} over $M$.
\end{remark}

In the remainder of this section we give some examples and describe relations between the definitions given here and existing ones.

\begin{example}
\label{ex:abelianbundlegerbes}
Let $A$ be an abelian Lie group, for instance $\ueins$. Then, $\mathcal{B}A$-bundle gerbes are  the same as the well-known $A$-bundle gerbes \cite{murray}. For more details see Remark \ref{re:1-morphisms} below.
\end{example}

\begin{example}
Let $(G,H,t,\alpha)$ be a smooth crossed module, and let $\act GH$ be the associated action groupoid. Then, a $(\act GH)$-bundle gerbe is the same as a \emph{crossed module bundle gerbe} in the sense of Jurco \cite{jurco1}. The equivalence relation \quot{stably isomorphic} of \cite{jurco1} is given by \quot{1-isomorphic} in terms of the bicategory constructed here. These coincidences come from the equivalence between $(\act GH)$-bundles and so-called $G$-$H$-bundles used in \cite{jurco1,aschieri} expressed by Lemma \ref{crossedbundles}. In particular, in case of the automorphism 2-group $\mathrm{AUT}(H)$ of a connected Lie group $H$, a $\mathrm{AUT}(H)$-bundle gerbe is the same as a $H$-bibundle gerbe in the sense of Aschieri, Cantini and Jurco \cite{aschieri}.
\end{example}

\begin{example}
\label{ex:idmorphgrb}
Let $G$ be a Lie group, so that $\idmorph{G}$ is a Lie 2-group. Then, there is an equivalence of 2-categories
\begin{equation*}
\grb{{\idmorph{G}}}M \cong \idmorph{\bun{\mathcal{B}G}M}\text{.}
\end{equation*}
Indeed, if $\mathcal{G}$ is a $\idmorph G$-bundle gerbe over $M$, its principal $\idmorph G$-bundle over $Y^{[2]}$ is by Example \ref{ex:idmorphbundle} just a smooth map $\alpha: Y^{[2]} \to G$, and its bundle gerbe product degenerates to an equality $\pi_{23}^{*}\alpha \cdot \pi_{12}^{*}\alpha = \pi_{13}^{*}\alpha$ for functions on $Y^{[3]}$. In other words, a $\idmorph{G}$-bundle gerbe is the same as a so-called \quot{$G$-bundle 0-gerbe}. These form a category that is equivalent to the one of ordinary principal $G$-bundles, as pointed out in  Section \ref{intro}.  
\end{example}

\begin{remark}
\label{re:1-morphisms}
There are two differences between the definitions given here (for $\Gamma = \mathcal{B}A$) and the ones of Murray et al. \cite{murray, murray2, stevenson1}. Firstly, we have a slightly  different  ordering of  tensor products of bundles. These orderings are not essential in the case of abelian groups because the tensor category of ordinary $A$-bundles is symmetric. In the non-abelian case, a consistent theory requires the conventions we have chosen here. Secondly, the definitions of 1-morphisms and 2-morphisms have been generalized step by step:
\begin{enumerate}
\item 
In \cite{murray}, 1-morphisms did not include a \emph{common} refinement, but rather required that the surjective submersion of one bundle gerbe refines
the other. This definition is too restrictive in the sense that e.g. $\ueins$-bundle gerbes are not classified by $\mathrm{H}^3(M,\Z)$, as intended.

\item
In \cite{murray2}, 1-morphisms were defined on the canonical refinement $Z := Y_1 \times_M Y_2$ of the surjective submersions of the bundle gerbes. This definition solves the previous problems concerning the classification of bundle gerbes, but makes the composition of 1-morphisms quite involved \cite{stevenson1}.

\item
In \cite{waldorf1}, 1-morphisms were defined on refinements $\zeta: Z \to Y_1 \times_M Y_2$. This generalization allows the same elegant definition of composition we have given here, and results in the same isomorphism classes of bundle gerbes. Moreover, 2-morphisms are defined with \emph{commutative} diagrams \erf{diag:refinement} -- this makes the structure of the bicategory outmost simple (see Remark \ref{rem:2morphisms}). 

\item
In the present article we have allowed for a yet more general refinement in the definition of 1-morphisms. Its achievement is that bundle gerbes come out as an example of a more general concept -- the plus construction -- and we get e.g. Theorem \ref{th:stack} for free.  

\end{enumerate}
Despite these different definitions of 1-morphisms and 2-morphisms, the resulting bicategories of $\mathcal{B}A$-bundle gerbes in 2., 3. and 4. are all equivalent (see \cite[Theorem 1]{waldorf1}, \cite[Remark 4.5]{nikolaus2} and Lemma \ref{lem:canonicalrefinement} below). 
\end{remark}

\subsection{Properties of Groupoid Bundle Gerbes}

\label{sec:gerbeprop}

We recall that a homomorphism
 $\Lambda: \Gamma \to \Omega$ between Lie 2-groups is an anafunctor together with a transformation \erf{eq:lie2grouphom} describing its compatibility with the multiplications. We recall further from Proposition \ref{extensionstackmorphism} that   extension along $\Lambda$ is a 1-morphism
\begin{equation*}
\Lambda: \bun{\Gamma}- \to \bun\Omega-
\end{equation*}
between monoidal stacks over smooth manifolds.
That is, extension along $\Lambda$ is compatible with pullbacks, tensor products, and morphisms between principal $\Gamma$-bundles. Applying it to the principal $\Gamma$-bundle $P$ of a $\Gamma$-bundle gerbe $\mathcal{G}$, and also to the bundle gerbe product $\mu$, we obtain immediately an $\Omega$-bundle gerbe $\Lambda\mathcal{G}$. The same is evidently true for morphisms and 2-morphisms. Summarizing,
we get:
\begin{proposition}
\label{grbextension}
Extension of bundle gerbes along a homomorphism $\Lambda: \Gamma \to \Omega$ between Lie 2-groups defines a 1-morphism
\begin{equation*}
\Lambda: \grb\Gamma- \to \grb\Omega-
\end{equation*}
of 2-stacks over smooth manifolds. 
\end{proposition}

We recall that a weak equivalence between Lie 2-groups is a homomorphism $\Lambda: \Gamma \to \Omega$ that is a weak equivalence (see Definition \ref{defweak}). We have:

\begin{theorem}\label{extequiv}
Suppose $\Lambda: \Gamma \to \Omega$ is a weak equivalence between Lie 2-groups. Then, the 1-morphism $\Lambda: \grb\Gamma- \to \grb\Omega-$ of Proposition \ref{grbextension} is an equivalence of 2-stacks.
\end{theorem}
\begin{proof}
The monoidal equivalence  $\Lambda:\bun\Gamma- \to \bun\Omega-$ between the monoidal stacks (Corollary \ref{co:extensionweakequivalence2groups})
induces an equivalence
$\triv\grb\Gamma M \to \triv\grb\Lambda M$ between pre-2-stacks. Since the plus construction is functorial, this induces in turn the claimed  equivalence of 2-stacks.
\end{proof}

Next we generalize a couple of well-known results from abelian to non-abelian bundle gerbes. We define a \emph{refinement} of a surjective submersion $\pi:Y \to M$ to be another surjective submersion $\omega: W \to M$ together with a smooth map $f:W \to Y$ such that $\zeta = \pi \circ f$.
Notice that such a refinement induces smooth maps $f_k: W^{[k]} \to Y^{[k]}$ that commute with the various projections $\omega_{i_1...i_k}$ and $\pi_{i_1...i_k}$.

\begin{lemma}
\label{lem:nonstablemorphisms}
Suppose $\mathcal{G}_1 = (Y_1,P_1,\mu_1)$ and $\mathcal{G}_2=(Y_2,P_2,\mu_2)$ are $\Gamma$-bundle gerbes over $M$, $f: Y_1 \to Y_2$ is a refinement of surjective submersions, and $\varphi: f_2^{*}P_2 \to P_1$ is an isomorphism of $\Gamma$-bundles over $Y_1^{[2]}$ that is compatible with the bundle gerbe products $\mu_1$ and $\mu_2$ in the sense that the diagram
\begin{equation*}
\alxydim{@=1.3cm}{\pi_{23}^{*}f_2^{*}P_2 \otimes \pi_{12}^{*}f_2^{*}P_2 \ar[d]_{\pi_{23}^{*}\varphi \otimes \pi_{12}^{*}\varphi} \ar[r]^-{f_3^{*}\mu} & \pi_{13}^{*}f_2^{*}P_2 \ar[d]^{\pi_{13}^{*}\varphi} \\ \pi_{23}^{*}P_1 \otimes \pi_{12}^{*}P_1 \ar[r]_-{\mu} & \pi_{13}^{*}P_1}
\end{equation*}
is commutative. Then, $\mathcal{G}_1$ and $\mathcal{G}_2$ are isomorphic.
\end{lemma}

The proof works just the same way as in the abelian case: one constructs the 1-isomorphism over the  common refinement $Z := Y_1 \times_M Y_2$ in a straightforward way. 
As a consequence of Lemma \ref{lem:nonstablemorphisms} we have

\begin{proposition}
\label{prop:pullback}
Let $\mathcal{G} = (Y,P,\mu)$ be a $\Gamma$-bundle gerbe over $M$, and let $f:W \to Y$ be a refinement of its surjective submersion $\pi:Y \to M$. Then, $(W,f_2^{*}P,f_3^{*}\mu)$
is a $\Gamma$-bundle gerbe over $M$, and is isomorphic to $\mathcal{G}$.
\end{proposition}

\begin{lemma}
\label{lem:buntriv}
Let $\mathcal{G}=(Y,P,\mu)$ be a $\Gamma$-bundle gerbe over $M$. Then, there exist unique smooth maps $i: P \to P$ and $t: Y \to P$ such that
\begin{enumerate}[(i)]
\item 
the diagrams
\begin{equation*}
\alxydim{@=1.3cm}{P \ar[r]^{i} \ar[d]_{\chi} & P \ar[d]^{\chi} \\ Y^{[2]} \ar[r]_{\mathrm{flip}} & Y^{[2]}}
\quand
\alxydim{@=1.3cm}{& P \ar[d]^{\chi} \\ Y \ar[r]_-{\mathrm{diag}} \ar[ur]^{t} & Y^{[2]}}
\end{equation*}
are commutative.

\item
the map 
$t$ is neutral with respect to the bundle gerbe product $\mu$, i.e.
\begin{equation*}
\mu(t(y_{2}),p) = p = \mu(p,t(y_{1}))\text{.}
\end{equation*}
for all $p \in P$ with $\chi(p)=(y_1,y_2)$.

\item
the map $i$ provides inverses with respect to the bundle gerbe product $\mu$, i.e.
\begin{equation*}
\mu(i(p),p) = t(y_{1})
\quand
\mu(p,i(p)) = t(y_2)
\end{equation*}
for all $p\in P$ with $\chi(p)=(y_1,y_2)$.
\end{enumerate}
Moreover, 
$\alpha(t(y))=1$ and $\alpha(i(p)) = \alpha(p)^{-1}$
for all $p\in P$ and $y\in Y$.
\end{lemma}

\begin{proof}
Concerning uniqueness, suppose $(t,i)$ and $(t',i')$ are pairs of maps satisfying (i), (ii) and (iii). Firstly, we have $t'(y) = \mu(t(y),t'(y)) = t(y)$ and so $t=t'$. Then, $\mu(i(p),p)= t(y_1)=t'(y_1)=\mu(i'(p),p)$ implies $i(p)=i'(p)$, and so $i=i'$. In order to see the existence of $t$ and $i$, denote by $Q := \mathrm{diag}^{*}P$ the pullback of $P$ to $Y$,  denote by $Q^{\vee}$ the dual bundle and by $d: Q \otimes Q^{\vee} \to \trivlin_1$ the death map. Consider the smooth map
\begin{equation*}
\alxydim{@C=1cm}{Y \ar[r]^{s} & \trivlin_1 \ar[r]^-{d^{-1}} & Q \otimes Q^{\vee} \ar[rr]^-{\mu^{-1} \otimes \id_{Q^{\vee}}} && Q \otimes Q \otimes Q^{\vee} \ar[r]^-{\id \otimes d} & \trivlin_1 \otimes Q \cong Q \ar[r]^-{\mathrm{diag}} & P }
\end{equation*}
where $s:Y \to \trivlin_1$ is the canonical section
(see the proof of Lemma \ref{lem:loctriv}). It is straightforward to see that this satisfies the properties of the map $t$. Since all maps in the above sequence are (anchor-preserving) bundle morphisms, it is clear that $t \circ \alpha = 1$.  
\end{proof}

\begin{corollary}
\label{co:groupoidstructure}
Let $\mathcal{G}=(Y,P,\mu)$ be a $\Gamma$-bundle gerbe over $M$, and let $t$ and $i$ be the unique maps of Lemma \ref{lem:buntriv}. Then,
\begin{enumerate}[(i)]
\item
\label{co:groupoidstructure:cantriv}
$t$ is a section of $\mathrm{diag}^{*}P$, and defines a trivialization $\mathrm{diag}^{*}P \cong \trivlin_1$.  

\item
$i$ is a bundle isomorphism $i: P^{\vee} \to \mathrm{flip}^{*}P$.

\item 
$\mathcal{C}_0 := Y$ and $\mathcal{C}_1 := P$ define a Lie groupoid with source and target maps $\pi_1 \circ \chi$ and $\pi_2 \circ \chi$, respectively, composition $\mu$, identity $t$ and inversion $i$. \end{enumerate}
\end{corollary}

The following statement is well-known for abelian gerbes; the general version can be proved by a straightforward generalization of the constructions given in the proof of \cite[Proposition 3]{waldorf1}. 

\begin{lemma}
\label{lem:invertibility}
Every 1-morphism $\mathcal{A}: \mathcal{G} \to \mathcal{H}$ between $\Gamma$-bundle gerbes over $M$ is invertible. 
\end{lemma}

The last statement of this section shows a way to bring 1-morphisms and 2-morphisms into a simpler form (see Remark \ref{re:1-morphisms}). For bundle gerbes $\mathcal{G}_1$ and $\mathcal{G}_2$ with surjective submersions $\pi_1:Y_1 \to M$ and $\pi_2:Y_2 \to M$ we denote by $\hom {\mathcal{G}_1}{\mathcal{G}_2}$ the Hom-category in the bicategory $\grb\Gamma M$, and by $\hom {\mathcal{G}_1}{\mathcal{G}_2}^{FP}$ the category  whose objects are those 1-morphisms whose common refinement is $Z := Y_1 \times_M Y_2$, and whose 2-morphisms are those 2-morphisms whose refinement is $W := Y_1 \times_M Y_2$ with the maps $r,r': W \to Z$ the identity maps.

\begin{lemma}
\label{lem:canonicalrefinement}
The inclusion
$\hom {\mathcal{G}_1}{\mathcal{G}_2}^{FP} \to  \hom {\mathcal{G}_1}{\mathcal{G}_2}$
is an equivalence of categories.
\end{lemma}

\begin{proof}
First we show that it is essentially surjective. We assume $\mathcal{A}: \mathcal{G}_1 \to \mathcal{G}_2$ is a general 1-morphism with a principal $\Gamma$-bundle $Q$ over a common refinement $Z$ of the surjective submersions $\pi_1:Y_1 \to M$ and $\pi_2:Y_2 \to M$ of the two bundle gerbes. We look at the principal $\Gamma$-bundle
\begin{equation*}
\tilde Q := \kappa_2^{*}P_2 \otimes \mathrm{pr}_2^{*}Q \otimes \kappa_1^{*}P_1
\end{equation*}
over $\tilde Z := Y_1 \times_M Z \times_M Y_2$, where 
\begin{equation*}
\kappa_1: \tilde Z \to Y_1^{[2]}: (y_1,z,y_2) \mapsto (y_1,y_1(z))
\;\text{ and }\;
\kappa_2: \tilde Z \to Y_2^{[2]}: (y_1,z,y_2) \mapsto (y_2(z),y_2)\text{.}
\end{equation*}
The projection $\mathrm{pr}_{13}: \tilde Z \to Y_1 \times_M Y_2$ is a surjective submersion, and over
 $\tilde Z \times_{Y_1 \times_M Y_2} \tilde Z$ we have a bundle morphism
$\alpha : \mathrm{pr}_1^{*}\tilde Q \to \mathrm{pr}_2^{*}\tilde Q$
defined over a point $(\tilde z,\tilde z')$ with $\tilde z = (y_1,z,y_2)$ and $\tilde z' = (y_1,z',y_2)$ by
\begin{equation*}
\alxydim{@R=0.6cm}{\tilde Q_{\tilde z} \ar@{=}[r] & P_2|_{y_2(z),y_2} \otimes Q_{z} \otimes P_1|_{y_1,y_1(z)} \ar[d]^-{\mu_2^{-1} \otimes \id \otimes \id} \\ & P_2|_{y_2(z'),y_2} \otimes P_2|_{y_2(z),y_2(z')} \otimes Q_{z} \otimes P_1|_{y_1,y_1(z)} \ar[d]^{\id \otimes \beta \otimes \id}  \\ & P_2|_{y_2(z'),y_2} \otimes Q_{z'} \otimes P_1|_{y_1(z),y_1(z')} \otimes P_1|_{y_1,y_1(z)} \ar[d]^{ \id \otimes \id \otimes \mu_1}  \\ &  P_2|_{y_2(z'),y_2} \otimes Q_{z'} \otimes P_1|_{y_1,y_1(z')} \ar@{=}[r] & \tilde Q_{\tilde z'}\text{.} }
\end{equation*}
The compatibility condition \erf{compgerbemorph} implies a cocycle condition for $\alpha$ over the three-fold fibre product of $\tilde Z$ over $Y_1 \times_M Y_2$, and since principal $\Gamma$-bundles form a stack, the pair $(\tilde Q,\alpha)$ defines a principal $\Gamma$-bundle $Q^{FP}$ over $Z^{FP} := Y_1 \times_M Y_2$. It is now straightforward to show that the bundle isomorphism $\beta$ itself descends to a bundle isomorphism $\beta^{FP}$ over $Z^{FP} \times_M Z^{FP}$ in such a way that the triple $(Z^{FP},Q^{FP},\beta^{FP})$ forms a 1-morphism $\mathcal{A}^{FP}: \mathcal{G}_1 \to \mathcal{G}_2$.

In order to show that $\mathcal{A}^{FP}$ is an essential preimage of $\mathcal{A}$, it remains to construct a 2-morphism $\varphi_{\mathcal{A}}^{FP}\maps  \mathcal{A} \Rightarrow \mathcal{A}^{FP}$. In the terminology of Definition \ref{def2mor}, we choose $W = \tilde Z$ with $r := \mathrm{pr}_2:W \to Z$   and $r' := \mathrm{pr}_{13} : W \to Z^{FP}$. Note that diagram \erf{diag:refinement} does not commute. The maps $x_W: W \to Y_1^{[2]}$ and $y_W: W \to Y_2^{[2]}$ are given by $x_W = s \circ \kappa_1$ and $y_W = \kappa_2$, where $s: Y_1^{[2]} \to Y_1^{[2]}$ switches the factors. Now, the bundle isomorphism of the 2-morphism $\varphi_{\mathcal{A}}^{FP}$ we want to construct is a bundle isomorphism
\begin{equation*}
\varphi : y_W^{*}P_2 \otimes r^{*}Q \to \tilde Q \otimes x_W^{*}P_1
\end{equation*}
over $W$, and is fibrewise over a point $w=(y_1,z,y_2)$ given by
\begin{equation*}
\alxydim{@R=0.6cm}{P_2|_{y_2(z),y_2} \otimes Q_z \ar[r]^-{\id \otimes \id \otimes t^{-1}} & P_2|_{y_2(z),y_2} \otimes Q_z \otimes P_{y_1(z),y_1(z)} \ar[d]^{\id \otimes \id \otimes \mu_1^{-1}}   \\ & P_2|_{y_2(z),y_2} \otimes Q_z \otimes P_1|_{y_1,y_1(z)} \otimes P_1|_{y_1(z),y_1} \ar@{=}[r] & \tilde Q_{w} \otimes P_1|_{s(y_1,y_1(z))}\text{,}}
\end{equation*}
where $t$ is the trivialization of $\mathrm{diag}^{*}P$ of Corollary \ref{co:groupoidstructure}. The compatibility condition \erf{eq:comp2morphcomplicated} is straightforward to check.

Now we show that the inclusion $\hom {\mathcal{G}_1}{\mathcal{G}_2}^{FP} \to  \hom {\mathcal{G}_1}{\mathcal{G}_2}$
is full and faithful. Since it is clearly faithful, it only remains to show that it is full. Given a morphism $\mathcal{A} \to \mathcal{A}' $ in $\hom {\mathcal{G}_1}{\mathcal{G}_2}$, i.e. a common refinement $W$ of $Y_1 \times_M Y_2$ with itself and a bundle morphism $\varphi$, we have to find a morphism in $\hom {\mathcal{G}_1}{\mathcal{G}_2}^{FP}$ such that the two morphisms are identified under the equivalence relation on bundle gerbe 2-morphisms. We denote the bundles over $Y_1 \times_M Y_2$ corresponding to $\mathcal{A}$ and $\mathcal{A}'$ by $Q$ and $Q'$. The refinement maps are denoted as before by $r=(s_1,s_2): W \to Y_1 \times_M Y_2$ and $r'= (t_1,t_2): W \to Y_1 \times_M Y_2$. Then we obtain an isomorphism
$r^*Q \to {r}^*Q'$ 
fibrewise over a point $w \in W$ by
\begin{equation}
\label{eq:morlong}
\alxydim{@=0.6cm}{
Q|_{s_1(w),s_2(w)} \ar[rr]^-{d^{-1} \otimes \id} & & 
P_2^\vee |_{s_2(w),t_2(w)} \otimes P_2|_{s_2(w),t_2(w)} \otimes Q|_{s_1(w),s_2(w)}\ar[d]^-{\id \otimes \varphi}&&  \\
&& P_2^\vee|_{s_2(w),t_2(w)} \otimes Q'|_{t_1(w),t_2(w)} \otimes P_1|_{s_1(w),t_1(w)} \ar[d]^-{\id \otimes \beta'^{-1}} &&\\
&& P_2^\vee|_{s_2(w),t_2(w)} \otimes P_2|_{s_2(w),t_2(w)} \otimes Q'|_{s_1(w),s_2(w)} \ar[rr]^-{d \otimes \id} && Q'|_{s_1(w),s_2(w)}
}\hspace{-1cm}
\end{equation}
where $d: P_2^\vee |_{s_2(w),t_2(w)} \otimes P_2|_{s_2(w),t_2(w)} \to \trivlin_1$ is the death map. One can use the compatibility condition for $\varphi$ to show that this morphism descends to a morphism $\psi: Q \to Q'$ which is a morphism in 
$\hom {\mathcal{G}_1}{\mathcal{G}_2}^{FP}$. The two morphisms $(W,\psi)$ and $(Y_1\times_MY_2,\varphi)$ are identified if their pullbacks to 
\begin{equation*}
W \times_{(Y_1 \times_M Y_2 \times_M Y_1 \times_M Y_2)} (Y_1 \times_M Y_2) = \{w \in W \mid r(w) = r'(w) \} =: W_0 
\end{equation*}
are equal. On the one side, the map $W_0 \to W$ is the inclusion and the map $W_0 \to Y_1 \times_M Y_2$ is equal to $r$. The pullback of $\psi$ along $r$ is by construction the map $r^*Q \to r^*Q'$ from \eqref{eq:morlong}. On the other side,  bundles $x^*_W P_1$ and $y^*_W P_2$ over $W_0$ have canonical trivializations (Lemma \ref{co:groupoidstructure} \erf{co:groupoidstructure:cantriv}) under which $\varphi$ becomes also equal to the morphism \eqref{eq:morlong}.
\end{proof}

\subsection{Classification by \v Cech Cohomology}

In this section we prove that Versions I (\v Cech $\Gamma$-1-cocycles) and III ($\Gamma$-bundle gerbes) are equivalent. For this purpose, 
we extract a \v Cech cocycle from a $\Gamma$-bundle gerbe $\mathcal{G}$ over $M$, and prove that this procedure defines a bijection on the level of equivalence classes (Theorem \ref{th:gerbesclass}). First we have to ensure the existence of appropriate open covers.

\begin{lemma}
\label{lem:goodcover}
For every $\Gamma$-bundle gerbe $\mathcal{G}=(Y,P,\mu)$ over $M$ there exists an open cover $\mathscr{U}=\left \lbrace U_i \right \rbrace_{i\in I}$ of $M $with sections $\sigma_i : U_i \to Y$, such that the principal $\Gamma$-bundles $(\sigma_{i} \times \sigma_j)^{*}P$ over $U_i \cap U_j$ are trivializable. \end{lemma}

\begin{proof}
One can choose an open cover such that the 2-fold intersections $U_i \cap U_j$ are contractible. Since every Lie 2-group is a crossed module $\act GH$ (Remark \ref{rem:crossedmodule}), and $\act GH$-bundles are ordinary $H$-bundles (Lemma \ref{actionbundles}), these  admit sections over contractible smooth manifolds. But a section is enough to trivialize the original $\Gamma$-bundle (Lemma \ref{lem:loctriv}).
\end{proof}

Let $\mathcal{G}$ be a $\Gamma$-bundle gerbe over $M$, and let $\mathscr{U} = \left \lbrace U_i \right \rbrace_{i\in I}$ be an open cover with the properties of Lemma \ref{lem:goodcover}. We denote by $M_\mathscr{U}$ the disjoint union of all the open sets $U_i$, and by $\sigma:M_{\mathscr{U}} \to Y$ the union of the sections $\sigma_i$. Then, $\sigma$ is a refinement of $\pi:Y \to M$, and we have a $\Gamma$-bundle gerbe $\mathcal{G}_{\mathscr{U},\sigma}$ that is isomorphic to $\mathcal{G}$ (Proposition \ref{prop:pullback}). 

The principal $\Gamma$-bundle $P_{ij}$ of $\mathcal{G}_{\mathscr{U},\sigma}$ over the component $U_i \cap U_j$ is by assumption trivializable. Thus there exists a trivialization
$t_{ij}: P_{ij} \to \trivlin_{f_{ij}}$
for  smooth functions $f_{ij}:U_i \cap U_j \to \Gamma_0$. We define an isomorphism $\mu_{ijk}$ between trivial bundles such that the diagram
\begin{equation*}
\alxydim{@=1.3cm}{P_{jk} \otimes P_{ij} \ar[r]^-{\mu} \ar[d]_{t_{jk} \otimes t_{ij}} & P_{ik} \ar[d]^{t_{ik}} \\ \trivlin_{f_{jk}} \otimes \trivlin_{f_{ij}} \ar[r]_-{\mu_{ijk}} & \trivlin_{f_{ik}}}
\end{equation*}
is commutative. Now we are in the situation of Lemma \ref{lem:nonstablemorphisms}, which implies that the $\Gamma$-bundle gerbe $\mathcal{G}_{\mathscr{U},\sigma,t} \df (M_{\mathscr{U}},\trivlin_{f_{ij}},\mu_{ijk})$ is still isomorphic to $\mathcal{G}$.

Combining Lemma \ref{homsets} with Example \ref{ex:tensorproduct} \erf{ex:tensorproduct:trivial}, we see that the  isomorphisms $\mu_{ijk}$ correspond to smooth maps $g_{ijk}:U_i \cap U_j \cap U_k \to \Gamma_1$ such that $s(g_{ijk})=f_{jk} \cdot f_{ij}$ and $t(g_{ijk})=f_{ik}$. The associativity condition for $\mu_{ijk}$ implies moreover that
\begin{equation*}
g_{\alpha\gamma\delta} \circ (g_{\alpha\beta\gamma} \cdot \id_{f_{\gamma \delta}}) = 
g_{\alpha \beta \delta} \circ (\id_{f_{\alpha \beta}} \cdot g_{\beta \gamma \delta})\text{.}
\end{equation*}
Hence, the collection $\left \lbrace f_{ij},g_{ijk} \right \rbrace$ is a $\Gamma$-1-cocycle on $M$ with respect to the open cover $\mathscr{U}$. 

\begin{theorem}
\label{th:gerbesclass}
Let $M$ be a smooth manifold and let $\Gamma$ be a Lie 2-group. The above construction defines a bijection
\begin{equation*}
\bigset{3.7cm}{Isomorphism classes of $\Gamma$-bundle gerbes over $M$} \cong \h^1(M,\Gamma)\text{.}
\end{equation*}
\end{theorem}

\begin{proof}
We claim  that $\Gamma$-bundle gerbes $(M_{\mathscr{U}},\trivlin_{f_{ij}},\mu_{ijk})$ and $(M_{\mathscr{V}},\trivlin_{h_{ij}},\nu_{ijk})$ are isomorphic if and only if the corresponding $\Gamma$-1-cocycles are equivalent. This proves at the same time that the choices of open covers and sections we have made during the construction do not matter,  that the resulting map is well-defined on isomorphism classes, and that this map is injective. Surjectivity follows by assigning to a $\Gamma$-1-cocycle $(f_{ij},g_{ijk})$ with respect to some cover $\mathscr{U}$ the $\Gamma$-bundle gerbe $(M_{\mathscr{U}},\trivlin_{f_{ij}},\mu_{ijk})$ with $\mu_{ijk}$ determined by Lemma \ref{homsets}.

It remains to prove that claim. We assume $\mathcal{A}=(Z,Q,\alpha)$ is a 1-isomorphism between the $\Gamma$-bundle gerbes $(M_{\mathscr{U}},\trivlin_{f_{ij}},\mu_{ijk})$ and $(M_{\mathscr{V}},\trivlin_{h_{ij}},\nu_{ijk})$. Similarly to Lemma \ref{lem:goodcover} one can show that there exists a cover $\mathscr{W}$ of $M$ by open sets $W_i$ that refines both $\mathscr{U}$ and $\mathscr{V}$, and that allows smooth sections $\omega_i:W_i \to Z$ for which the $\Gamma$-bundle $\omega_i^{*}Q$ is trivializable. In the terminology of the above construction, choosing a trivialization $t: \omega^{*}Q \to \trivlin_{h_i}$ with smooth maps $h_i: W_i \to \Gamma_0$ over $M_{\mathscr{W}}$ converts the isomorphism $\alpha$ into smooth functions $s_{ij}: W_i \cap W_j \to \Gamma_1$ satisfying $s(s_{ij}) = g_{ij}' \cdot h_{i}$ and $t(s_{ij}) = h_{j} \cdot g_{ij}$. The compatibility diagram \erf{compgerbemorph} implies the remaining condition that makes $(h_{i},s_{ij})$ an equivalence between the $\Gamma$-2-cocycles $(f_{ij},g_{ijk})$ and $(f_{ij}',g_{ijk}')$.
\end{proof}

\section{Version IV: Principal 2-Bundles}
\label{sec:2bundle}

The basic idea of a smooth 2-bundle is that it gives for every point $x$ in the base manifold $M$ a Lie groupoid $\mathcal{P}_x$ varying smoothly with $x$. Numerous different versions have appeared so far in the literature, e.g  \cite{bartels,baez2,wockel1,pries2}. 
The main objective of \emph{our} version of principal $2$-bundles is to make the definition of the objects (i.e. the 2-bundles) \emph{as simple as possible}, while keeping their isomorphism classes in bijection with non-abelian cohomology.   Thus, our principal 2-bundles will be defined using \emph{strict} actions of Lie 2-groups on Lie groupoids, and \emph{not} using anafunctors. The necessary \quot{weakness} will be pushed into the definition of 1-morphisms.

\subsection{Definition of Principal 2-Bundles}

\label{sec:definitionzwoabundle}

As an important prerequisite for principal 2-bundles we have to discuss actions of Lie 2-groups on Lie groupoids, and equivariant anafunctors.

\begin{definition}
\label{def:rightaction}
Let $\mathcal{P}$ be a Lie groupoid, and let $\Gamma$ be a Lie 2-group. A \emph{smooth right action of $\Gamma$ on $\mathcal{P}$} is a smooth functor
$R: \mathcal{P} \times \Gamma \to \mathcal{P}$
such that $R(p,1)=p$ and $R(\rho,\id_1)=\rho$ for all $p\in \mathcal{P}_0$ and $\rho\in \mathcal{P}_1$,  and such that the diagram
\begin{equation*}
\alxydim{@=1.3cm}{\mathcal{P} \times \Gamma \times \Gamma \ar[r]^-{\id \times m} \ar[d]_{R \times \id} & \mathcal{P} \times \Gamma  \ar[d]^{R} \\ \mathcal{P} \times \Gamma \ar[r]_{m} & \mathcal{P}}
\end{equation*}
of smooth functors is commutative (strictly, on the nose).
\end{definition}

For example, every Lie 2-group acts on itself via multiplication. Note that due to strict commutativity, one has $R(R(p,g),g^{-1})=p$ and $R(R(\rho,\gamma),i(\gamma))=\rho$ for all $g\in \Gamma_0$, $p\in \mathcal{P}_0$, $\gamma\in \Gamma_1$ and $\rho\in \mathcal{P}_1$. 

\begin{remark}
This definition could be weakened  in two steps. First, one could allow a natural transformation in the above diagram instead of commutativity. Secondly, one could allow $R$ to be an anafunctor instead of an ordinary functor.  It turns out that for our purposes the above definition is sufficient. 
\end{remark}

\begin{definition}
\label{def:equivanafunctor}
Let $\mathcal{X}$ and $\mathcal{Y}$ be Lie groupoids with smooth actions $(R_1,\rho_1)$, $(R_2,\rho_2)$ of a Lie 2-group $\Gamma$. An \emph{equivariant structure} on an anafunctor $F: \mathcal{X} \to \mathcal{Y}$ is a transformation
\begin{equation*}
\alxydim{@=1.3cm}{\mathcal{X} \times \Gamma \ar[d]_{F \times\id} \ar[r]^-{R_1} & \mathcal{X} \ar@{=>}[dl]|*+{\lambda} \ar[d]^{F} \\ \mathcal{Y} \times \Gamma \ar[r]_-{R_2} & \mathcal{Y}}
\end{equation*}
satisfying the following condition:
\begin{equation*}
\alxydim{@R=0.6cm@C=0.2cm}{\mathcal{X} \times \Gamma \times \Gamma \ar[rr]^{\id \times m} \ar[dr]|{R_1 \times \id} \ar[dd]_{F \times \id\times\id} && \mathcal{X} \times \Gamma  \ar[dr]^{R_1} & \\ & \mathcal{X} \times \Gamma \ar@{=>}[dl]|{\lambda \times \id} \ar[rr]|{R_1} \ar[dd]|{F \times \id} && \mathcal{X} \ar@{=>}[ddll]|*+{\lambda} \ar[dd]^{F} \\ \mathcal{Y} \times \Gamma \times \Gamma \ar[dr]_{R_2 \times \id} &&& \\ & \mathcal{Y} \times \Gamma  \ar[rr]_{R_2} && \mathcal{Y}}
\qquad=\qquad
\alxydim{@R=0.6cm@C=0.2cm}{\mathcal{X} \times \Gamma \times \Gamma \ar[rr]^{\id \times m} \ar[dd]_{F \times \id \times \id} && \mathcal{X} \times \Gamma  \ar[dd]|{F \times \id} \ar[dr]^{R_1} \\ &&&\mathcal{X}\ar@{=>}[dl]|{\lambda} \ar[dd]^{F} \\ \mathcal{Y} \times \Gamma \times \Gamma \ar[dr]_{R_2 \times \id} \ar[rr]|{\id \times m} && \mathcal{Y} \times \Gamma \ar[dr]|{R_2} & \\ & \mathcal{Y} \times \Gamma \ar[rr]_{R_2} && \mathcal{Y}}
\end{equation*}
An anafunctor together with a $\Gamma$-equivariant structure is called \emph{$\Gamma$-equivariant anafunctor}. 
\end{definition}

In Appendix \ref{sec:equivariantanafunctorsandgroupactions} we translate this abstract (but evidently correct) definition of equivariance into more concrete terms involving a $\Gamma_1$-action on the total space of the anafunctor.

\begin{definition}
\label{def:equivtransformation}
If $(F,\lambda):\mathcal{X} \to \mathcal{Y}$ and $(G,\gamma): \mathcal{X} \to \mathcal{Y}$ are $\Gamma$-equivariant anafunctors, a transformation $\eta: F \Rightarrow G$ is called \emph{$\Gamma$-equivariant}, if the following equality of transformation holds:
\begin{equation*}
\alxydim{@R=1.3cm@C=2cm}{\mathcal{X} \ar@/_1.5pc/[d]_<<<<<<{G \times \id}="1" \times \Gamma \ar@/^1.5pc/[d]^<<<<<<{F \times\id}="2" \ar@{=>}"2";"1"|{\eta \times \id} \ar[r]^-{R_1} & \mathcal{X} \ar@{=>}[dl]|>>>>>>>>>>*+{\lambda} \ar[d]^{F} \\ \mathcal{Y} \times \Gamma \ar[r]_-{R_2} & \mathcal{Y}}
\quad=\quad 
\alxydim{@R=1.5cm@C=2cm}{\mathcal{X} \times \Gamma \ar[d]_{G \times\id} \ar[r]^-{R_1} & \mathcal{X} \ar@{=>}[dl]|*+{\gamma} \ar@/_1.5pc/[d]_{G}="1" \ar@/^1.5pc/[d]^{F}="2" \ar@{=>}"2";"1"|*+{\eta} \\ \mathcal{Y} \times \Gamma \ar[r]_-{R_2} & \mathcal{Y}}
\end{equation*}
\end{definition}

It follows from abstract nonsense in the bicategory of Lie groupoids, anafunctors and transformations that we have another bicategory with
\begin{itemize}
\item 
objects: Lie groupoids with smooth right $\Gamma$-actions.

\item
1-morphisms: $\Gamma$-equivariant anafunctors.

\item
2-morphisms: $\Gamma$-equivariant transformations.
\end{itemize}

We need three further notions for the definition of a principal 2-bundle. Let $M$ be a smooth manifold, and let
 $\mathcal{P}$ be a Lie groupoid. We say that a smooth functor $\pi: \mathcal{P} \to \idmorph M$ is a \emph{surjective submersion functor}, if $\pi: \mathcal{P}_0 \to M$ is a surjective submersion. 
Let $\pi:\mathcal{P} \to \idmorph M$ be a surjective submersion functor, and let $\mathcal{Q}$ be a Lie groupoid with some smooth functor $\chi:\mathcal{Q} \to \idmorph M$. Then, the fibre product $\mathcal{P} \times_M \mathcal{Q}$ is defined to be the full subcategory of $\mathcal{P} \times \mathcal{Q}$ over the submanifold $\mathcal{P}_0 \times_M \mathcal{Q}_0 \subset \mathcal{P}_0 \times \mathcal{Q}_0$.

\begin{definition}
\label{def:zwoabun}
Let $M$ be a smooth manifold and let $\Gamma$ be a Lie 2-group.
\begin{enumerate}[(a)]

\item 
A \emph{principal $\Gamma$-2-bundle over $M$} is a Lie groupoid $\mathcal{P}$, a surjective submersion  functor $\pi: \mathcal{P} \to \idmorph M$,  and a smooth right action $R$ of $\Gamma$ on $\mathcal{P}$ that preserves  $\pi$,
such that the smooth functor
\begin{equation*}
\tau :=(\mathrm{pr}_1, R) : \mathcal{P} \times \Gamma \to \mathcal{P} \times_M \mathcal{P} \end{equation*}
is a weak equivalence.

\item
A \emph{1-morphism} between principal $\Gamma$-2-bundles is a  $\Gamma$-equivariant anafunctor 
\begin{equation*}
F: \mathcal{P}_1 \to \mathcal{P}_2
\end{equation*}
that respects the surjective submersion functors to $M$.
 
\item 
A \emph{2-morphism} between 1-morphisms is a $\Gamma$-equivariant transformation between these.
\end{enumerate}
\end{definition}

\begin{remark}
\label{rem:preserving}
\begin{enumerate}[(a)]
\item 
\label{rem:preserving:a}
The condition in (a) that the action $R$ preserves
the surjective submersion functor $\pi$ means that the diagram of functors
\begin{equation*}
\alxydim{@R=1.3cm}{\mathcal{P} \times \Gamma \ar[r]^{R} \ar[d]_{\mathrm{pr}_1} & \mathcal{P} \ar[d]^{\pi} \\ \mathcal{P} \ar[r]_-{\pi} &\idmorph M&}
\end{equation*} 
is commutative. 

\item
\label{rem:preserving:b}
The condition in (b) that the anafunctor $F$ respects the surjective submersion functors means in the first place
that there exists a transformation
\begin{equation*}
\alxydim{@R=1.3cm}{\mathcal{P}_1 \ar[rr]^{F} \ar[dr]_{\pi_1}="1" && \mathcal{P}_2 \ar@{=>}"1"\ar[dl]^{\pi_2} \\ &\idmorph M\text{.}&}
\end{equation*} 
However, since the target of the anafunctors $\pi_1$ and $\pi_2 \circ F$ is the \emph{discrete} groupoid $\idmorph{M}$, the equivalence of Example \ref{anafunctors:c} applies, and implies that if such a transformation exists, it is unique. Indeed, it is easy to see that an  anafunctor $F: \mathcal{P} \to \mathcal{Q}$ with anchors  $\alpha_l: F \to \mathcal{P}_0$ and $\alpha_r: F \to \mathcal{Q}_0$ 
respects smooth functors $\pi: \mathcal{P} \to \idmorph{M}$ and $\chi: \mathcal{Q} \to \idmorph{M}$ if and only if
$\pi \circ \alpha_l = \chi \circ \alpha_r$.
\end{enumerate}
\end{remark}

\begin{example}
\label{ex:trivzwoabun}
The \emph{trivial  $\Gamma$-2-bundle} over $M$ is defined by
\begin{equation*}
\mathcal{P} := \idmorph M \times \Gamma
\quomma
\pi:= \mathrm{pr}_1 
\quomma 
R := \id_M \times  m
\text{.}
\end{equation*}
Here, the smooth functor $\tau$  even has a smooth inverse \emph{functor}. In the following we denote the trivial $\Gamma$-2-bundle by $\mathcal{I}$.
\end{example}

\begin{remark}
\label{rem:ginot}
The principal $\Gamma$-2-bundles of Definition \ref{def:zwoabun} are very similar to those of Bartels \cite{bartels} and Wockel \cite{wockel1}, in the sense that their fibres are groupoids with a $\Gamma$-action. They only differ in the strictness assumptions for the action, and in the formulation of principality. Opposed to that, the \emph{principal 2-group bundles} introduced in \cite{ginot1} are quite different: their fibres are Lie 2-groupoids equipped with a certain Lie 2-groupoid morphism to $B\Gamma$. 
\end{remark}

\subsection{Properties of Principal 2-Bundles}

Principal $\Gamma$-2-bundles over $M$ form a bicategory denoted $\zwoabun\Gamma M$. There is an evident pullback 2-functor 
\begin{equation*}
f^{*}: \zwoabun\Gamma N \to \zwoabun\Gamma M
\end{equation*}
associated to smooth maps $f:M \to N$,
and these make $\zwoabun\Gamma -$ a pre-2-stack over smooth manifolds. We deduce the following important two theorems about this pre-2-stack. The first asserts that it actually is a 2-stack:

\begin{theorem}
\label{th:2stack}
Principal 
$\Gamma$-2-bundles form a  2-stack $\zwoabun\Gamma-$ over smooth manifolds. \end{theorem}

\begin{proof}
This follows from Theorem \ref{th:stack} ($\Gamma$-bundle gerbes form a 2-stack) and Theorem \ref{th:equivalence} (the equivalence $\grb\Gamma- \cong \zwoabun \Gamma-$)  we prove in Section \ref{sec:equivalences}.
\end{proof}

\begin{remark}
\label{rem:2bungrp}
Similar to Remark \ref{rem:grbgrp}, we obtain automatically bicategories $\zwoabun\Gamma{\mathcal{X}}$ of principal $\Gamma$-2-bundles over Lie groupoids $\mathcal{X}$, including bicategories of \emph{equivariant principal $\Gamma$-2-bundles}.
\end{remark}

The second concerns a homomorphism  $\Lambda: \Gamma \to \Omega$ of Lie 2-groups, which induces the extension  $\Lambda \maps \grb\Gamma- \to \grb\Omega-$ between 2-stacks of bundle gerbes (Proposition \ref{grbextension}). Combined with the equivalence $\grb\Gamma- \cong \zwoabun \Gamma-$ of Theorem \ref{th:equivalence}, it defines a 1-morphism
\begin{equation*}
\Lambda : \zwoabun\Gamma- \to \zwoabun\Omega-
\end{equation*}
between 2-stacks of principal 2-bundles. Now we get as a direct consequence of Theorem \ref{extequiv}:

\begin{theorem}
\label{extequivbun}
If $\Lambda: \Gamma \to \Omega$ is a weak equivalence between Lie 2-groups, then the 1-morphism $\Lambda \maps \zwoabun\Gamma- \to \zwoabun\Omega-$ is an equivalence of 2-stacks.
\end{theorem}

A third consequence of the equivalence of Theorem \ref{th:equivalence} in combination with Lemma \ref{lem:invertibility} is
\begin{corollary}
\label{co:bunmorphinvertible}
Every 1-morphism $F: \mathcal{P}_1 \to \mathcal{P}_2$ between principal $\Gamma$-2-bundles over $M$ is invertible. 
\end{corollary}

The following discussion centers around \emph{local trivializability} that is implicitly contained in Definition \ref{def:zwoabun}.
A principal  $\Gamma$-2-bundle that is isomorphic to the trivial $\Gamma$-2-bundle $\mathcal{I}$ introduced in Example  \ref{ex:trivzwoabun} is called \emph{trivializable}. A \emph{section} of a principal $\Gamma$-2-bundle $\mathcal{P}$ over $M$ is an anafunctor $S: \idmorph M \to \mathcal{P}$ such that $\pi \nobr\circ\nobr S \eq \id_{\idmorph M}$ (recall that an anafunctor $\pi \circ S:M \to M$ is the same as a smooth map). One can show that every point $x\in M$  has an open neighborhood $U$ together with a section $s: \idmorph U \to \mathcal{P}|_U$. Such sections  can even be chosen to be smooth functors, rather than anafunctors, namely simply as ordinary sections of the surjective submersion $\pi: (\mathcal{P}|_U)_0 \to \idmorph U$.

\begin{lemma}
\label{lem:loctriv2}
A principal $\Gamma$-2-bundle over $M$ is trivializable if and only if it has a smooth section.
\end{lemma}

\begin{proof}
The trivial $\Gamma$-2-bundle $\mathcal{I}$ has the section $S(m) := (m,1)$, where $1$ denotes the unit of $\Gamma_0$. If $\mathcal{P}$ is trivializable, and $F: \mathcal{I} \to \mathcal{P}$ is an isomorphism, then, $F \circ S$ is a section of $\mathcal{P}$. Conversely, suppose $\mathcal{P}$ has a section $S: \idmorph M \to \mathcal{P}$. Then, we get the anafunctor
\begin{equation}
\label{seccomp}
\alxydim{@C=1.5cm}{\mathcal{I}=\idmorph M \times \Gamma \ar[r]^-{S\times \id} & \mathcal{P}\times \Gamma \ar[r]^-{R} & \mathcal{P}\text{.}}
\end{equation}
It has an evident  $\Gamma$-equivariant structure and respects the projections to $M$. According to Corollary \ref{co:bunmorphinvertible}, this is sufficient to have a 1-isomorphism.
\end{proof}

\begin{corollary}
Every principal $\Gamma$-2-bundle is locally trivializable, i.e. every point $x\in M$ has an open neighborhood $U$ and a 1-morphism $T: \mathcal{I} \to \mathcal{P}|_U$.
\end{corollary}

\begin{remark}
In Wockel's version \cite{wockel1} of principal 2-bundles, local trivializations are required to be smooth functors and to be invertible as smooth functors, rather than allowing anafunctors. This version turns out to be too restrictive in the sense that the resulting bicategory receives no 2-functor from the bicategory $\grb\Gamma M$ of $\Gamma$-bundle gerbes that would establish an equivalence. 
\end{remark}

It is also possible to reformulate our definition of principal 2-bundles in terms of local trivializations. This reformulation gives us criteria which might be easier to check than the actual definition, similar to the case of ordinary principal bundles.

\begin{proposition}
Let $\mathcal{P}$ be a Lie groupoid,  $\pi: \mathcal{P} \to \idmorph M$ be a surjective submersion  functor,  and  $R$ be a smooth right action of $\Gamma$ on $\mathcal{P}$ that preserves $\pi$. Suppose every point $x\in M$ has an open neighborhood $U$ together with a $\Gamma$-equivariant anafunctor $T: \mathcal{I} \to \mathcal{P}|_U$ that respects the projections. Then, $\pi:\mathcal{P} \to \idmorph M$ is a principal $\Gamma$-2-bundle over $M$.
\end{proposition}

\begin{proof}
We only have to prove that the functor $\tau$ is a weak equivalence, and we use  Theorem \ref{weak}. Since all morphisms of $\mathcal{P}$ have source and target in the same fibre of $\pi: \mathcal{P}_0 \to \idmorph M$, we may check the two conditions of Theorem \ref{weak} locally, i.e. for $\mathcal{P}|_{U_i}$ where $U_i$ is an open cover of $M$. Using local trivializations $\mathcal{T}_i: \mathcal{I} \to \mathcal{P}|_{U_i}$, the smooth functor $\tau$ translates into the smooth functor $(\id,\mathrm{pr}_1,m) :\idmorph M \times \Gamma \times \Gamma \to (\idmorph M \times \Gamma) \times_{M} (\idmorph M \times \Gamma)$. This functor is an isomorphism of Lie groupoids, and hence essentially surjective and fully faithful.  
\end{proof}

\setsecnumdepth{1}

\section{Equivalence between Bundle Gerbes and 2-Bundles}

\label{sec:equivalences}

In this section we show that Versions III and IV of smooth $\Gamma$-gerbes are equivalent in the strongest possible sense:

\begin{theorem}
\label{th:equivalence}
For $M$ a smooth manifold and $\Gamma$ a Lie 2-group, there is an equivalence of bicategories
\begin{equation*}
\grb\Gamma M \cong \zwoabun \Gamma M
\end{equation*}
between the bicategories of $\Gamma$-bundle gerbes and principal $\Gamma$-2-bundles over $M$.
This equivalence is natural in $M$, i.e. it is an  equivalence between pre-2-stacks.
\end{theorem}

Since the definitions of the bicategories  $\grb\Gamma M$ and $\zwoabun \Gamma M$, \emph{and} the above equivalence are all natural in $M$, we obtain automatically an induced equivalence for the induced bicategories over Lie groupoids (see Remarks \ref{rem:grbgrp} and \ref{rem:2bungrp}).

\begin{corollary}
For $\mathcal{X}$ a Lie groupoid and $\Gamma$ a Lie 2-group, there is an equivalence
\begin{equation*}
\grb\Gamma{\mathcal{X}} \cong \zwoabun\Gamma{\mathcal{X}}\text{.}
\end{equation*}
\end{corollary}

The following outlines the proof of Theorem \ref{th:equivalence}.  
In Section \ref{sec:extraction} we construct explicitly a 2-functor 
\begin{equation*}
\mathscr{E}_M: \zwoabun \Gamma M \to \grb\Gamma M\text{.}
\end{equation*}
Then we use a general criterion  assuring that $\mathscr{E}_M$ is an equivalence of bicategories. This criterion is stated in Lemma \ref{lem:equivalence}: it requires (A) that $\mathscr{E}_M$ is fully faithful on Hom-categories, and (B) to choose certain preimages of objects and 1-morphisms under $\mathscr{E}_M$. Under these circumstances, Lemma \ref{lem:equivalence} constructs an inverse 2-functor $\mathscr{R}_M$ together with the required pseudonatural transformations assuring that $\mathscr{E}_M$ and $\mathscr{R}_M$  form an equivalence of bicategories. Condition (A) is proved as Lemma \ref{lem:efullyfaithful} in Section \ref{sec:extraction}. The choices (B) are constructed in Section \ref{sec:reconstruction}.

In order to prove that the 2-functors $\mathscr{E}_M$ extend to the claimed equivalence between pre-2-stacks, we use another criterion stated in Lemma \ref{lem:stackequivalence}. The only additional assumption of Lemma \ref{lem:stackequivalence} is that the given 2-functors $\mathscr{E}_M$ form a 1-morphism of pre-2-stacks; this is proved in Proposition \ref{prop:estack}. Then, the inverse 2-functors $\mathscr{R}_M$ obtained before automatically form an inverse 1-morphism between pre-2-stacks.

\setsecnumdepth{2}

\subsection{From Principal 2-Bundles to Bundle Gerbes}

\label{sec:extraction}

In this section we define the 2-functor
$\mathscr{E}_M: \zwoabun \Gamma M \to \grb\Gamma M$.

\subsubsection*{Definition of $\mathscr{E}_M$ on objects}

Let $\mathcal{P}$ be a principal $\Gamma$-2-bundle over $M$, with projection $\pi: \mathcal{P} \to M$ and right action $R$ of $\Gamma$ on $\mathcal{P}$. The first ingredient of the $\Gamma$-bundle gerbe $\mathscr{E}_M(\mathcal{P})$ is the surjective submersion $\pi: \mathcal{P}_0 \to M$. The second ingredient is a principal $\Gamma$-bundle $P$ over $\mathcal{P}_0^{[2]}$. We put
\begin{equation*}
P := \mathcal{P}_1 \times \Gamma_0\text{.}
\end{equation*}
Bundle projection, anchor and $\Gamma$-action are given, respectively, by \begin{multline}
\label{defgerbefrombundle}
\chi(\rho,g) := (t(\rho),R(s(\rho),g^{-1}))
\quomma
\alpha(\rho,g) := g
\quand
(\rho,g) \circ \gamma := (R(\rho,\id_{g^{-1}} \cdot \gamma),s(\gamma))\text{.}
\end{multline}
These definitions are motivated by Remark \ref{rem:motivation} below. 

\begin{lemma}
\label{lem:bundleforgerbe}
This defines a principal $\Gamma$-bundle over $\mathcal{P}_0^{[2]}$.
\end{lemma}

\begin{proof}
First we check that $\chi: P \to \mathcal{P}_0^{[2]}$ is a surjective submersion. Since the functor $\tau = (\id,R)$ is a weak equivalence, we know from Theorem \ref{weak} that
\begin{equation*}
f:(\mathcal{P}_0 \times \Gamma_0) \lli{\tau} \times_{t \times t} \mathcal{P}_1^{[2]} \to \mathcal{P}_0^{[2]} : (p,g,\rho_1,\rho_2) \mapsto (s(\rho_1),s(\rho_2))
\end{equation*} 
is a surjective submersion. Now consider the smooth surjective map
\begin{equation*}
g: (\mathcal{P}_0 \times \Gamma_0) \lli{\tau} \times_{t \times t} \mathcal{P}_1^{[2]} \to \mathcal{P}_1 \times \Gamma_0  : (p,g,\rho_1,\rho_2) \mapsto (\rho_1^{-1} \circ R(\rho_2,\id_{g^{-1}}),g^{-1})\text{.}
\end{equation*} 
We have $\chi \circ g = f$; thus, $\chi$ is a surjective submersion.
Next we check that we have defined an action. Suppose $(\rho,g) \in P$ and $\gamma\in \Gamma_1$ such that $\alpha(\rho,g) = g = t(\gamma)$. Then, $\alpha((\rho,g) \circ \gamma) =  s(\gamma)$. Moreover, suppose $\gamma_1,\gamma_2\in \Gamma_1$ with $t(\gamma_1)= g$ and $t(\gamma_2)=s(\gamma_1)$. Then,
\begin{multline*}
((\rho,g) \circ \gamma_1) \circ \gamma_2 
=
(R(\rho,\id_{g^{-1}} \cdot \gamma_1),s(\gamma_1)) \circ \gamma_2
=
(R(\rho,\id_{g^{-1}} \cdot \gamma_1 \cdot \id_{s(\gamma_1)^{-1}} \cdot \gamma_2),s(\gamma_2))
= (\rho,g) \circ (\gamma_1 \circ \gamma_2)\text{,}
\end{multline*}
where we have used that $\gamma_1 \circ \gamma_2 = \gamma_1 \cdot \id_{s(\gamma_1)^{-1}} \cdot \gamma_2$ in any 2-group. 
It remains to check that the smooth map
\begin{equation*}
\tilde\tau: P \lli{\alpha}\times_{t} \Gamma_1 \to P \lli{\chi}\times_{\chi} P : ((\rho,g),\gamma) \mapsto ((\rho,g),(\rho,g)\circ \gamma)
\end{equation*} 
is a diffeomorphism. For this purpose, we consider the diagram
\begin{equation}
\label{eq:pullback}
\alxydim{@C=1.5cm@R=1.3cm}{ & \mathcal{P}_1^{[2]} \ar[d]^{s \times t} \\ (\mathcal{P}_0 \times \Gamma_0) \times (\mathcal{P}_0 \times \Gamma_0) \ar[r]_-{\tau \times \tau} & \mathcal{P}_0^{[2]} \times \mathcal{P}_0^{[2]}}
\end{equation}
and claim that (a)  $N_1 :=P \lli{\alpha} \times_{t}\Gamma_1$ is a pullback of \erf{eq:pullback}, (b) $N_2 :=P \lli{\chi} \times_{\chi} P$ is a pullback of \erf{eq:pullback}, and (c)  the unique map $N_1 \to N_2$ is $\tilde\tau$. Thus, $\tilde \tau$ is a diffeomorphism.

In order to prove  claim (a)
we use again that the functor $\tau = (\id, R)$ is  a weak equivalence, so that by Theorem \ref{weak} the triple $(\mathcal{P}_1 \times \Gamma_1,\tau,s \times t)$ is a pullback of \erf{eq:pullback}. 
We consider the smooth map 
\begin{equation*}
\xi: N_{1} \to \mathcal{P}_1 \times \Gamma_1:((\rho,g),\gamma) \mapsto (R(\rho,\id_{g^{-1}}),\gamma)
\end{equation*}
which is a diffeomorphism because  $(\rho,\gamma) \mapsto ((R(\rho,\id_{t(\gamma)}),t(\gamma)),\gamma)$ is a smooth map which is  inverse to $\xi$. 
Thus, putting  $f_1 := \tau \circ \xi$ and $g_1 := (s \times t) \circ \xi$ we see that $(N_1,f_1,g_1)$ is a pullback of \erf{eq:pullback}. In order to prove claim (b), we put
\begin{eqnarray*}
f_2((\rho_1,g_1),(\rho_2,g_2)) &:=& (R(\rho_{1},\id_{g_1^{-1}}),\rho_2)
\\
g_2((\rho_1,g_1),(\rho_2,g_2)) &:=& (R(s(\rho),g_1^{-1}),g_2,R(t(\rho_1),g_1^{-1}),g_1)\text{,}
\end{eqnarray*}
and it is straightforward to check that the cone $(N_2,f_2,g_2)$ makes \erf{eq:pullback} commutative. 
The triple $(N_2,f_2,g_2)$ is also universal: in order to see this suppose $N'$ is any smooth manifold with smooth maps $f': N' \to \mathcal{P}_1^{[2]}$ and $g': N' \to (\mathcal{P}_0 \times \Gamma_0) \times (\mathcal{P}_0 \times \Gamma_0)$ so that \erf{eq:pullback} is commutative. For $n\in N'$, we write $f'(n) = (\rho_1,\rho_2)$ and $g'(n)=(p_1,g_1,p_2,g_2)$. Then, $\sigma(n) := ((R(\rho_1,\id_{g_2^{-1}}),g_2),(\rho_2,g_1))$ defines a smooth map $\sigma:N' \to P \lli{\chi} \times_{\chi} P$. One checks that $f_2 \circ \sigma = f'$ and $g_2 \circ \sigma = g'$, and that $\sigma$ is the only smooth map satisfying these equations. This proves that $(N_2,f_2,g_2)$ is a pullback. We are left with claim (c). Here one only has to check that $\tau: N_1 \to N_2$ satisfies $f_2=f_1 \circ \tau$ and $g_2 = g_1 \circ \tau$. 
\end{proof}

\begin{remark}
\label{rem:motivation}
The smooth functor $\tau=(\id,R): \mathcal{P} \times \Gamma \to \mathcal{P} \times_M \mathcal{P}$ is a weak equivalence, and so has a canonical inverse anafunctor $\tau^{-1}$ (Remark \ref{rem:inverses}). The anafunctor
\begin{equation*}
\alxydim{}{\mathcal{P}_0^{[2]}  \ar[r]^-{\iota} & \mathcal{P} \times_M \mathcal{P} \ar[r]^-{c} & \mathcal{P} \times_M \mathcal{P}\ar[r]^-{\tau^{-1}} & \mathcal{P} \times \Gamma \ar[r]^-{\mathrm{pr}_2} &  \Gamma\text{,}}
\end{equation*}
where $c$ is the functor that switches the factors,
corresponds to a principal $\Gamma$-bundle over $\mathcal{P}_0^{[2]}$ that is canonically isomorphic to the bundle $P$ defined above. 
\end{remark}

It remains to  provide the bundle gerbe  product 
\begin{equation*}
\mu: \pi_{23}^{*}P \otimes \pi_{12}^{*}P \to \pi_{13}^{*}P\text{,}
\end{equation*}
which we define by the formula
\begin{equation}
\label{defmu}
\mu((\rho_{23},g_{23}),(\rho_{12},g_{12})) := (\rho_{12} \circ R(\rho_{23},\id_{g_{12}}) ,g_{23}g_{12})\text{.}
\end{equation}

\begin{lemma}
Formula \erf{defmu} defines an associative isomorphism $\mu: \pi_{23}^{*}P \otimes \pi_{12}^{*}P \to \pi_{13}^{*}P$ of principal $\Gamma$-bundles over $\mathcal{P}_0^{[3]}$. \end{lemma}

\begin{proof}
First of all, we recall from Example \ref{ex:tensorproduct} \erf{ex:tensorproduct:general} that an element in the tensor product $\pi_{23}^{*}P \otimes \pi_{12}^{*}P$ is  represented by a triple $(p_{23},p_{12},\gamma)$ where $p_{23},p_{12} \in P$ with $\pi_1(\chi(p_{23})) = \pi_2(\chi(p_{12}))$, and $\alpha(p_{23}) \cdot \alpha(p_{12})=t(\gamma)$.
In \erf{defmu} we refer to triples where $\gamma = \id_{g_{23}g_{12}}$, and this definition extends to triples with general $\gamma \in \Gamma_1$ by employing the equivalence relation
\begin{equation}
\label{equivreltp}
(p_1,p_2,\gamma) \sim  (p_1 \circ (\gamma \cdot \id_{\alpha(p_2)^{-1}}),p_2,\id_{s(\gamma)})\text{.} \end{equation}
The complete formula for $\mu$ is then
\begin{equation}
\label{defmufull}
\mu((\rho_{23},g_{23}),(\rho_{12},g_{12}),\gamma) = (\rho_{12} \circ R(\rho_{23},\id_{g_{23}^{-1}} \cdot \gamma)  ,s(\gamma))\text{.}
\end{equation}
Next we check that \erf{defmufull} is well-defined under the equivalence relation \erf{equivreltp}: 
\begin{eqnarray*}
\mu(((\rho_{23},g_{23}),(\rho_{12},g_{12}),\gamma)) 
&=& (\rho_{12} \circ R(\rho_{23},\id_{g_{23}^{-1}} \cdot \gamma)  ,s(\gamma))
\\&=&   (\rho_{12} \circ R(\rho_{23} \circ R(\id_{R(s(\rho_{23}),g_{23}^{-1})},\gamma \cdot \id_{g_{12}^{-1}}),\id_{g_{12}} )  ,s(\gamma))
\\&=& \mu((\rho_{23} \circ R(\id_{R(s(\rho_{23}),g_{23}^{-1})},\gamma \cdot \id_{g_{12}^{-1}}),s(\gamma)g_{12}^{-1}),(\rho_{12},g_{12}),\id_{s(\gamma)}))
\\&=& \mu(((\rho_{23},g_{23}) \circ (\gamma \cdot \id_{g_{12}^{-1}}),(\rho_{12},g_{12}),\id_{s(\gamma)}))\text{.}
\end{eqnarray*}
Now we have shown that $\mu$ is a well-defined map from $\pi_{23}^{*}P \otimes \pi_{12}^{*}P$ to $\pi_{13}^{*}P$, and it remains to prove that it is a bundle morphism. Checking that it preserves fibres and anchors is straightforward.
It remains to check that \erf{defmufull} preserves the  $\Gamma$-action. We calculate
\begin{eqnarray*}
\mu(((\rho_{23},g_{23}),(\rho_{12},g_{12}),\gamma) \circ \tilde\gamma)
&=& \mu((\rho_{23},g_{23}),(\rho_{12},g_{12}),\gamma \circ \tilde\gamma)
\\&=& (\rho_{23} \circ R(\rho_{12},  \id_{g_{12}} \cdot i(\gamma \circ \tilde\gamma)) ,s( \tilde\gamma) )
\\&=& (\rho_{23} \circ R(R(\rho_{12},  \id_{g_{12}}), i(\gamma)\circ i(\tilde\gamma)),s(\tilde\gamma))
\\&=&(\rho_{23} \circ R(R(\rho_{12},  \id_{g_{12}}), i(\gamma)) \circ R(\id_{R(s(\rho_{12}),  g)},i(\tilde\gamma)),s(\tilde\gamma))
\\&=& (\rho_{23} \circ R(\rho_{12},  \id_{g_{12}} \cdot i(\gamma)) \circ R(\id_{R(s(\rho_{12}),  g)},i(\tilde\gamma)),s(\tilde\gamma))
\\&=& (\rho_{23} \circ R(\rho_{12},  \id_{g_{12}} \cdot i(\gamma)) ,s(\gamma) )\circ \tilde\gamma
\\&=& \mu((\rho_{23},g_{23}),(\rho_{12},g_{12}),\gamma)  \circ \tilde\gamma\text{.}
\end{eqnarray*}
Summarizing, $\mu$ is a morphism of $\Gamma$-bundles over $\mathcal{P}_0^{[3]}$. The associativity of $\mu$ follows directly from the definitions.
\end{proof}

\subsubsection*{Definition of $\mathscr{E}_M$ on 1-morphisms}

We define a 1-morphism $\mathscr{E}_M(F): \mathscr{E}_M(\mathcal{P}) \to \mathscr{E}_M(\mathcal{P}')$ between $\Gamma$-bundle gerbes from a 1-morphism $F: \mathcal{P} \to \mathcal{P}'$ between principal $\Gamma$-2-bundles. The refinement of the surjective submersions $\pi: \mathcal{P} \to M$ and $\pi': \mathcal{P}' \to M$ is the fibre product $Z := \mathcal{P}_0 \times_M \mathcal{P}_0'$. Its principal $\Gamma$-bundle has the total space
\begin{equation*}
Q := F \times \Gamma_0\text{,}
\end{equation*}
and its projection, anchor and $\Gamma$-action are given, respectively, by
\begin{multline}
\label{eq:emorph}
\chi(f,g) := (\alpha_l(f), R(\alpha_r(f),g^{-1}))
\quomma
\alpha(f,g) := g
\quand
(f,g) \circ \gamma := (\rho(f,\id_{g^{-1}}\cdot \gamma),s(\gamma))\text{,}
\end{multline}
where $\rho: F \times \Gamma_1 \to F$ denotes the $\Gamma_1$-action on $F$ that comes from the given $\Gamma$-equivariant structure on $F$ (see Appendix \ref{sec:equivariantanafunctorsandgroupactions}).

\begin{lemma}
This defines a principal $\Gamma$-bundle $Q$ over $Z$. 
\end{lemma}

\begin{proof}
We show first the the projection $\chi:Q \to Z$ is a surjective submersion. Since the  functor $\tau'\maps \mathcal{P}' \nobr\times\nobr \Gamma \to \mathcal{P} \times_M \mathcal{P}$ is a weak equivalence, we have by Theorem \ref{weak} a pullback 
\begin{equation*}
\alxydim{@R=1.3cm}{X \ar[r] \ar[d]_{\xi} & (\mathcal{P}_0' \times \Gamma_0) \lli{R}\times_{t} (\mathcal{P}_1' \times_M \mathcal{P}_1') \ar[d]^{s \circ \mathrm{pr}_2} \\ F \lli{\pi' \circ \alpha_l(f)}\times_{\pi'} \mathcal{P}_0' \ar[r] & \mathcal{P}_0' \times_M \mathcal{P}_0' }
\end{equation*}
along the bottom map $(f,p') \mapsto (\alpha_r(f),p')$, which is well-defined because the anafunctor $F$ preserves the projections to $M$ (see Remark \ref{rem:preserving} \erf{rem:preserving:b}). In particular, the map $\xi$ is a surjective submersion. It is easy to see that the smooth map
\begin{equation*}
k: X \to F \times \Gamma_0: ((f,p'),(p_0',g,\rho,\tilde\rho)) \mapsto (f \circ \rho^{-1} \circ R(\tilde\rho,\id_{g^{-1}}),g^{-1})
\end{equation*}
is surjective.
Now we consider the commutative diagram
\begin{equation*}
\alxydim{@C=1.5cm@R=1.3cm}{X \ar[d]_{\xi} \ar[r]^-{k} & F \times \Gamma_0 \ar[d]^{\chi} \\  F \lli{\pi' \circ \alpha_l(f)}\times_{\pi'} \mathcal{P}_0 \ar[r]_-{\alpha_l \times \id} & \mathcal{P}_0 \times_M \mathcal{P}_0'\text{.}}
\end{equation*}
The surjectivity of $k$ and the fact that $\xi$ and $\alpha_l \times \id$ are surjective submersions shows that $\chi$ is one, too.

Next, one checks (as in the proof of Lemma \ref{lem:bundleforgerbe}) that the $\Gamma$-action on $Q$ defined above is well-defined and preserves the projection.
Then it remains to check that the smooth map
\begin{equation*}
\xi \;:\;\; Q \lli{\alpha}\times_t \Gamma_1 \to Q \times_{\mathcal{P}_0 \times_M \mathcal{P}_0'} Q:(f,g,\gamma) \mapsto (f,g,\rho(f,\id_{g^{-1}}\cdot \gamma),s(\gamma))
\end{equation*}
is a diffeomorphism. An inverse map is given as follows. Given an element $(f_1,g_1,f_2,g_2)$ on the right hand side, we have $\alpha_l(f_1) = \alpha_l(f_2)$, so that there exists a unique element $\rho' \in \mathcal{P}_1'$ such that $f_1 \circ \rho' = f_2$. One calculates that $(\rho',g_2)$ and $(\id_{\alpha_r(f_1)},g_1)$ are elements of the principal $\Gamma$-bundle $\mathcal{P}' \times \Gamma_0$ over $\mathcal{P}_0'^{[2]}$   of Lemma \ref{lem:bundleforgerbe}. Thus, there exists a unique element $\gamma \in \Gamma_1$ such that $(\rho',g_2)=(\id_{\alpha_r(f_1)},g_1) \circ \gamma$. Clearly, $t(\gamma)=g_1$ and $s(\gamma)=g_2$, and we have $\rho' = R(\id_{\alpha_r(f_1)},\id_{g_1^{-1}} \cdot \gamma)$. We define $\xi^{-1}(f_1,g_1,f_2,g_2) := (f_1,g_1,\gamma)$. The calculation that $\xi^{-1}$ is an inverse for $\xi$ uses property (ii) of Definition \ref{def:actionanafunctor} for the action $\rho$, and is left to the reader.
\end{proof}

The next step in the definition of the 1-morphism $\mathscr{E}(F)$ is to define the bundle morphism
\begin{equation*}
\beta: P' \otimes \zeta_1^{*}Q  \to \zeta_2^{*}Q \otimes P
\end{equation*}
over $Z \times_M Z$. We use the notation of Example \ref{ex:tensorproduct} \erf{ex:tensorproduct:general} for elements of tensor pro\-ducts of principal $\Gamma$-bundles; in this notation, the morphism $\beta$ in the fibre over a point $((p_1,p_1'),(p_2,p_2')) \in Z \times_M Z$ is given by
\begin{equation*}
\beta:((\rho',g'),(f,g),\gamma) \mapsto ((\tilde f,g'gh),(\tilde \rho, h^{-1}),\gamma)\text{,}
\end{equation*}
where $h\in \Gamma_0$ and $\tilde\rho\in \mathcal{P}_1'$ are chosen such that $s(\tilde\rho) = R(p_2,h^{-1})$ and $t(\tilde\rho) = p_1$, and \begin{equation}
\label{eq:ftilde}
\tilde f := \rho(\tilde\rho^{-1} \circ f \circ R(\rho',\id_g),\id_h)\text{.}
\end{equation} 

\begin{lemma}
\label{lem:eonemorphisms}
This defines an isomorphism between principal $\Gamma$-bundles. 
\end{lemma}

\begin{proof}
The existence of  choices of $\tilde \rho,h$ follows because the  functor $\tau': \mathcal{P}' \times\Gamma \to \mathcal{P}' \times_M \mathcal{P}'$ is smoothly essentially surjective (Theorem \ref{weak}); in particular, one can choose them locally in a smooth way. We claim that the equivalence relation on $\zeta_2^{*}Q \otimes P$ identifies different choices; thus, we have a well-defined smooth map. In order to prove this claim, we assume other choices $\tilde\rho',h'$. The pairs $(\tilde\rho,h^{-1})$ and $(\tilde\rho',h'^{-1})$ are elements in the principal $\Gamma$-bundle $P'$ over $\mathcal{P}_0' \times_M \mathcal{P}_0'$ and sit over the same fibre; thus, there exists a unique $\tilde\gamma\in \Gamma_1$ such that $(\tilde\rho,h^{-1}) \circ \tilde\gamma = (\tilde\rho',h'^{-1})$, in particular, $R(\tilde\rho,\id_{h}\cdot \tilde\gamma) = \tilde\rho'$. Now we have
\begin{multline*}
((\tilde f,g'gh),(\tilde \rho, h^{-1}),\gamma) = ((\tilde f,g'gh),(\tilde \rho, h^{-1}),(\id_{t(\gamma)} \cdot i(\tilde\gamma)\cdot \tilde \gamma) \circ \gamma) \sim((\tilde f,g'gh) \circ (\id_{t(\gamma)} \cdot i(\tilde\gamma)),(\tilde \rho, h^{-1}) \circ \tilde \gamma,\gamma) 
\end{multline*}
so that it suffices to calculate
\begin{eqnarray*}
(\tilde f,g'gh) \circ (\id_{t(\gamma)} \cdot i(\tilde\gamma))
&=& (\rho(\tilde f,\id_{h^{-1}}\cdot i(\tilde \gamma)),g'gh')
\\&=& (\rho(\tilde\rho^{-1} \circ f \circ R(\rho',\id_g),i(\tilde\gamma)), g'gh')
\\&=& (\rho(R(\tilde\rho^{-1},i(\tilde\gamma)\cdot\id_{h'^{-1}}) \circ f \circ R(\rho',\id_g),\id_{h'}), g'gh')\text{,}
\end{eqnarray*} 
where the last step uses the compatibility condition for $\rho$ from Definition \ref{def:actionanafunctor} (ii). In any 2-group, we have $i(\tilde\gamma)\cdot\id_{s(\tilde\gamma)} = (\id_{t(\tilde\gamma)^{-1}} \cdot \tilde\gamma)^{-1}$, in which case the last line is exactly the formula \erf{eq:ftilde} for the pair $(\tilde\rho',h')$. 

Next we check that $\beta$ is well-defined under the equivalence relation on the tensor product $P' \otimes \zeta_1^{*}Q$. We have
\begin{equation*}
x:=((\rho',g'),(f,g),(\gamma_1\cdot\gamma_2) \circ \gamma) \sim ((\rho',g') \circ \gamma_1,(f,g) \circ \gamma_2,\gamma) =: x'
\end{equation*}
for $\gamma_1,\gamma_2 \in \Gamma_1$ such that $t(\gamma_1)=g'$, $t(\gamma_2) = g$ and $s(\gamma_1)s(\gamma_2) = t(\gamma)$. Taking advantage of the fact that we can make the same choice of $(\tilde\rho,h)$ for both representatives $x$ and $x'$, it is straightforward to show that $\beta(x) = \beta(x')$. 
Finally, it is obvious from the definition of $\beta$ that it is anchor-preserving and $\Gamma$-equivariant.
\end{proof}

In order to show that the triple $(Z,Q,\beta)$ defines a 1-morphism between bundle gerbes, it remains to verify that the bundle isomorphism $\beta$ is compatible the with the bundle gerbe products $\mu_1$ and $\mu_2$ in the sense of diagram \erf{compgerbemorph}. This is straightforward to do and left for the reader.

\subsubsection*{Definition of $\mathscr{E}_M$ on 2-morphisms, compositors and unitors}

Let $F_1,F_2: \mathcal{P} \to \mathcal{P}'$ be 1-morphisms between principal $\Gamma$-bundles over $M$, and let $\eta: F \Rightarrow G$ be a 2-morphism. Between the $\Gamma$-bundles $Q_1$ and $Q_2$, which live over the same common refinement $Z = \mathcal{P}_0 \times_M \mathcal{P}_0'$, we find immediately the smooth map
\begin{equation*}
\eta: Q_1 \to Q_2: (f_1,g) \mapsto (\eta(f_1),g)
\end{equation*}
which is easily verified to be a bundle morphism. 
Its compatibility with the bundle morphisms $\beta_1$ and $\beta_2$ in the sense of the simplified diagram \erf{eq:comp2morphsimple} is also easy to check.
Thus, we have defined a 2-morphism $\mathscr{E}_M(\eta)\maps \mathscr{E}_M(F_1) \Rightarrow \mathscr{E}_M(F_2)$.

The compositor for  1-morphisms $F_1: \mathcal{P} \to \mathcal{P}'$ and $F_2: \mathcal{P}' \to \mathcal{P}''$ is a bundle gerbe 2-morphism 
\begin{equation*}
c_{F_1,F_2} : \mathscr{E}_M(F_2 \circ F_1) \to \mathscr{E}_M(F_2) \circ \mathscr{E}_M(F_1)\text{.}
\end{equation*}
Employing the above constructions, the 1-morphism $\mathscr{E}_M(F_2 \circ F_1)$ is defined on the common refinement $Z_{12} \df \mathcal{P}_0 \times_M \mathcal{P}_0''$ and has the $\Gamma$-bundle $Q_{12} = (F_1 \times_{\mathcal{P}_0'} F_2)/\mathcal{P}_1' \times \Gamma_0$, whereas the 1-morphism $\mathscr{E}_M(F_2) \circ \mathscr{E}_M(F_1)$ is defined on the common refinement $Z :=\mathcal{P}_0 \times_M \mathcal{P}_0' \times_M \mathcal{P}_0''$ and has the $\Gamma$-bundle $Q_2 \otimes Q_1$ with $Q_k = F_k \times \Gamma_0$. The compositor $c_{F_1,F_2}$ is defined over the refinement $Z$ with the obvious refinement maps  $\mathrm{pr}_{13}:Z \to Z_{12}$ and $\id:Z \to Z$ making diagram \erf{diag:refinement} commutative. It is thus a  bundle morphism $c_{F_1,F_2}: \mathrm{pr}_{13}^{*}Q_{12} \to Q_2 \otimes Q_1$. For elements in a tensor product of $\Gamma$-bundles we use the notation of Example \ref{ex:tensorproduct} \erf{ex:tensorproduct:general}. Then, we define $c_{F_1,F_2}$ by
\begin{equation}
\label{eq:assignmentcompositor}
((p,p',p''),(f_1,f_2,g)) \mapsto ((\rho_2(\tilde\rho^{-1} \circ f_2,\id_h),gh),(f_1 \circ \tilde \rho,h^{-1}),\id_g)\text{,}
\end{equation}
where $h\in \Gamma_0$  and $\tilde\rho: R(p',h^{-1}) \to \alpha_r(f_1)=\alpha_l(f_2)$ are chosen in the same way as in the proof of Lemma \ref{lem:eonemorphisms}. 
The assignment \erf{eq:assignmentcompositor} does not depend on the choices of $h$ and $\tilde\rho$, nor on the choice of the representative $(f_1,f_2)$ in $(F_1 \times_{\mathcal{P}_0'} F_2)/\mathcal{P}_1' $. 
It is obvious that \erf{eq:assignmentcompositor} is anchor-preserving, and its $\Gamma$-equivariance can be seen by choosing $(\tilde\rho,h)$ in order to compute $c_{F_1,F_2}((p,p',p''),(f_1,f_2,g))$  and $(\tilde\rho',h)$ with $\tilde\rho' := R(\tilde\rho,\id_{g^{-1}} \cdot \gamma^{-1})$ in order to compute $c_{F_1,F_2}(((p,p',p''),(f_1,f_2,g)) \circ \gamma)$.
In order to complete the construction of the bundle gerbe 2-morphism $c_{F_1,F_2}$ we have to prove that the bundle morphism $c_{F_1,F_2}$ is compatible with the isomorphisms $\beta_{12}$ of $\mathscr{E}_M(F_2 \circ F_1)$ and $(\id \otimes \beta_{1}) \circ (\beta_2 \otimes \id)$ of $\mathscr{E}_M(F_2) \circ \mathscr{E}_M(F_1)$ in the sense of diagram \erf{eq:comp2morphsimple}. We start with an element $((\rho'',g''),(f_{12},g)) \in \mathscr{E}_M(\mathcal{P}'') \otimes \zeta_1^{*}Q_{12}$, where $f_{12} = (f_1,f_2)$. We have
\begin{equation*}
\beta_{12}((\rho'',g''),(f_{12},g)) = (\widetilde{f_{12}},g''gh,\tilde \rho,h^{-1})
\end{equation*}
upon choosing $(\tilde \rho,h)$ as required in the definition of $\mathscr{E}_M(F_2 \circ F_1)$. Writing $\widetilde{f_{12}}=(\tilde f_1,\tilde f_2)$ further we have
\begin{equation}
\label{eq:prop719eq1}
(\zeta_2^{*}c_{F_1,F_2} \otimes \id)(\widetilde{f_{12}},g''gh,\tilde \rho,h^{-1}) = (\rho_2(\tilde\rho_2^{-1} \circ \tilde f_2,\id_{h_2}),g''ghh_2,\tilde f_1 \circ \tilde \rho_2,h_2^{-1},\tilde\rho,h^{-1})
\end{equation}
upon choosing appropriate $(\tilde\rho_2,h_2)$ as required in the definition of $c_{F_1,F_2}$. This is the result of the clockwise composition of diagram \erf{eq:comp2morphsimple}. Counter-clockwise, we first get
\begin{equation*}
(\id \otimes \zeta_1^{*}c_{F_1,F_2})((\rho'',g''),(f_{12},g)) = (\rho'',g'',f'',gh_1,f',h_1^{-1})
\end{equation*}
for choices $(\tilde\rho_1,h_1)$, where $f'' := \rho_2(\tilde\rho_1^{-1} \circ f_2 , \id_{h_1})$ and $f' := f_1 \circ \tilde \rho_1$. Next we apply the isomorphism $\beta_2$ of $\mathscr{E}_M(F_2)$ and get 
\begin{equation*}
(\beta_2 \otimes \id)(\rho'',g'',f'',gh_1,f'_1,h_1^{-1}) = (\widetilde {f''},g''ghh_2,\hat\rho,\hat h^{-1},f_1',h_1^{-1})
\end{equation*}
where we have used the choices $(\hat\rho,\hat h)$ defined by $\hat\rho := R(\tilde\rho_1^{-1},h_1) \circ R(\tilde\rho_2,h^{-1}h_1)$ and $\hat h := h_1^{-1}hh_2$. The last step is to apply the isomorphism $\beta_1$ of  $\mathscr{E}_M(F_2)$  which gives
\begin{equation}
\label{eq:prop719eq2}
(\id \otimes \beta_1)(\widetilde {f''},g''ghh_2,\hat\rho,\hat h^{-1},f_1',h_1^{-1}) = (\widetilde{f''},g''ghh_2,\widetilde{f'},h_2^{-1},\tilde\rho,h^{-1})\text{,}
\end{equation}
where we have used the choices $(\tilde\rho,h)$ from above. Comparing \erf{eq:prop719eq1} and \erf{eq:prop719eq2}, we have obvious coincidence in all but the first and the third components. For these remaining factors, coincidence follows from the definitions of the various variables.

Finally, we have to construct unitors. The unitor for a principal $\Gamma$-2-bundle $\mathcal{P}$ over $M$ is a bundle gerbe 2-morphism
\begin{equation*}
u_{\mathcal{P}}: \mathscr{E}_M(\id_{\mathcal{P}}) \Rightarrow \id_{\mathscr{E}_M(\mathcal{P})}\text{.}
\end{equation*}
Abstractly, one can associate to $\id_{\mathscr{E}_M(\mathcal{P})}$ the 1-morphism $\id_{\mathscr{E}_M(\mathcal{P})}^{FP}$ constructed in the proof of Lemma  \ref{lem:canonicalrefinement}, and then notice that $\id_{\mathscr{E}_M(\mathcal{P})}^{FP}$ and $\mathscr{E}_M(\id_{\mathcal{P}})$ are canonically 2-isomorphic. In more concrete terms, the unitor $u_{\mathcal{P}}$ has the refinement $W := \mathcal{P}_0^{[3]}$ with the surjective submersions $r := \mathrm{pr}_{12}$ and $r' := \mathrm{pr}_{3}$ to the refinements $Z=\mathcal{P}_0^{[2]}$ and $Z'= \mathcal{P}_0$ of the 1-morphisms $\mathscr{E}_M(\id_{\mathcal{P}})$ and $\id_{\mathscr{E}_M(\mathcal{P})}$, respectively. The relevant maps $x_W$ and $y_W$ are $\mathrm{pr}_{13}$ and $\mathrm{pr}_{23}$, respectively. The principal $\Gamma$-bundle of the 1-morphism $\id_{\mathscr{E}_M(\mathcal{P})}$ is the trivial bundle $Q' = \trivlin_1$. We claim that the principal $\Gamma$-bundle $Q$ of $\mathscr{E}_M(\id_{\mathcal{P}})$ is the bundle $P$ of the bundle gerbe $\mathscr{E}_M(\mathcal{P})$. Indeed, the formulae \erf{eq:emorph} reduce for the identity anafunctor $\id_{\mathcal{P}}$ to those of \erf{defgerbefrombundle}. Now, the bundle isomorphism of the unitor $u_{\mathcal{P}}$ is
\begin{equation*}
\alxydim{}{y_W^{*}P \otimes r^{*}Q = \mathrm{pr}_{23}^{*}P \otimes \mathrm{pr}_{12}^{*}P \ar[r]^-{\mu} & \mathrm{pr}_{13}^{*}P \cong {r'}^{*}Q' \otimes x_W^{*}P\text{,} }
\end{equation*}
where $\mu$ is the bundle gerbe product of $\mathscr{E}_M(\mathcal{P})$. The commutativity of diagram \erf{eq:comp2morphcomplicated} follows from the associativity of $\mu$.

\begin{proposition}
\label{prop:efunctor}
The assignments $\mathscr{E}_M$ for objects, 1-morphisms  and 2-morphisms, together with the compositors and unitors defined above, define a 2-functor
\begin{equation*}
\mathscr{E}_M: \zwoabun \Gamma M \to \grb\Gamma M\text{.}
\end{equation*}
\end{proposition}

\begin{proof}
A list of axioms for a 2-functor with the same conventions as we use here can be found in \cite[Appendix A]{schreiber2}. The first axiom  requires that the 2-functor $\mathscr{E}_M$ respects the vertical composition of 2-morphisms -- this follows immediately from the definition. 

The second axiom requires that the compositors respect the horizontal composition of 2-morphisms. To see this, let $F_1,F_1': \mathcal{P} \to \mathcal{P}'$  and $F_2,F_2': \mathcal{P}' \to \mathcal{P}''$ be 1-morphisms between principal $\Gamma$-2-bundles, and let $\eta_1: F_1 \Rightarrow F_1'$ and $\eta_2: F_2 \Rightarrow F_2'$ be 2-morphisms. Then, the diagram
\begin{equation*}
\alxydim{@C=3cm@R=1.3cm}{\mathscr{E}_M(F_2 \circ F_1) \ar@{=>}[d]_{c_{F_1,F_2}} \ar@{=>}[r]^{\mathscr{E}_M(\eta_1 \circ \eta_2)} & \mathscr{E}_M(F_2' \circ F_1') \ar@{=>}[d]^{c_{F_1',F_2'}} \\ \mathscr{E}_M(F_2) \circ \mathscr{E}_M(F_1) \ar@{=>}[r]_{\mathscr{E}_M(\eta_1) \circ \mathscr{E}_M(\eta_2)} & \mathscr{E}_M(F_2') \circ \mathscr{E}_M(F_1')}
\end{equation*}
has to commute. Indeed, in order to compute $c_{F_1,F_2}$ and $c_{F_1',F_2'}$ one can make the same choice of $(\tilde\rho,h)$, because the transformations $\eta$ and $\eta_2$ preserve the anchors. Then, commutativity follows  from the fact that  $\eta_1$ and $\eta_2$ commute with the groupoid actions and the $\Gamma_1$-action according to Definition \ref{def:actionanafunctor}. 

The third axiom describes the compatibility of the compositors with the composition of 1-morphisms in the sense that the diagram
\begin{equation*}
\alxydim{@C=1.8cm@R=1.3cm}{\mathscr{E}_M(F_3 \circ F_2 \circ F_1) \ar@{=>}[r]^-{c_{F_2 \circ F_1,F_3}}  \ar@{=>}[d]_{c_{F_3 \circ F_2, F_1}} & \mathscr{E}_M(F_3) \circ \mathscr{E}_M(F_2 \circ F_1) \ar@{=>}[d]^{\id \circ c_{F_2,F_1}} \\ \mathscr{E}_M(F_3 \circ F_2) \circ \mathscr{E}_M(F_1) \ar@{=>}[r]_-{c_{F_3,F_2} \circ \id} & \mathscr{E}_M(F_3) \circ \mathscr{E}_M(F_2) \circ \mathscr{E}_M(F_1)\text{.}}
\end{equation*}
is commutative. In order to verify this, one starts with an element $(f_1,f_2,f_3,g)$ in $\mathscr{E}_M(F_3 \circ F_2 \circ F_1)$. In order to go clockwise, one chooses pairs $(\tilde \rho_{12,3},h_{12,3})$ and $(\tilde\rho_{1,2},h_{1,2})$ and gets from the definitions
\begin{equation*}
\mathrm{CW} = ((\rho_3(\tilde\rho_{12,3}^{-1} \circ f_3,\id_{h_{12,3}}),gh_{12,3}),(\rho_2(\tilde\rho_{1,2}^{-1} \circ f_2 \circ \tilde\rho_{12,3},\id_{h_{1,2}}),h_{12,3}^{-1}h_{1,2}),(f_1 \circ \tilde\rho_{1,2},h^{-1}_{1,2}))\text{.}
\end{equation*}
Counter-clockwise, one can choose firstly again the pair $(\tilde \rho_{1,2},h_{1,2})$ and then the pair $(\tilde\rho_{2,3},h_{2,3})$ with $\tilde\rho_{2,3} \eq R(\tilde\rho_{12,3},\id_{h_{1,2}})$ and $h_{2,3} = h^{-1}_{1,2}h_{12,3}$. Then, one gets
\begin{multline*}
\mathrm{CCW} = ((\rho_3(\tilde\rho_{2,3}^{-1} \circ \rho_3(f_3,\id_{h_{1,2}}),\id_{h_{2,3}}),gh_{1,2}h_{2,3}),(\rho_2(\tilde\rho_{1,2}^{-1} \circ f_2,\id_{h_{1,2}}) \circ \tilde\rho_{2,3},h_{2,3}^{-1}),(f_1 \circ \tilde\rho_{1,2},h^{-1}_{1,2}))\text{,}
\end{multline*}
where one has to use formula \erf{eq:companafunctorsaction} for the $\Gamma_1$-action on the composition of equivariant anafunctors.
Using the definitions of $h_{2,3}$ and $\tilde\rho_{2,3}$ as well as the axiom of Definition \ref{def:actionanafunctor} (ii) one can show that $\mathrm{CW}=\mathrm{CCW}$.

The fourth and last axiom requires that compositors and unitors are  compatible with each other in the sense that for each 1-morphism $F: \mathcal{P} \to \mathcal{P}'$ the 2-morphisms
\begin{equation*}
\alxydim{@C=1.5cm}{\mathscr{E}_M(F) \cong \mathscr{E}_M(F \circ \id_{\mathcal{P}}) \ar@{=>}[r]^-{c_{\id_{\mathcal{P}},F}} & \mathscr{E}_M(F) \circ \mathscr{E}_M(\id_{\mathcal{P}}) \ar@{=>}[r]^-{\id \circ u_{\mathcal{P}}} & \mathscr{E}_M(F) \circ \id_{\mathscr{E}_M(\mathcal{P})} \cong \mathscr{E}_M(F)}
\end{equation*}
and   
\begin{equation*}
\alxydim{@C=1.3cm}{\mathscr{E}_M(F) \cong \mathscr{E}_M(\id_{\mathcal{P'}} \circ F) \ar@{=>}[r]^-{c_{F,\id_{\mathcal{P'}}}} & \mathscr{E}_M(\id_{\mathcal{P'}}) \circ \mathscr{E}_M(F) \ar@{=>}[r]^-{ u_{\mathcal{P}'} \circ \id} & \id_{\mathscr{E}_M(\mathcal{P}')} \circ \mathscr{E}_M(F) \cong \mathscr{E}_M(F)}
\end{equation*}
are the identity 2-morphisms. We prove this for the first one and leave the second as an exercise. Using the definitions, we see that the 2-morphism has the refinement $W := \mathcal{P}_0 \times_M \mathcal{P}_0 \times_M \mathcal{P}_0'$ with $r = \mathrm{pr}_{13}$ and $r' = \mathrm{pr}_{23}$. The maps $x_W: W \to \mathcal{P}_0 \times_M \mathcal{P}_0$ and $y_W: W \to \mathcal{P}_0' \times_M \mathcal{P}_0'$ are $\mathrm{pr}_{12}$ and $\Delta \circ \mathrm{pr}_3$, respectively, where $\Delta$ is the diagonal map. Its bundle morphism is
a morphism
\begin{equation*}
\varphi : \mathrm{pr}_{13}^{*}Q \to \mathrm{pr}_{23}^{*}Q \otimes \mathrm{pr}_{12}^{*}P\text{,}
\end{equation*}
where $Q = F \times \Gamma_0$ is the principal $\Gamma$-bundle of $\mathscr{E}_M(F)$, and  $P=\mathcal{P}_1 \times \Gamma_0$ is the principal $\Gamma$-bundle of $\mathscr{E}_M(\mathcal{P})$. Over a point $(p_1,p_2,p')$ and $(f,g)\in \mathrm{pr}_{13}^{*}Q$, i.e. $\alpha_l(f)=p_1$ and $R(\alpha_r(f),g^{-1})=p'$, the bundle morphism $\varphi$ is given by
\begin{equation*}
 (f,g) \mapsto (\rho(\tilde\rho^{-1} \circ f,\id_h),gh,\tilde\rho,h^{-1})\text{,}
\end{equation*}
where $h \in \Gamma_0$, and $\tilde\rho \in \mathcal{P}_1$ with $s(\tilde\rho) = R(p_2,h^{-1})$ and $t(\tilde\rho) = \alpha_l(f)$. We have to compare $(W,\varphi)$ with the identity 2-morphism of $\mathscr{E}_M(F)$, which has the refinement $Z$ with $r=r' = \id$ and the identity bundle morphism.  According to the equivalence relation on bundle gerbe 2-morphisms we have to evaluate $\varphi$ over a point $w \in W$ with $r(w) = r'(w)$, i.e. $w$ is of the form $w=(p,p,p')$. Here we can choose $h=1$ and $\tilde\rho = \id_{p}$, in which case we have $\varphi(f,g) = ((f,g),(\id_p,1))$. This is indeed the identity on $Q$. 
\end{proof}

\subsubsection*{Properties of the 2-functor $\mathscr{E}_M$}

For the proof of Theorem \ref{th:equivalence} we provide the following two statements.

\begin{lemma}
\label{lem:efullyfaithful}
The 2-functor $\mathscr{E}_M$ is fully faithful on Hom-categories. 
\end{lemma}

\begin{proof}
Let $\mathcal{P}, \mathcal{P}'$ be principal $\Gamma$-2-bundles over $M$, and let $F_1,F_2: \mathcal{P} \to \mathcal{P}'$ be 1-morphisms. By Lemma \ref{lem:canonicalrefinement} every 2-morphism $\eta: \mathscr{E}_M(F_1) \Rightarrow \mathscr{E}_M(F_2)$ can be represented by one whose refinement is $\mathcal{P}_0 \times_M \mathcal{P}_0'$, so that its bundle isomorphism is $\eta: Q_1 \to Q_2$, where  $Q_k := F_k \times \Gamma$ for $k=1,2$. We can read off a map $\eta: F_1 \to F_2$, and it is easy to see that this is a 2-morphism $\eta: F_1 \Rightarrow F_2$. This procedure is clearly inverse to the 2-functor $\mathscr{E}_M$ on 2-morphisms. \end{proof}

\begin{proposition}
\label{prop:estack}
The 2-functors $\mathscr{E}_M$ form a 1-morphism between pre-2-stacks. 
\end{proposition}

\begin{proof}
For a smooth map $f:M \to N$, we have to look at the diagram
\begin{equation*}
\alxydim{@R=1.3cm}{\zwoabun \Gamma N \ar[d]_{\mathscr{E}_N} \ar[r]^{f^{*}} & \zwoabun \Gamma M \ar[d]^{\mathscr{E}_M} \\ \grb\Gamma N \ar[r]_{f^{*}} & \grb\Gamma M}
\end{equation*}
of 2-functors. For $\mathcal{P}$ a principal $\Gamma$-2-bundle over $N$, the $\Gamma$-bundle gerbe $\mathscr{E}_M(f^{*}\mathcal{P})$ has the surjective submersion $\mathrm{pr}_1: Y:=M \times_N \mathcal{P}_0 \to M$, the principal $\Gamma$-bundle $P :=M \times_N \mathcal{P}_1 \times \Gamma_0$ over $Y^{[2]}$, and a bundle gerbe product $\mu$ defined as in \erf{defmu} that ignores the $M$-factor.  On the other hand, the $\Gamma$-bundle gerbe $f^{*}\mathscr{E}_N(\mathcal{P})$ has the same surjective submersion, and  -- up to canonical identifications between fibre products --  the same $\Gamma$-bundle and the same bundle gerbe product. These canonical identifications make up a pseudonatural transformation that renders the above diagram commutative.  
\end{proof}

\subsection{From Bundle Gerbes to Principal 2-Bundles}

\label{sec:reconstruction}

We now provide the data we will feed into Lemma \ref{lem:equivalence} in order to produce a 2-functor
$\mathscr{R}_M\maps \grb \Gamma M \to \zwoabun \Gamma M$
that is inverse to the 2-functor $\mathscr{E}_M$ constructed in the previous section. These data are:
\begin{enumerate}
\item 
A principal $\Gamma$-2-bundle $\mathscr{R}_\mathcal{G}$ for each $\Gamma$-bundle gerbe $\mathcal{G}$ over $M$.

\item
A 1-isomorphism $\mathcal{A}_{\mathcal{G}}: \mathcal{G} \to \mathscr{E}_M(\mathscr{R}_{\mathcal{G}})$ for each  $\Gamma$-bundle gerbe $\mathcal{G}$ over $M$.

\item
A 1-isomorphism $\mathscr{R}_{\mathcal{A}}: \mathcal{P} \to \mathcal{P}'$ and a 2-isomorphism $\eta_{\mathcal{A}}: \mathcal{A} \Rightarrow \mathscr{E}_M(\mathscr{R}_{\mathcal{A}})$ for all principal $\Gamma$-2-bundles $\mathcal{P},\mathcal{P}'$ over $M$ and all bundle gerbe 1-morphisms $\mathcal{A}: \mathscr{E}_M(\mathcal{P}) \to \mathscr{E}_M(\mathcal{P}')$.

\end{enumerate}

\subsubsection*{Construction of the principal $\Gamma$-2-bundle $\mathscr{R}_{\mathcal{G}}$}

 We assume that $\mathcal{G}$ consists of a surjective submersion $\pi:Y \to M$, a principal $\Gamma$-bundle $P$ over $Y^{[2]}$ and a bundle gerbe product $\mu$. Let $\alpha: P \to \Gamma_0$ be the anchor of $P$, and let $\chi: P \to Y^{[2]}$ be the bundle projection.

The Lie groupoid $\mathcal{P}$ of the principal 2-bundle $\mathscr{R}_\mathcal{G}$ is defined by
\begin{equation*}
\mathcal{P}_0 := Y \times \Gamma_0
\quand
\mathcal{P}_1 := P \times \Gamma_0\text{;}
\end{equation*}
source map, target maps, and composition are given by, respectively,
\begin{multline}
\label{defbundlefromgerbe}
s(p,g):= (\pi_2(\chi(p)),g)
\quomma
t(p,g):=  (\pi_1(\chi(p)),\alpha(p)^{-1} \cdot g)
\quand 
(p_2,g_2) \circ (p_1,g_1) := (\mu(p_1 , p_2), g_1)\text{.}
\end{multline}
The identity morphism of an object $(y,g) \in \mathcal{P}_0$ is $(t_y,g) \in \mathcal{P}_1$, where $t_{y}$ denotes the unit element in $P$ over the point $(y,y)$, see Lemma \ref{lem:buntriv}. The inverse of a morphism $(p,g) \in \mathcal{P}_1$ is $(i(p),\alpha(p)^{-1}g)$, where $i: P \to P$  is the map from Lemma \ref{lem:buntriv}.
The bundle projection is $\pi(y,g) := \pi(y)$. The action is given on objects and morphisms by
\begin{equation}
\label{defaction}
R_0((y,g),g') := (y,gg')
\quand
R_1((p,g),\gamma) := \Big(p \circ \big(\id_{g} \cdot \gamma \cdot \id_{t(\gamma)^{-1}g^{-1}\alpha(p)}\big),g\cdot s(\gamma)\Big)\text{.}
\end{equation}

\begin{lemma}
\label{lem:action}
This defines a functor 
$R: \mathcal{P} \times \Gamma \to \mathcal{P}$, and $R$ is an action of $\Gamma$ on $\mathcal{P}$.
\end{lemma}

\begin{proof}
We assume that $t:H \to G$ is a smooth crossed module, and that $\Gamma$ is the Lie 2-group associated to it, see Example \ref{crossedmodule} and Remark \ref{rem:crossedmodule}. Then we use the correspondence between principal $\Gamma$-bundles and principal $H$-bundles with $H$-anti-equivariant maps to $G$ of Lemma \ref{actionbundles}.
Writing $\gamma = (h,g')$, we have
\begin{equation*}
R_1((p,g),\gamma) = (p \star \lw{g}{h},gg')\text{.}
\end{equation*}
With this simple formula at hand it is straightforward to show that $R$ respects source and target maps and satisfies the axiom of an action.  
For the composition, we assume composable  $(p_2,g_2),(p_1,g_1) \in \mathcal{P}_1$, i.e. $g_2 = \alpha(p_1)^{-1}g_1$, and composable $(h_2,g_2'),(h_1,g_1') \in \Gamma_1$, i.e. $g_2' = t(h_1)g_1'$. Then we have
\begin{eqnarray*}
R((p_2,g_2) \circ (p_1,g_1),(h_2,g_2') \circ (h_1,g_1')) 
&=& R((\mu(p_1,p_2),g_1),(h_2h_1,g_1'))\\
&=& (\mu(p_1,p_2) \star \lw{g_1}{(h_2h_1)},g_1g_1')\\
&=& (\mu(p_1 \star \lw{g_1}{h_2},p_2) \star \lw{g_1}{h_1},g_1g_1')\\
&=& (\mu(p_1,p_2 \star \lw{g_2}{h_2}) \star \lw{g_1}{h_1},g_1g_1')\\
&=& (\mu(p_1\star \lw{g_1}{h_1},p_2 \star \lw{g_2}{h_2}) ,g_1g_1')\\
&=& (p_2 \star \lw{g_2}h_2,g_{2}g_{2}') \circ (p_1 \star \lw{g_1}h_1,g_1g_1')\\
&=& R((p_2,g_2),(h_2,g_2')) \circ R((p_1,g_1),(h_1,g_1'))\text{,}
\end{eqnarray*} 
finishing the proof.
\end{proof}

It is obvious that the action $R$ preserves the projection $\pi$. Thus,  in order to complete the construction of the principal 2-bundle  $\mathscr{R}_\mathcal{G}$  it  remains to show that the functor $\tau = (\mathrm{pr}_1,R)$ is a weak equivalence.
This is the content of the following two lemmata in connection with Theorem \ref{weak}.

\begin{lemma}
$\tau$ is smoothly essentially surjective. 
\end{lemma}

\begin{proof}
The condition we have to check is whether or not the map
\begin{equation*}
\alxydim{@C=2cm}{(Y \times \Gamma_0 \times \Gamma_0) \lli{\tau}\times_t ((P \times \Gamma_0) \times_M (P \times \Gamma_0)) \ar[r]^-{(s \times s) \circ \mathrm{pr}_2} & (Y \times \Gamma_0) \times_M (Y \times \Gamma_0)}
\end{equation*}
is a surjective submersion. The left hand side is diffeomorphic to $(P \times \Gamma_0) \lli{\pi_1} \times_{\pi_1} (P \times \Gamma_0)$ via $\mathrm{pr}_2$, so that this is equivalent to checking that
\begin{equation*}
s \times s: (P \times \Gamma_0) \lli{\pi_1 \circ \chi} \times_{\pi_1 \circ \chi} (P \times \Gamma_0) \to  (Y \times \Gamma_0) \times_M (Y \times \Gamma_0)
\end{equation*}
is a surjective submersion. Since the $\Gamma_0$-factors are just spectators, this is in turn equivalent to checking that
\begin{equation*}
(\pi_2 \times \pi_2) \circ (\chi \times \chi): P \lli{\pi_1 \circ \chi}\times_{\pi_1 \circ \chi} P \to Y^{[2]}
\end{equation*} 
is a surjective submersion. It fits into the pullback diagram
\begin{equation*}
\alxydim{@R=1.3cm}{P \lli{\pi_1 \circ \chi}\times_{\pi_1 \circ \chi} P \ar@{^(->}[r] \ar[d]_{\chi \times \chi} & P \times P \ar[d]^{\chi \times \chi} \\ Y^{[2]} \lli{\pi_1}\times_{\pi_1} Y^{[2]} \ar[d]_{\pi_2 \times \pi_2} \ar@{^(->}[r] & Y^{[2]} \times Y^{[2]} \ar[d]^{\pi_2 \times \pi_2} \\ Y^{[2]} \ar@{^(->}[r] & Y \times Y}
\end{equation*}
which has  a surjective submersion on the right hand side; hence, also the map on the left hand side must be a surjective submersion. 
\end{proof}

\begin{lemma}
$\tau$ is smoothly fully faithful. 
\end{lemma}

\begin{proof}
We assume a smooth manifold $N$ with two smooth maps 
\begin{equation*}
f: N \to (\mathcal{P}_0 \times \Gamma_0) \times (\mathcal{P}_0 \times \Gamma_0)
\quand
g: N \to \mathcal{P}_1 \times_M \mathcal{P}_1
\end{equation*}
such that the diagram
\begin{equation*}
\alxydim{@R=1.3cm}{N \ar[d]_{f} \ar[r]^-{g} & \mathcal{P}_1 \times_M \mathcal{P}_1 \ar[d]^{s \times t} \\  (\mathcal{P}_0 \times \Gamma_0) \times (\mathcal{P}_0 \times \Gamma_0) \ar[r]_-{\tau \times \tau} & (\mathcal{P}_0 \times_M \mathcal{P}_0) \times (\mathcal{P}_0 \times_M \mathcal{P}_0) }
\end{equation*}
is commutative. For a fixed point $n \in N$ we put 
\begin{equation*}
((p_1,g_1),(p_2,g_2)):= g(n)\in (P \times \Gamma_0) \times_M (P \times \Gamma_0)
\end{equation*}
and
\begin{equation*}
((y,g,\tilde g),(y',g',\tilde g')) := f(n) \in (Y \times \Gamma_0 \times \Gamma_0) \times (Y \times \Gamma_0 \times \Gamma_0)\text{.}
\end{equation*}
The commutativity of the diagram implies $\chi(p_1)=\chi(p_2)=(y',y)$, so that there exists  $\gamma' \in \Gamma_1$ with $p_2 = p_1 \circ \gamma'$. We define 
$\gamma := \id_{g_1^{-1}} \cdot \gamma'\cdot \id_{\alpha(p_2)^{-1}g_2}$, 
which yields a morphism $\gamma\in \Gamma_1$ satisfying $\tau(p_1,g_1,\gamma) = (p_1,g_1,p_2,g_2) = g(n)$. On the other hand, we check that 
\begin{equation*}
(s(p_1,g_1,\gamma), t(p_1,g_1,\gamma)) = (\pi_2(p_1),g_1,s(\gamma),\pi_1(p_1),\alpha(p_1)^{-1}g_1,t(\gamma))  = f(n)\text{,}
\end{equation*} 
using that $s(\gamma) = g_1^{-1}g_2$ and $t(\gamma) = g_1^{-1}\alpha(p_1)\alpha(p_2)^{-1}g_2$. Summarizing, we have defined a smooth map
\begin{equation*}
\sigma: N \to \mathcal{P}_1 \times \Gamma_1:n \mapsto (p_1,g_1,\gamma)
\end{equation*}
such that $\tau \circ \sigma = g$ and $(s \times t) \circ \sigma=f$. Now let $\sigma':N \to \mathcal{P}_1 \times \Gamma_1$ be another such map, and let $\sigma'(n) =: (p_1',g_1',\gamma')$. The condition that $\tau(\sigma(n)) =g(n) = \tau(\sigma'(n))$  shows immediately that $p_1=p_1'$ and $g_1=g_1'$, and then that $p_1 \circ \gamma = p_1 \circ \gamma'$. But since the $\Gamma$-action on $P$ is principal, we have $\gamma=\gamma'$. This shows $\sigma=\sigma'$.
Summarizing, $\mathcal{P}_1 \times \Gamma_1$ is a pullback. 
\end{proof}

\begin{example}
\label{trans_abelian}
Suppose $\Gamma = \mathcal{B}\ueins$ (see Example \ref{example_groupoids} \erf{exBG}) and suppose $\mathcal{G}$ is a $\Gamma$-bundle gerbe over $M$, also known as a $\ueins$-bundle gerbe, see Example \ref{ex:abelianbundlegerbes}. Then, the associated principal $\mathcal{B}\ueins$-2-bundle $\mathscr{R}_\mathcal{G}$ has the  groupoid $\mathcal{P}$ with $\mathcal{P}_0=Y$ and $\mathcal{P}_1=P$, source and target maps $s = \pi_2 \circ \chi$ and $t=\pi_1 \circ \chi$, and composition $p_2 \circ p_1 = \mu(p_1,p_2)$. The action of $\mathcal{B}\ueins$ on $\mathcal{P}$ is trivial on the level of objects and the given $\ueins$-action on $P$ on the level of morphisms. The same applies for general abelian Lie groups $A$ instead of $\ueins.$
\end{example}

\subsubsection*{Construction of the 1-isomorphism $\mathcal{A}_{\mathcal{G}}: \mathcal{G} \to \mathscr{E}_M(\mathscr{R}_{\mathcal{G}})$}

The $\Gamma$-bundle gerbe $\mathscr{E}_M(\mathscr{R}_\mathcal{G})$ has the surjective submersion $\tilde Y := Y \times \Gamma_0$ with $\tilde\pi(y,g) := \pi(y)$. The total space of its $\Gamma$-bundle $\tilde P$ is $\tilde P := P \times \Gamma_0 \times \Gamma_0$; it has the anchor $\alpha(p,g,h)= h$, the bundle projection
\begin{equation*}
\tilde\chi: \tilde P \to \tilde Y^{[2]}: (p,g,h) \mapsto (  (\pi_1(\chi(p)),\alpha(p)^{-1}  g),(\pi_2(\chi(p)),gh^{-1}))\text{,}
\end{equation*}
 the $\Gamma$-action is
\begin{eqnarray*}
(p,g,h) \circ \gamma 
&\stackrel{\text{\erf{defgerbefrombundle}}}{=}&  ((p,g) \circ R((t_{\pi_2(\chi(p))},gh^{-1}),\gamma),s(\gamma))
 \\&\stackrel{\text{\erf{defaction}}}{=}& ((p,g) \circ (t_{\pi_2(\chi(p))} \circ (\id_{gh^{-1}} \cdot \gamma \cdot \id_{g^{-1}}),gh^{-1} s(\gamma)),s(\gamma))
 \\&\stackrel{\text{\erf{defbundlefromgerbe}}}{=}& (\mu (t_{\pi_2(\chi(p))} \circ \big(\id_{gh^{-1}} \cdot \gamma \cdot \id_{g^{-1}}\big),p),gh^{-1} s(\gamma),s(\gamma))
 \\&\stackrel{\text{\erf{deftensorproductrel}}}{=}& (p \circ \big(\id_{gh^{-1}} \cdot \gamma \cdot \id_{g^{-1}\alpha(p)}\big),gh^{-1} s(\gamma),s(\gamma))\text{,}
\end{eqnarray*}
 and its bundle gerbe product $\tilde\mu$ is given by
\begin{eqnarray*}
\tilde\mu((p_{23},g_{23},h_{23}),(p_{12},g_{12},h_{12})) 
&\stackrel{\text{\erf{defmu}}}{=}&   ((p_{12},g_{12}) \circ R((p_{23},g_{23}),\id_{h_{12}}) ,h_{23}h_{12})
\\&\stackrel{\text{\erf{defaction}}}{=}& ((p_{12},g_{12}) \circ (p_{23},g_{23}h_{12}) ,h_{23}h_{12})
\\&\stackrel{\text{\erf{defbundlefromgerbe}}}{=}&   (\mu(p_{23},p_{12}),g_{23}h_{12} ,h_{23}h_{12})\text{.}
\end{eqnarray*}
In order to compare the bundle gerbes $\mathcal{G}$ and $\mathscr{E}_M(\mathscr{R}_\mathcal{G})$ we consider the smooth maps  $\sigma: Y \to Y \times \Gamma_0$ and $\tilde\sigma: P \to \tilde P$  that are defined  by $\sigma(y):=(y,1)$ and $\tilde\sigma(p) := (p,\alpha(p),\alpha(p))$. 

\begin{lemma}
$\tilde\sigma$ defines an isomorphism $\tilde\sigma: P \to (\sigma \times \sigma)^{*}\tilde P$ of $\Gamma$-bundles over $Y^{[2]}$. Moreover, the diagram
\begin{equation*}
\alxydim{@C=1.4cm@R=1.3cm}{\pi_{23}^{*}P \otimes \pi_{12}^{*}P \ar[d]_{\mu} \ar[r]^{\tilde\sigma \otimes \tilde\sigma} & \tilde\pi_{23}^{*}\tilde P  \otimes \tilde\pi_{12}^{*}\tilde P \ar[d]^{\tilde\mu} \\ \pi_{13}^{*}P \ar[r]_{\tilde\sigma} & \tilde\pi_{13}^{*}\tilde P}
\end{equation*}
is commutative.
\end{lemma}

\begin{proof}
For the first part it suffices to prove that $\tilde\sigma$ is $\Gamma$-equivariant, preserves the anchors, and that the diagram
\begin{equation*}
\alxydim{@R=1.3cm}{P \ar[d]_{\chi} \ar[r]^{\tilde\sigma} & \tilde P \ar[d]^{\tilde\chi} \\ Y^{[2]} \ar[r]_{\sigma \times \sigma} & \tilde Y^{[2]}}
\end{equation*}
is commutative. Indeed, the commutativity of the diagram is obvious, and also that the anchors are preserved. For the $\Gamma$-equivariance, we have
\begin{equation*}
\tilde\sigma(p \circ \gamma) = (p \circ \gamma, s(\gamma), s(\gamma)) = (p,\alpha(p),\alpha(p)) \circ \gamma= \tilde\sigma(p) \circ \gamma\text{.}
\end{equation*}  
Finally, we calculate
\begin{eqnarray*}
\tilde\mu((p_{23},\alpha(p_{23}),\alpha(p_{23})),(p_{12},\alpha(p_{12}),\alpha(p_{12}))) &=& (\mu(p_{23},p_{12}),\alpha(p_{23})\alpha(p_{12}) ,\alpha(p_{23})\alpha(p_{12}))
\\&=& (\mu(p_{23},p_{12}),\alpha(\mu(p_{23},p_{12})),\alpha(\mu(p_{23},p_{12})))
\end{eqnarray*}
which shows the commutativity of the diagram.
\end{proof}

Via Lemma \ref{lem:invertibility}  the bundle morphism $\tilde\sigma$ defines the required 1-morphism $\mathcal{A}_{\mathcal{G}}$, and Lemma \ref{lem:nonstablemorphisms} guarantees that $\mathcal{A}_{\mathcal{G}}$ is  a 1-\emph{iso}morphism.

\subsubsection*{Construction of the 1-morphism $\mathscr{R}_{\mathcal{A}}: \mathcal{P} \to \mathcal{P}'$}

Let $\mathcal{A}: \mathscr{E}_M(\mathcal{P}) \to \mathscr{E}_M(\mathcal{P}')$ be a 1-morphism between $\Gamma$-bundle gerbes obtained from principal $\Gamma$-2-bundles $\mathcal{P}$ and $\mathcal{P}'$ over $M$.
By Lemma \ref{lem:canonicalrefinement} we can assume that $\mathcal{A}$ consists of a principal $\Gamma$-bundle $\chi:Q\to Z$ with $Z=\mathcal{P}_0 \times_M \mathcal{P}_0'$, and some isomorphism $\beta$ over $Z^{[2]}$. For preparation,
we consider the fibre products $Z_r:= \mathcal{P}_0 \times_M \mathcal{P}_0'^{[2]}$ and $Z_l := \mathcal{P}_0^{[2]} \times_M \mathcal{P}_0'$ with the obvious embeddings $\iota_l: Z_l \to Z$ and $\iota_r: Z_r \to Z$ obtained by doubling elements. Together with the trivialization of Corollary \ref{co:groupoidstructure}, the pullbacks of $\beta$ along $\iota_l$ and $\iota_r$ yield bundle morphisms
\begin{equation*}
\beta_l := \iota_l^{*}\beta : \mathrm{pr}_{13}^{*}Q \to \mathrm{pr}_{23}^{*}Q \otimes \mathrm{pr}_{12}^{*}P
\quand
 \beta_r := \iota_r^{*}\beta : \mathrm{pr}_{23}^{*}P' \otimes \mathrm{pr}_{12}^{*}Q \to \mathrm{pr}_{13}^{*}Q\text{,}
\end{equation*}
where $P := \mathcal{P}_1 \times \Gamma_0$ and $P' := \mathcal{P}' \times \Gamma_0$ are the principal $\Gamma$-bundles of the $\Gamma$-bundle gerbes $\mathscr{E}_M(\mathcal{P})$ and $\mathscr{E}_M(\mathcal{P}')$, respectively. 

\begin{lemma}
\label{lem:betaprop}
The bundle morphisms $\beta_l$ and $\beta_r$ have the following properties:
\begin{enumerate}[(i)]
\item 
They commute with each other in these sense that the diagram
\begin{equation*}
\alxydim{@C=1.7cm@R=1.3cm}{P'_{p_1',p_2'} \otimes Q_{p_1,p_1'} \ar[dr]|{\beta}  \ar[r]^-{\beta_r } \ar[d]_{\id \otimes \beta_l} & Q_{p_1,p_2'} \ar[d]^{\beta_l} \\ P'_{p_1',p_2'} \otimes Q_{p_2,p_1'} \otimes P_{p_1,p_2} \ar[r]_-{\beta_r \otimes \id} & Q_{p_2,p_2'} \otimes P_{p_1,p_2}}
\end{equation*}
is commutative for all $((p_1,p_1'),(p_2,p_2')) \in Z^{[2]}$. 

\item
$\beta_l$ is compatible with the bundle gerbe product $\mu$ in the sense that
\begin{equation*}
\beta_l|_{p_1,p_3,p'} = (\id \otimes \mu_{p_1,p_2,p_3}) \circ (\beta_l|_{p_2,p_3,p'} \otimes \id) \circ \beta_l|_{p_1,p_2,p'}
\end{equation*}
for all $(p_1,p_2,p_3,p')\in \mathcal{P}_0^{[3]} \times_M \mathcal{P}_0'$.

\item
$\beta_r$ is compatible with the bundle gerbe product $\mu'$ in the sense that
\begin{equation*}
\beta_r|_{p,p_1',p_3'} \circ (\mu'_{p_1',p_2',p_3'} \otimes \id)  = \beta_r|_{p,p_2',p_3'} \circ (\id \otimes \beta_r|_{p,p_1',p_2'})
\end{equation*}
for all $(p,p_1',p_2',p_3')\in \mathcal{P}_0 \times_M \mathcal{P}_0'^{[3]}$.

\end{enumerate}
\end{lemma}

\begin{proof}
The identities (ii) and (iii) follow by restricting the commutative diagram \erf{compgerbemorph} to the submanifolds $\mathcal{P}_0^{[3]} \times_M \mathcal{P}_0'$ and $\mathcal{P}_0 \times_M \mathcal{P}_0'^{[3]}$  of $Z^{[3]}$, respectively. Similarly, the commutativity of the two triangular subdiagrams in (i) follows by restricting \erf{compgerbemorph} along appropriate embeddings $Z^{[2]} \to Z^{[3]}$.
\end{proof}

Now we are in position to define the anafunctor $\mathscr{R}_{\mathcal{A}}$. First, we consider the left action
\begin{equation*}
\beta_0: \Gamma_0 \times Q \to Q : (g,q) \mapsto \beta_r((\id,g),q)
\end{equation*}
that satisfies $\alpha(\beta_0(g,q)) = g\alpha(q)$. 
The action $\beta_0$ is properly discontinuous and free because $\beta_r$ is a bundle isomorphism. 
The quotient $F := Q/\Gamma_0$ is the total space of the anafunctor $\mathscr{R}_{\mathcal{A}}$ we want to construct. Left and right anchors of an element $q\in F$ with $\chi(q)=(p,p')$ are given by 
\begin{equation*}
\alpha_l(q) := p
\quand
\alpha_r(q) := R(p',\alpha(q))\text{.}
\end{equation*}
The actions are  defined by
\begin{equation*}
\rho_l(\rho, q) := \beta_l^{-1}(q,(\rho,1))
\quand
\rho_r(q , \rho') := \beta_r((R(\rho',\id_{\alpha(q)^{-1}}),1),q)\text{.}
\end{equation*}
The left action is invariant under the action $\beta_0$ because of Lemma \ref{lem:betaprop} (i). For the right action, invariance follows from Lemma \ref{lem:betaprop} (ii) and the identity
\begin{equation*}
\mu'((R(\rho',\id_{\alpha(q)^{-1}g^{-1}}),1),(\id,g)) \stackrel{\text{\erf{defmu}}}{=}  \mu'((\id,g),(R(\rho',\id_{\alpha(q)^{-1}}),1))\text{.}
\end{equation*}

\begin{lemma}
The above formulas  define an anafunctor $F : \mathcal{P} \to \mathcal{P}'$.
\end{lemma}

\begin{proof}
The compatibility between anchors and actions is easy to check. The axiom for the actions $\rho_l$ and $\rho_r$ follows from Lemma \ref{lem:betaprop} (ii) and (iii). Lemma \ref{lem:betaprop} (i) shows that the actions commute.
It remains to prove that $\alpha_l: F \to \mathcal{P}_0$ is a principal $\mathcal{P}'$-bundle. Since $\alpha_l$ is a composition of surjective submersions, we only have to show that the map
\begin{equation*}
\tau: F \lli{\alpha_r} \times_{t} \mathcal{P}' \to F \lli{\alpha_l} \times_{\alpha_l} F : (q,\rho') \mapsto (q, \rho_r(q ,\rho'))
\end{equation*} 
is a diffeomorphism. We construct an inverse map $\tau^{-1}$ as follows. For $(q_1,q_2)$ with $\chi(q_1) = (p,p')$ and $\chi(q_2)\eq (p,\tilde p')$, choose a representative \begin{equation*}
((\tilde\rho',g'),\tilde q) := \beta_r|_{p,p',\tilde p'}^{-1}(q_2)\text{.}
\end{equation*}
Such choices can be made locally in a smooth way, and the result will not depend on them. We have $\chi(\tilde q) = (p,p')$ that that there exists a unique $\gamma \in \Gamma_1$ such that $q_1 = \tilde q \circ \gamma$. Now we put
\begin{equation*}
\tau^{-1}(q_1,q_2) := (q_1, R(\tilde\rho',\gamma^{-1}))\text{.}
\end{equation*}
The calculation of $\tau^{-1} \circ \tau$ is straightforward.
For the calculation of $(\tau \circ \tau^{-1})(q_1,q_2)$ we have to compute in the second component
\begin{eqnarray*}
\beta_r((R(\tilde\rho',\gamma^{-1} \cdot \id_{\alpha(q_1)^{-1}}),1),q_1) &=& \beta_r((R(\tilde\rho',\gamma^{-1}\cdot \id_{\alpha(q_1)^{-1}}),1) \circ (\gamma \cdot \id_{\alpha(\tilde q)^{-1}}),\tilde q)
\\&=& \beta_r((\tilde\rho',\alpha(q_1)\alpha(\tilde q)^{-1}),\tilde q)
\\&=& \beta_0(\alpha(q_1)\alpha(\tilde q)^{-1}g'^{-1},\beta_r((\tilde\rho',g'),\tilde q))
\\&=& \beta_0(\alpha(q_1)\alpha(\tilde q)^{-1}g'^{-1},q_{2})\text{,}
\end{eqnarray*}
and this is equivalent to $q_2$. 
\end{proof}

In order to promote the anafunctor $F$ to a 1-morphism between principal 2-bundles, we have to do two things: we have to check that $F$ commutes with the projections of the bundle $\mathcal{P}_1$ and $\mathcal{P}_2$, and we have to construct a $\Gamma$-equivariant structure on $F$. For the first point we use Remark \ref{rem:preserving} \erf{rem:preserving:b}, whose criterion $\pi \circ \alpha_l = \pi \circ \alpha_r$ is clearly satisfied. For the second point we provide a smooth action $\rho: F \times \Gamma_1 \to F$ in the sense of Definition \ref{def:actionanafunctor} and  use Lemma \ref{lem:equivarianceidentification}, which provides a construction of a $\Gamma$-equivariant structure. The action is defined by
\begin{equation}
\label{gammaactionanafunctorconstruction}
\rho(q,\gamma) := \beta_l^{-1}(q \circ (\id_{\alpha(q)} \cdot \gamma \cdot \id_{t(\gamma)^{-1}}),(\id_{R(\alpha_l(q),t(\gamma))},t(\gamma))) \text{.}
\end{equation}

\begin{lemma}
This defines a smooth action of $\Gamma_1$ on $F$ in the sense of Definition \ref{def:actionanafunctor}.
\end{lemma}

\begin{proof}
Smoothness is clear from the definition. The identity
\begin{equation*}
\rho(\rho(q,\gamma_1),\gamma_2) = \beta_l^{-1}(q \circ (\id_{\alpha(q)} \cdot  \gamma_1 \cdot \gamma_2 \cdot \id_{t(\gamma_2)^{-1}t(\gamma_1)^{-1}}),(\id,t(\gamma_{1}\cdot\gamma_{2}))) = \rho(q,\gamma_1 \cdot \gamma_2)
\end{equation*}
follows from the definition and the two identities 
\begin{multline}
\label{eq:lem729id1}
\alpha(\rho(q,\gamma)) = \alpha(q)s(\gamma)
\quand
(\gamma_{1} \cdot \id_{t(\gamma_{1})^{-1}}) \cdot ( \id_{s(\gamma_1)} \cdot \gamma_{2} \cdot \id_{t(\gamma_{2})^{-1}t(\gamma_1)^{-1}}) =\gamma_1 \cdot \gamma_2 \cdot \id_{t(\gamma_2)^{-1}t(\gamma_1)^{-1}} \text{.}
\end{multline}
The latter can easily be verified upon substituting a crossed module for $\Gamma$. 
Checking  condition (i) of Definition \ref{def:actionanafunctor} just uses the definitions.
We check condition (ii) in two steps. First we prove the identity
\begin{equation*}
\rho(\rho_l(\rho,q),\gamma_l \circ \gamma) = \rho_l(R(\rho,\gamma_l),\rho(q,\gamma))\text{.}
\end{equation*}
The main ingredient is the decomposition
\begin{equation}
\label{eq:lem729id2}
\id_{\alpha(q)} \cdot (\gamma_l \circ \gamma) \cdot \id_{t(\gamma_l)^{-1}} = (\id_{\alpha(q)} \cdot \gamma \cdot \id_{t(\gamma)^{-1}}) \circ (\id_{\alpha(q)s(\gamma)t(\gamma)^{-1}} \cdot \gamma_l \cdot \id_{t(\gamma_l)^{-1}})
\end{equation}
that can e.g. be verified in the crossed module language. 
Now we compute
\begin{eqnarray*}
\rho(\rho_l(\rho,q),\gamma_l \circ \gamma) &=& \beta_l^{-1}(q \circ (\id_{\alpha(q)} \cdot (\gamma_l \circ \gamma) \cdot \id_{t(\gamma_l)^{-1}}), (R(\rho,t(\gamma_l)),t(\gamma_l)))
\\&\stackrel{\text{\erf{eq:lem729id2}}}{=}&  \beta_l^{-1}(q \circ  (\id_{\alpha(q)} \cdot \gamma \cdot \id_{t(\gamma)^{-1}}) ,(R(\rho,\gamma_l),t(\gamma_l)))
\\&=& \rho_l(R(\rho,\gamma_l),\rho(q,\gamma))\text{.}
\end{eqnarray*}
The second step is to show the identity
\begin{equation*}
\rho(\rho_r(q,\rho'),\gamma \circ \gamma_r) = \rho_r(\rho(q,\gamma),R(\rho',\gamma_r))\text{.}
\end{equation*}
Here we use the decomposition
\begin{equation}
\label{eq:lem729id3}
\id_{\alpha(q)} \cdot (\gamma \circ \gamma_r) \cdot \id_{t(\gamma)^{-1}} = (\id_{\alpha(q)} \cdot \gamma \cdot \id_{t(\gamma)^{-1}}) \circ (\id_{\alpha(q)}\cdot \gamma_r \cdot \id_{t(\gamma)^{-1}})\text{.}
\end{equation}
Then we compute
\begin{eqnarray*}
&&\hspace{-1.5cm}\rho(\rho_r(q,\rho'),\gamma \circ \gamma_r) \\&=& \beta_l^{-1}(\beta_r((R(\rho',\id_{\alpha(q)^{-1}}),1),q \circ (\id_{\alpha(q)} \cdot (\gamma \circ \gamma_r) \cdot \id_{t(\gamma)^{-1}})),(\id,t(\gamma)))
\\&\stackrel{\text{\erf{eq:lem729id3}}}{=}&
\beta_l^{-1}(\beta_r((R(\rho',\gamma_r\cdot \id_{s(\gamma)^{-1}\alpha(q)^{-1}}),1),\\&&\hspace{2cm}\beta_0(\alpha(q)s(\gamma_r)s(\gamma)^{-1}\alpha(q)^{-1},q \circ (\id_{\alpha(q)} \cdot \gamma \cdot \id_{t(\gamma)^{-1}}))),(\id,t(\gamma)))
\\&\stackrel{\text{\erf{eq:lem729id1}}}{=}& \beta_l^{-1}(\beta_r((R(\rho',\gamma_r \cdot \id_{\alpha(\rho(q,\gamma))^{-1}})),q \circ (\id_{\alpha(q)} \cdot \gamma \cdot \id_{t(\gamma)^{-1}})),(\id,t(\gamma)))
\\ &=& \rho_r(\rho(q,\gamma),R(\rho',\gamma_r))\text{,}
\end{eqnarray*}
where we have employed the equivalence relation on $F$ that was generated by the action of $\beta_0$. 
\end{proof}

\subsubsection*{Construction of a 2-isomorphism $\eta_{\mathcal{A}}: \mathcal{A} \Rightarrow \mathscr{E}_M(\mathscr{R}_{\mathcal{A}})$}

We may again assume that the common refinement of $\mathcal{A}$ is the fibre product $\mathcal{P}_0 \times_M \mathcal{P}_0'$; otherwise, the proof of Lemma \ref{lem:canonicalrefinement} provides a 2-isomorphism between $\mathcal{A}$ and one of these.  
Now, $\mathcal{A}$ and $\mathscr{E}_M(\mathscr{R}_{\mathcal{A}})$ have the same common refinement, and $\eta_{\mathcal{A}}$ is given by the map
\begin{equation*}
\eta: Q \to F \times \Gamma_0 : q \mapsto (q,\alpha(q))\text{.}
\end{equation*}
This is obviously smooth and respects the projections to the base: if $\chi(q)=(p,p')$, then 
\begin{equation*}
\chi(q,\alpha(q)) \stackrel{\text{\erf{eq:emorph}}}{=} (\alpha_l(q),R(\alpha_r(q),\alpha(q)^{-1}))=(p,p')\text{.}
\end{equation*}
Further, it respects the $\Gamma$-actions:
\begin{equation*}
\eta(q \circ \gamma) = (q \circ \gamma, s(\gamma)) = \beta_l^{-1}(q \circ \gamma,(\id,1))  \stackrel{\text{\erf{gammaactionanafunctorconstruction}}}{=} (\rho(q,\id_{\alpha(q)^{-1}}\cdot \gamma),s(\gamma))\stackrel{\text{\erf{eq:emorph}}}{=} \eta(q) \circ \gamma\text{,}
\end{equation*}
so that $\eta$ is a bundle morphism. It remains to verify the commutativity of the compatibility diagram \erf{eq:comp2morphsimple}. Let $((\rho',g'),q') \in P' \otimes \zeta_1^{*}Q$, and let $(q,(\rho,g)) \in \zeta_2^{*}Q \otimes P$ be a representative for $\beta((\rho',g'),q')$. In particular, we have $\alpha(q)g=g'\alpha(q')$, since $\beta_r$ is anchor-preserving. Then, we get clockwise
\begin{equation}
\label{eq:clockwiseeta}
(\eta \otimes \id)(\beta((\rho',g'),q')) = ((q,\alpha(q)),(\rho,g))\text{.}
\end{equation}
Counter-clockwise, we have to use the isomorphism  of Lemma \ref{lem:eonemorphisms} that we call $\tilde\beta$ here. Then,
\begin{equation}
\label{eq:counterclockwiseeta}
\tilde\beta((\id \otimes \eta)((\rho',g'),q')) = \tilde\beta((\rho',g'),(q',\alpha(q')))
= ((\tilde q,g'\alpha(q')g^{-1}),(\rho, g))
\end{equation}
where the choices $(\tilde\rho,h)$ we have to make for the definition of $\tilde\beta$ are here $(\rho,g^{-1})$, and $\tilde q$ is defined in  \erf{eq:ftilde}, which gives  here
\begin{equation*}
\tilde q = \beta_l^{-1}(\beta_r((\rho',1),q'),(R(\rho^{-1},\id_{g^{-1}}),g^{-1}))\text{.}
\end{equation*}
Comparing \erf{eq:clockwiseeta} and \erf{eq:counterclockwiseeta} it remains to prove $q = \tilde q$ in $F$. As $F$ was the quotient of $Q$ by the action $\beta_0$, it suffices to have
\begin{eqnarray*}
\beta_0(g',\tilde q) &\stackrel{\text{(i)}}{=}& \beta_l^{-1}(\beta_r((\id,g'),\beta_r((\rho',1),q')), (R(\rho^{-1},\id_{g^{-1}}),g^{-1}))
\\ &\stackrel{\text{(iii)}}{=}& \beta_l^{-1}(\beta_r((\rho',g'),q'), (R(\rho^{-1},\id_{g^{-1}}),g^{-1}))
\\ &=& \beta_l^{-1}(\beta_l^{-1}(q,(\rho,g)), (R(\rho^{-1},\id_{g^{-1}}),g^{-1}))
\\ &\stackrel{\text{(ii)}}{=}& \beta_l^{-1}(q,(\id,1))
\\ &=& q\text{.}
\end{eqnarray*}
This finishes the construction of the 2-isomorphism $\eta_{\mathcal{A}}$.

\begin{appendix}

\setsecnumdepth{1}

\section{Equivariant Anafunctors and Group Actions}

\label{sec:equivariantanafunctorsandgroupactions}

In this section we are concerned with a Lie 2-group $\Gamma$ and Lie groupoids   $\mathcal{X}$ and $\mathcal{Y}$  with actions $R_1: \mathcal{X} \times \Gamma \to \mathcal{X}$ and $R_2: \mathcal{Y} \times \Gamma \to \mathcal{Y}$.

\begin{definition}
\label{def:actionanafunctor}
An \emph{action} of the 2-group $\Gamma$ on an anafunctor $F:\mathcal{X} \to \mathcal{Y}$ is an ordinary smooth action
$\rho: F \times \Gamma_1 \to F$
of the group $\Gamma_1$ on the total space $F$ that
\begin{enumerate}[(i)]

\item
preserves the anchors in the sense that the diagrams
\begin{equation*}
\alxydim{@R=1.3cm}{F \times \Gamma_1 \ar[d]_{\alpha_l \times t} \ar[r]^-{\rho} & F \ar[d]^{\alpha_l} \\ \mathcal{X}_0 \times \Gamma_0 \ar[r]_-{R_1} & \mathcal{X}_0}
\quand
\alxydim{@R=1.3cm}{F \times \Gamma_1 \ar[r]^-{\rho} \ar[d]_{\alpha_r \times s} & F \ar[d]^{\alpha_r} \\ \mathcal{Y}_0 \times \Gamma_0 \ar[r]_-{R_2} & \mathcal{Y}_0}
\end{equation*}
are commutative.

\item
is compatible with the $\Gamma$-actions in the sense that the identity
\begin{equation*}
\rho(\chi \circ f \circ \eta,\gamma_l \circ \gamma \circ \gamma_r) =R_1(\chi,\gamma_l) \circ  \rho(f,\gamma) \circ R_2(\eta,\gamma_r)
\end{equation*}
holds for all appropriately composable $\chi\in \mathcal{X}_1$, $\eta\in \mathcal{Y}_1$, $f\in F$, and $\gamma_l,\gamma,\gamma_r\in \Gamma_1$.
\end{enumerate} 
If $F_1,F_2:\mathcal{X} \to \mathcal{Y}$ are anafunctors with $\Gamma$-action, a transformation $\eta: F_1 \Rightarrow F_2$ is called $\Gamma$-equivariant if the map $\eta:F_1 \to F_2$ between total spaces is $\Gamma_1$-equivariant in the ordinary sense. 
\end{definition}

Anafunctors $\mathcal{X} \to \mathcal{Y}$ with $\Gamma$-actions together with $\Gamma$-equivariant transformations form a groupoid $\af_{\Gamma}(\mathcal{X},\mathcal{Y})$. On the other hand, there is another groupoid $\Gamma\text{-}\af(\mathcal{X},\mathcal{Y})$ consisting of  $\Gamma$-equivariant anafunctors (Definition \ref{def:equivanafunctor}) and $\Gamma$-equivariant transformations (Definition \ref{def:equivtransformation}). 
\begin{lemma}
\label{lem:equivarianceidentification}
The categories $\af_{\Gamma}(\mathcal{X},\mathcal{Y})$ and $\Gamma\text{-}\af(\mathcal{X},\mathcal{Y})$ are canonically isomorphic.
\end{lemma}

\begin{proof}
We construct a functor
\begin{equation}
\mathcal{E}: \af_{\Gamma}(\mathcal{X},\mathcal{Y}) \to \Gamma\text{-}\af(\mathcal{X},\mathcal{Y})\text{.}
\end{equation}
Let $F:\mathcal{X} \to \mathcal{Y}$ be an anafunctor with $\Gamma$-action $\rho$. We shall define a transformation
\begin{equation*}
\lambda_{\rho}: F \circ R_1 \Rightarrow R_2 \circ (F \times \id)\text{.}
\end{equation*}
First of all, the composite
\begin{equation*}
\alxydim{}{\mathcal{X} \times \Gamma \ar[r]^-{R_1} & \mathcal{X} \ar[r]^{F} & \mathcal{Y}}
\end{equation*}
is given by the total space $(\mathcal{X}_0 \times \Gamma_0) \lli{R_1} \times_{\alpha_l} F$, left and right anchors send an element $(x,g,f)$ to $(x,g)$ and $\alpha_r(f)$, respectively, and the actions are
\begin{equation*}
(\chi,\gamma) \circ (x,g,f) = (t(\chi),t(\gamma),R_1(\chi,\gamma) \circ f)
\quand
(x,g,f) \circ \eta = (x,g,f\circ \eta)\text{.}
\end{equation*}
On the other hand, the composite
\begin{equation*}
\alxydim{}{\mathcal{X} \times \Gamma \ar[r]^{F \times \id} & \mathcal{Y} \times \Gamma \ar[r]^-{R_2} & \mathcal{Y}} 
\end{equation*}
is given by the total space $((F \times \Gamma_1) \lli{R_2 \circ (\alpha_r \times s)} \times_{t} \mathcal{Y}_1) / \sim$ with the equivalence relation
\begin{equation*}
(f \circ \eta',\gamma \circ \gamma',\eta) \sim (f,\gamma,R_2(\eta',\gamma') \circ \eta)\text{.}
\end{equation*}
The left and right anchors send an element $(f,\gamma,\eta)$ to $(\alpha_l(f),t(\gamma))$ and $s(\eta)$, respectively, and the actions are
\begin{equation*}
(\chi,\gamma') \circ (f,\gamma,\eta) = (\chi \circ f,\gamma'\circ \gamma,\eta)
\quand
(f,\gamma,\eta) \circ \eta' = (f,\gamma,\eta \circ \eta')\text{.}
\end{equation*}
The inverse of the following map will define the transformation $\lambda$:
\begin{equation*}
 (F \times \Gamma_1) \lli{R_2 \circ (\alpha_r \times s)} \times_{t} \mathcal{Y}_1 \to (\mathcal{X}_0 \times \Gamma_0) \lli{R_1} \times_{\alpha_l} F: (f,\gamma,\eta) \mapsto (\alpha_l(f),t(\gamma),\rho(f,\gamma)\circ \eta)\text{.} \end{equation*}
Condition (i) ensures that this map ends in the correct fibre product, and condition (ii) assures that it is well-defined under the equivalence relation $\sim$. The left anchors are automatically respected, and the right anchors require condition (i). Similarly, the left action is respected automatically, and the right actions due to condition (ii). The axiom for a transformation is satisfied because $\rho$ is a group action. This defines the functor $\mathcal{E}$ on objects. On morphisms, it is straightforward to check that the conditions on both hand sides coincide; in particular, $\mathcal{E}$ is full and faithful.

In order to prove that the functor $\mathcal{E}$ is an isomorphism, we start with a given $\Gamma$-equivariant structure $\lambda$ on the anafunctor $F$. Then, an action $\rho: F \times \Gamma_1 \to F$ is defined by
\begin{equation*}
(f,\gamma) \mapsto \mathrm{pr}_3(\lambda^{-1}(f,\gamma,\id_{R_2(\alpha_r(f),s(\gamma))})) \end{equation*}
with $\mathrm{pr}_3: (\mathcal{X}_0 \times \Gamma_0) \lli{R_1} \times_{\alpha_l} F \to F$ the projection. The axiom for an action is satisfied due to the identity $\lambda$ obeys. 
It is straightforward to verify conditions (i) and (ii) of Definition \ref{def:actionanafunctor}.
To close the proof it suffices to notice that the two procedures we have defined are (strictly) inverse to each other.
\end{proof}

We are also concerned with  the composition of anafunctors with $\Gamma$-action. Suppose that $\mathcal{Z}$ is a third Lie groupoid with a $\Gamma$-action $R_3$, and $F:\mathcal{X} \to \mathcal{Y}$ and $G:\mathcal{Y} \to \mathcal{Z}$ are anafunctors with $\Gamma$-actions  $\rho:F \times \Gamma_1 \to F$ and $\tau: G \times \Gamma_1 \to G$. Then, the composition $G \circ F$ is equipped with the $\Gamma$-action defined by
\begin{equation}
\label{eq:companafunctorsaction}
(F \times_{\mathcal{Y}_0} G) \times  \Gamma_1 \to (F \times_{\mathcal{Y}_0} G): ((f,g),\gamma) \mapsto (\rho(f,\gamma),\tau(g,\id_{s(\gamma)}))\text{.}
\end{equation}
We leave it to the reader to check

\begin{lemma}
\label{lem:actioncomposition}
\label{lemactionequivariantidentity}
Let $\mathcal{X}$, $\mathcal{Y}$ and $\mathcal{Z}$ be Lie groupoids with $\Gamma$-actions $R_1$, $R_2$ and $R_3$.
\begin{enumerate}[(a)]
\item 
Let $F: \mathcal{X} \to \mathcal{Y}$ and $G: \mathcal{Y} \to \mathcal{Z}$ be $\Gamma$-equivariant anafunctors. If $\Gamma$-equivariant structures on $F$ and $G$ correspond to $\Gamma_1$-actions under the isomorphism of Lemma \ref{lem:equivarianceidentification}, then the $\Gamma$-equivariant structure on the composite $F \circ G$ corresponds to the $\Gamma_1$-action defined above. 

\item
The isomorphism of Lemma \ref{lem:equivarianceidentification} identifies the trivial $\Gamma$-equivariant structure on the identity anafunctor $\id: \mathcal{X} \to \mathcal{X}$ with the $\Gamma_1$-action $R_1: \mathcal{X}_1 \times \Gamma_1 \to \mathcal{X}_1$ on its total space $\mathcal{X}$.
\end{enumerate}
\end{lemma}

\setsecnumdepth{1}

\section{Constructing Equivalences between 2-Stacks}

\label{app:bicategories}

Let $\mathcal{C}$ be a bicategory  (we assume that associators and unifiers are  \emph{invertible} 2-morphisms). We fix the following terminology: a 1-\emph{iso}morphism $f: X_1 \to X_2$ in $\mathcal{C}$ always includes the data of an inverse 1-morphism $\bar f\maps X_2 \to X_1$ and of 2-isomorphisms $i: \bar f \circ f \Rightarrow \id$ and $j: \id \Rightarrow f \circ \bar f$ satisfying the zigzag identities. Let $\mathcal{D}$ be another bicategory. A 2-functor $F:\mathcal{C} \to \mathcal{D}$ is assumed to have \emph{invertible} compositors and unitors. 

The following lemma is  certainly \quot{well-known}, although we have not been able to find a reference for exactly this statement. 

\begin{lemma}
\label{lem:equivalence}
Let $F: \mathcal{C} \to \mathcal{D}$ be a 2-functor  that is fully faithful on Hom-categories. 
 Suppose one has chosen:
\begin{enumerate}

\item 
for every object $Y \in \mathcal{D}$ an object $G_Y \in \mathcal{C}$ and   a 1-isomorphism $\xi_Y: Y \to F(G_Y)$.

\item
for all objects $X_1,X_2\in \mathcal{C}$ and all 1-morphisms $g:F(X_1) \to F(X_2)$, a 1-morphism $G_g: X_1 \to X_2$ in $\mathcal{C}$ together with a 2-isomorphism $\eta_g: g \Rightarrow F(G_g)$. \footnote{More accurately we should write $G_{X_1,X_2,g}$ and $\eta_{X_1,X_2,g}$, but we will suppress $X_1$ and $X_2$ in the notation.} 
\end{enumerate}
Then, there is a  2-functor $G:\mathcal{D} \to \mathcal{C}$ and pseudonatural equivalences
\begin{equation*}
a: \id_{\mathcal{D}} \Rightarrow F \circ G
\quand
b:G \circ F \Rightarrow \id_{\mathcal{C}}\text{.}
\end{equation*}
In particular, $F$ is an equivalence of bicategories.
\end{lemma}

Proof. 
We recall  our convention concerning 1-isomorphisms: the 1-isomorphisms $\xi_Y$ include choices of inverse 1-morphisms $\bar\xi_Y$  together with 2-isomorphisms $i_Y: \bar \xi_Y \circ \xi_Y \Rightarrow \id$ and $j_Y: \id \Rightarrow \xi_Y \circ \bar \xi_Y$ satisfying the zigzag identities. 

First we explicitly construct the 2-functor $G$. 
On objects, we put $G(Y) := G_Y$.  We use the notation $\tilde g := (\xi_{Y_2} \circ g) \circ \bar\xi_{Y_1}$ for all 1-morphisms $g: Y_1 \to Y_2$ in $\mathcal{D}$, and define $G(g) = G_{\tilde g}$.   If $g,g': Y_1 \to Y_2$ are 1-morphisms, and $\psi:g \Rightarrow g'$ is a 2-morphism, we consider the 2-morphism $\tilde \psi$ defined by
\begin{equation*}
\alxydim{@C=1.5cm}{F(G_{\tilde g}) \ar@{=>}[r]^-{\eta_{\tilde g}^{-1}} & (\xi_{Y_2} \circ g) \circ \bar\xi_{Y_1} \ar@{=>}[r]^-{(\id \circ \psi) \circ \id} & (\xi_{Y_2} \circ g') \circ \bar\xi_{Y_1} \ar@{=>}[r]^-{\eta_{\tilde g'}} & F(G_{\tilde g'})\text{.} }
\end{equation*}
Since $F$ is fully faithful on 2-morphisms, we may choose the unique 2-morphism $G(\psi): G(g) \Rightarrow G(g')$ such that $F(G(\psi))=\tilde \psi$. In order to define the compositor of $G$ we look at 1-morphisms $g_{12}:Y_1 \to Y_2$ and $g_{23}: Y_2 \to Y_3$. We consider the 2-morphism
\begin{equation*}
\alxydim{@C=1.2cm@R=0.6cm}{F(G(g_{23}) \circ G(g_{12})) \ar@{=>}[r]^-{c^{-1}_{G(g_{12}),G(g_{23})}} & F(G_{\tilde g_{23}}) \circ F(G_{\tilde g_{12}}) \ar@{=>}[d]^{\eta_{\tilde g_{23}}^{-1} \circ \eta_{\tilde g_{12}}^{-1}} \\ & ((\xi_{Y_3} \circ g_{23}) \circ \bar\xi_{Y_2}) \circ ((\xi_{Y_2} \circ g_{12}) \circ \bar\xi_{Y_1}) \ar@{=>}[d]^{a,i_{Y_2}} \\ & (\xi_{Y_3} \circ (g_{23} \circ g_{12})) \circ \bar\xi_{Y_1} \ar@{=>}[r]_-{\eta_{\widetilde {g_{23} \circ g_{12}}}} &  F(G(g_{23} \circ g_{12}))\text{;}}
\end{equation*}
its unique preimage under the 2-functor $F$ is the compositor 
\begin{equation*}
c_{g_{12},g_{23}}:G(g_{23}) \circ G(g_{12}) \Rightarrow G(g_{23} \circ g_{12})\text{.}
\end{equation*}
In order to define the unitor of $G$ we consider an object $Y \in \mathcal{D}$ and look at the 2-morphism
\begin{equation*}
\alxydim{@C=1.6cm}{F(G(\id_Y))  \ar@{=>}[r]^-{\eta^{-1}_{\widetilde{\id_Y}}} & (\xi_{Y} \circ \id_Y) \circ \bar\xi_{Y} \ar@{=>}[r]^-{l_{\xi_{Y}},j_Y^{-1}} &   \id_{F(G(Y))} \ar@{=>}[r]^-{u_{G(Y)}^{-1}} &  F(\id_{G(Y)})\text{.}}
\end{equation*}
Its unique preimage under the 2-functor $F$ is the unitor
$u_Y: G(\id_Y) \Rightarrow \id_{G(Y)}$.
The second step is to verify the axioms of a 2-functor. This is simple but extremely tedious and can only be left as an exercise. 
The third step is to  construct the pseudonatural transformation \begin{equation*}
a: \id_{\mathcal{D}} \Rightarrow F \circ G\text{.}
\end{equation*}
Its component at an object $Y$ in $\mathcal{D}$ is the 1-morphism $a(Y) := \xi_Y: Y \to F(G(Y))$. Its component at a 1-morphism $g:Y_1 \to Y_2$ is the 2-morphism $a(g)$ defined by
\begin{equation*}
\alxydim{@R=0.6cm}{a(Y_2) \circ g \ar@{=}[r] & \xi_{Y_2} \circ g \ar@{=>}[d]^{\id \circ l^{-1}_{\xi_{Y_2} \circ g}} \\ & (\xi_{Y_2} \circ g) \circ \id  \ar@{=>}[d]^{a,i^{-1}_{Y_2}}  \\ & ((\xi_{Y_2} \circ g) \circ \bar\xi_{Y_1}) \circ \xi_{Y_1} \ar@{=>}[d]^{\eta_{\tilde g} \circ \id} \\ & F(G_{\tilde g}) \circ \xi_{Y_1} \ar@{=}[r] & F(G(g)) \circ a(Y_1)\text{.}}
\end{equation*}
There are two axioms a pseudonatural transformation has to satisfy, and their proofs are again left as an exercise.
It is easy to see that $a$ is a pseudonatural \emph{equivalence}, with an inverse transformation given by $\bar a(Y) := \bar\xi_Y$.
The fourth and last step is to construct the pseudonatural transformation 
\begin{equation*}
b: G \circ F \Rightarrow \id_{\mathcal{C}}\text{.}
\end{equation*}
Its component at an object $X$ is $b(X) := G_{\bar\xi_{F(X)}}: G(F(X)) \to X$. Its component at a 1-morphism $f\maps X_2 \to X_2$ is the  2-morphism 
\begin{equation*}
b(f): b(X_2) \circ G(F(f)) \Rightarrow f \circ b(X_1)
\end{equation*}
given as the unique preimage under $F$ of the 2-morphism
\begin{equation*}
\alxydim{@R=0.6cm}{F(b(X_2) \circ G(F(f))) \ar@{=>}[r]^-{c^{-1}} & F(b(X_2)) \circ F(G(F(f))) \ar@{=>}[d]_{\eta^{-1}_{\bar\xi_{F(X_2)}} \circ \eta^{-1}_{F(f)}} \\ & \bar\xi_{F(X_2)} \circ ((\xi_{F(X_2)} \circ F(f)) \circ \xi_{F(X_1)}) \ar@{=>}[d]_{a,i_{F(X_2)},r}  \\ & F(f) \circ \bar\xi_{F(X_1)} \ar@{=>}[d]_{\id_{F(f)} \circ \eta_{\bar\xi_{F(X_1)}}} \\ & F(f) \circ F(b(X_1)) \ar@{=>}[r]_-{c} & F(f \circ b(X_1))\text{.}}
\end{equation*}
The proofs of the axioms are again left for the reader, and again it is easy to see that $b$ is a pseudonatural \emph{equivalence} with an inverse transformation given by $\bar b(X) := G_{\xi_{F(X)}}$.
\strut\hfill$\square$

As a consequence of Lemma \ref{lem:equivalence} we obtain the certainly well-known result:

\begin{corollary}
Let $F: \mathcal{C} \to \mathcal{D}$ be essentially surjective, and an equivalence on all Hom-categories. Then, $F$ is an equivalence of bicategories.  
\end{corollary}

Since we work with 2-stacks over manifolds, we need the following \quot{stacky} extension of Lemma \ref{lem:equivalence}.
For a pre-2-stack $\mathcal{C}$, we denote by $\mathcal{C}_M$ the 2-category it associates to a smooth manifold $M$, and by $\psi^{*}: \mathcal{C}_N \to \mathcal{C}_M$ the 2-functor it associates to a smooth map $\psi: M \to N$. The pseudonatural equivalences $\psi^{*} \circ \varphi^{*} \cong (\varphi\circ \psi)^{*}$ will be suppressed from the notation in the following.
If $\mathcal{C}$ and $\mathcal{D}$ are pre-2-stacks, a \emph{ 1-morphism} $F: \mathcal{C} \to \mathcal{D}$ associates 2-functors $F_M: \mathcal{C}_M \to \mathcal{D}_M$ to a smooth manifold $M$,  pseudonatural equivalences
\begin{equation*}
F_{\psi}: \psi^{*} \circ F_N \to F_M \circ \psi^{*}
\end{equation*}
to smooth maps $\psi:M \to N$,
and certain modifications $F_{\psi,\varphi}$ that control the relation between $F_{\psi}$ and $F_{\varphi}$ for composable maps $\psi$ and $\varphi$. 

\begin{lemma}
\label{lem:stackequivalence}
Suppose $\mathcal{C}$ and $\mathcal{D}$ are pre-2-stacks over smooth manifolds, and $F: \mathcal{C} \to \mathcal{D}$ is a  1-morphism. Suppose
that for every smooth manifold $M$
\begin{enumerate}
\item 
the assumptions of Lemma \ref{lem:equivalence} for the 2-functor $F_{M}$ are satisfied and 

\item
the  data $(G_Y,\xi_Y)$ and $(G_g,\eta_g)$ is chosen for all objects $Y$ and 1-morphisms $g$  in $\mathcal{D}_M$. 
\end{enumerate}
Then, there is a  1-morphism 
$G:\mathcal{D} \to \mathcal{C}$
of pre-2-stacks together with 2-isomorphisms
\begin{equation*}
a:F \circ G \Rightarrow \id_{\mathcal{D}}
\quand
b:G \circ F \Rightarrow \id_{\mathcal{C}}
\end{equation*}
such that
for every smooth manifold $M$ the 2-functor $G_M$ and the pseudonatural transformations $a_M$ and $b_M$ are the ones of Lemma  \ref{lem:equivalence}. In particular, $F$ is an equivalence of  pre-2-stacks. 
\end{lemma}

For the proof one constructs the required pseudonatural equivalences $G_{\psi}$ and the modifications $G_{\psi,\varphi}$ from the given ones, $F_{\psi}$ and $F_{\psi,\varphi}$, respectively, in a similar way as explained in the proof of Lemma \ref{lem:equivalence}. Since these constructions are straightforward to do but would consume many pages,  and the statement of the lemma is not too surprising and certainly  well-known to many people, we   leave these constructions for the interested reader.

\end{appendix}

\kobib{.}


\begin{thebibliography}{CLGX09}
\addcontentsline{toc}{section}{\refname}

\bibitem[ACJ05]{aschieri}
P.~Aschieri, L.~Cantini, and B.~Jurco, \quot{Nonabelian bundle gerbes, their
  differential geometry and gauge theory}.
\newblock {\em Commun. Math. Phys.}, 254:367--400, 2005.
\newblock \kobiburl{http://arxiv.org/abs/hep-th/0312154}
\bibitem[AN09]{Aldrovandi2009}
E.~Aldrovandi and B.~Noohi, \quot{Butterflies I: Morphisms of 2-group stacks}.
\newblock {\em Adv. Math.}, 221(3):687--773, 2009.
\newblock \kobiburl{http://arxiv.org/abs/0808.3627}
\bibitem[Bar04]{bartels}
T.~Bartels, {\em 2-bundles and higher gauge theory}.
\newblock PhD thesis, University of California, Riverside, 2004.
\newblock \kobiburl{http://arxiv.org/abs/math/0410328}
\bibitem[BBK12]{baas2}
N.~A. Baas, M.~B\"okstedt, and T.~A. Kro, \quot{Two-Categorical Bundles and
  Their Classifying Spaces}.
\newblock {\em J. K-Theory}, 10(2):299--369, 2012.
\newblock \kobiburl{http://arxiv.org/abs/math/0612549}
\bibitem[BCSS07]{baez9}
J.~C. Baez, A.~S. Crans, D.~Stevenson, and U.~Schreiber, \quot{From loop groups
  to 2-groups}.
\newblock {\em Homology, Homotopy Appl.}, 9(2):101--135, 2007.
\newblock \kobiburl{http://arxiv.org/abs/math.QA/0504123}
\bibitem[B{\'e}n67]{benabou2}
J.~B{\'e}nabou, \quot{Introduction to bicategories}.
\newblock In {\em Lecture Notes in Math.}, volume~46. Springer, 1967.
\bibitem[B{\'e}n73]{benabou1}
J.~B{\'e}nabou, \quot{Les distributeurs}.
\newblock In {\em Univ. Cath. de Louvain, Seminaires de Math. Pure}, volume~33.
  1973.
\bibitem[Bre90]{breen3}
L.~Breen, \quot{Bitorseurs et cohomologie non ab{\'e}lienne}.
\newblock In {\em Grothendieck Festschrift}, pages 401--476. 1990.
\bibitem[Bre94]{breen2}
L.~Breen, \quot{On the Classification of 2-Gerbes and 2-Stacks}.
\newblock {\em Ast\'erisque}, 225, 1994.
\bibitem[BS07]{baez2}
J.~C. Baez and U.~Schreiber, \quot{Higher Gauge Theory}.
\newblock In A.~Davydov, editor, {\em Categories in Algebra, Geometry and
  Mathematical Physics}, Proc. Contemp. Math. AMS, Providence, Rhode Island,
  2007.
\newblock \kobiburl{http://arxiv.org/abs/math/0511710}
\bibitem[BS09]{baez8}
J.~C. Baez and D.~Stevenson, \quot{The Classifying Space of a Topological
  2-Group}.
\newblock In N.~Baas, editor, {\em Algebraic Topology}, volume~4 of {\em Abel
  Symposia}, pages 1--31. Springer, 2009.
\newblock \kobiburl{http://arxiv.org/abs/0801.3843}
\bibitem[Bun11]{bunke1}
U.~Bunke, \quot{String Structures and Trivialisations of a Pfaffian Line
  Bundle}.
\newblock {\em Commun. Math. Phys.}, 307(3):675--712, 2011.
\newblock \kobiburl{http://arxiv.org/abs/0909.0846}
\bibitem[CLGX09]{laurent1}
M.~S. C.~Laurent-Gengoux and P.~Xu, \quot{Non-abelian differentiable gerbes}.
\newblock {\em Adv. Math.}, 220(5):1357--1427, 2009.
\newblock \kobiburl{http://arxiv.org/abs/math/0511696}
\bibitem[Gir71]{giraud}
J.~Giraud, \quot{Cohomologie non-abélienne}.
\newblock {\em Grundl. Math. Wiss.}, 197, 1971.
\bibitem[GS]{ginot1}
G.~Ginot and M.~Sti\'enon, \quot{$G$-Gerbes, principal 2-group bundles and
  characteristic classes}.
\newblock Preprint.
\newblock \kobiburl{http://arxiv.org/abs/0801.1238}
\bibitem[Hei04]{heinloth}
J.~Heinloth, \quot{Some notes on differentiable stacks}.
\newblock In {\em Math. Inst., Seminars}, volume~5, pages 1--32. 2004.
\newblock
  \kobiburl{http://www.uni-math.gwdg.de/heinloth/seminar/SeminarFHT.pdf}
\bibitem[Joh77]{johnstone1}
P.~T. Johnstone, {\em Topos Theory}.
\newblock London Math. Soc. Monogr. Ser., 1977.
\bibitem[Jur11]{jurco1}
B.~Jurco, \quot{Crossed module bundle gerbes; classification, string group and
  differential geometry}.
\newblock {\em Int. J. Geom. Methods Mod. Phys.}, 8(5):1079--1095, 2011.
\newblock \kobiburl{http://arxiv.org/abs/math/0510078}
\bibitem[Ler]{lerman1}
E.~Lerman, \quot{Orbifolds as stacks?}
\newblock Preprint.
\newblock \kobiburl{http://arxiv.org/abs/0806.4160}
\bibitem[Met]{metzler}
D.~S. Metzler, \quot{Topological and Smooth Stacks}.
\newblock Preprint.
\newblock \kobiburl{http://arxiv.org/abs/math.DG/0306176}
\bibitem[Mil56]{milnor}
J.~Milnor, \quot{Construction of Universal bundles, II}.
\newblock {\em Ann. of Math.}, 63(3):430--436, 1956.
\bibitem[MM03]{moerdijk}
I.~Moerdijk and J.~Mr\v{c}un, {\em Introduction to foliations and {L}ie
  groupoids}, volume~91 of {\em Cambridge Studies in Adv. Math.}
\newblock Cambridge Univ. Press, 2003.
\bibitem[MRS]{murray5}
M.~K. Murray, D.~M. Roberts, and D.~S. Stevenson, \quot{On the existence of
  bibundles}.
\newblock {\em Proc. Lond. Math. Soc.}, to appear.
\newblock \kobiburl{http://arxiv.org/abs/1102.4388}
\bibitem[MS00]{murray2}
M.~K. Murray and D.~Stevenson, \quot{Bundle gerbes: stable isomorphism and
  local theory}.
\newblock {\em J. Lond. Math. Soc.}, 62:925--937, 2000.
\newblock \kobiburl{http://arxiv.org/abs/math/9908135}
\bibitem[Mur96]{murray}
M.~K. Murray, \quot{Bundle gerbes}.
\newblock {\em J. Lond. Math. Soc.}, 54:403--416, 1996.
\newblock \kobiburl{http://arxiv.org/abs/dg-ga/9407015}
\bibitem[MW09]{wockel2}
C.~Müller and C.~Wockel, \quot{Equivalences of Smooth and Continuous Principal
  Bundles with Infinite-Dimensional Structure Group}.
\newblock {\em Adv. Geom.}, 9(4):605--626, 2009.
\newblock \kobiburl{http://arxiv.org/abs/math/0604142}
\bibitem[NS11]{nikolaus2}
T.~Nikolaus and C.~Schweigert, \quot{Equivariance in higher geometry}.
\newblock {\em Adv. Math.}, 226(4):3367--3408, 2011.
\newblock \kobiburl{http://arxiv.org/abs/1004.4558}
\bibitem[OV91]{onishchik1}
A.~L. Onishchik and E.~B. Vinberg, \quot{I. Foundations of Lie Theory}.
\newblock In A.~L. Onishchik, editor, {\em Lie Groups and Lie Algebras I},
  volume~20 of {\em Encyclopaedia of Mathematical Sciences}. Springer, 1991.
\bibitem[Pro96]{pronk}
D.~A. Pronk, \quot{Entendues and stacks as bicategories of fractions}.
\newblock {\em Compos. Math.}, 102:243--303, 1996.
\bibitem[Seg68]{segal3}
G.~Segal, \quot{Classifying spaces and spectral sequences}.
\newblock {\em Publ. Math. Inst. Hautes \'Etudes Sci.}, 34:105--112, 1968.
\bibitem[SP11]{pries2}
C.~Schommer-Pries, \quot{Central extensions of smooth 2-groups and a
  finite-dimensional string 2-group}.
\newblock {\em Geom. Topol.}, 15:609--676, 2011.
\newblock \kobiburl{http://arxiv.org/abs/0911.2483}
\bibitem[SSW07]{schreiber1}
U.~Schreiber, C.~Schweigert, and K.~Waldorf, \quot{Unoriented {WZW} models and
  holonomy of bundle gerbes}.
\newblock {\em Commun. Math. Phys.}, 274(1):31--64, 2007.
\newblock \kobiburl{http://arxiv.org/abs/hep-th/0512283}
\bibitem[Ste00]{stevenson1}
D.~Stevenson, {\em The geometry of bundle gerbes}.
\newblock PhD thesis, University of Adelaide, 2000.
\newblock \kobiburl{http://arxiv.org/abs/math.DG/0004117}
\bibitem[SW]{schreiber2}
U.~Schreiber and K.~Waldorf, \quot{Connections on non-abelian gerbes and their
  holonomy}.
\newblock Preprint.
\newblock \kobiburl{http://arxiv.org/abs/0808.1923}
\bibitem[Wal07]{waldorf1}
K.~Waldorf, \quot{More morphisms between bundle gerbes}.
\newblock {\em Theory Appl. Categ.}, 18(9):240--273, 2007.
\newblock \kobiburl{http://arxiv.org/abs/math.CT/0702652}
\bibitem[Woc11]{wockel1}
C.~Wockel, \quot{Principal 2-bundles and their gauge 2-groups}.
\newblock {\em Forum Math.}, 23:565--610, 2011.
\newblock \kobiburl{http://arxiv.org/abs/0803.3692}
\end{thebibliography}
\end{document}